\documentclass[11pt,leqno]{amsart}
\usepackage{amsmath,amssymb,amsthm,booktabs,geometry,enumitem,array,tabularx,graphicx}
\usepackage{tikz}
\usetikzlibrary{arrows.meta,positioning,calc}
\usepackage[colorlinks=true,linkcolor=blue,citecolor=blue,urlcolor=blue]{hyperref}
\setlist[enumerate]{itemsep=5pt,topsep=6pt}
\setlist[itemize]{itemsep=4pt,topsep=6pt}
\numberwithin{equation}{section}
\newtheoremstyle{roomy}{11pt}{11pt}{\itshape}{}{\bfseries}{.}{.5em}{}
\newtheoremstyle{roomydef}{11pt}{11pt}{}{}{\bfseries}{.}{.5em}{}
\newtheoremstyle{roomyrem}{9pt}{9pt}{}{}{\itshape}{.}{.5em}{}
\theoremstyle{roomy}
\newtheorem{theorem}{Theorem}
\newtheorem{corollary}[theorem]{Corollary}
\newtheorem{lemma}{Lemma}[section]
\newtheorem{proposition}[lemma]{Proposition}
\newtheorem{cor}[lemma]{Corollary}
\theoremstyle{roomydef}
\newtheorem{fact}[lemma]{Fact}

\theoremstyle{roomyrem}
\newtheorem{remark}[lemma]{Remark}
\newcommand{\kron}[2]{\big(\tfrac{#1}{#2}\big)}
\newcommand{\recall}[1]{\par\smallskip\noindent\emph{Recall.} #1\par\smallskip}

\newcommand{\upshot}[1]{\par\smallskip\noindent\textbf{Upshot.} #1\par\smallskip}
\newcommand{\Z}{\mathbb{Z}}
\newcommand{\R}{\mathbb{R}}
\newcommand{\C}{\mathbb{C}}
\newcommand{\N}{\mathbb{N}}
\newcommand{\Q}{\mathbb{Q}}
\newcommand{\Hh}{\mathbb{H}}
\newcommand{\dd}{\,\mathrm{d}}
\newcommand{\Mm}{\mathcal{M}}
\newcommand{\Pp}{\mathcal{P}}
\newcommand{\Qf}{\mathcal{Q}}
\newcommand{\Rr}{\mathcal{R}}
\newcommand{\Ee}{\mathcal{E}}
\newcommand{\Tr}{\mathrm{Tr}}
\newcommand{\h}{\mathfrak{h}}
\newcommand{\PSL}{\mathrm{PSL}_2(\Z)}
\newcommand{\Gt}{\Gamma_\vartheta}
\newcommand{\vol}{\mathrm{vol}}
\newcommand{\dist}{\mathrm{dist}}
\newcommand{\vpo}{\varpi_{\mathrm{odd}}}

\title[A spectral approach]{Solutions of $x_1^2+x_2^2-x_3^2=D$ with small $x_3$:\\ a spectral approach}
\author{Stephan Baier}
\address{Stephan Baier, Ramakrishna Mission Vivekananda Educational and Research Institute, Department of Mathematics, G.~T.~Road, PO Belur Math, Howrah, West Bengal 711202, India}
\email{stephanbaier2017@gmail.com}
\date{November 2026}

\begin{document}
\begin{abstract}
We count the integral solutions of $x_1^2+x_2^2-x_3^2=D$ with $|x_3|$ of size at most $M$, for $M$ smaller than $\sqrt D$, using a smooth weight $\Phi(x_3/M)$. Such solutions correspond to binary quadratic forms of discriminant $D$, and hence to geodesics in the hyperbolic plane. Under this correspondence, $x_3$ is small exactly when the geodesic passes close to the point $i$. We count these geodesics by the spectral theory of automorphic forms. Throughout, $\theta=7/64$ is the Kim--Sarnak exponent.

For non-square $D=f^2D_0$, with $D_0$ a fundamental discriminant, the geodesics are closed. We obtain an asymptotic formula with error $O(D^{\varepsilon}f^{1+\theta}D_0^{2/3}M^{-1/2})$. The main inputs are the formulas of Duke--Imamo\={g}lu--T\'oth and Baruch--Mao for cycle integrals, the cubic moment bounds of Young and of Andersen--Wu, and a Hecke relation which reduces everything to the fundamental discriminant $D_0$. For $D\equiv5\bmod 8$ this counts all solutions, and the range $M\ge D^{1/2-1/1332}$ of Friedlander and Iwaniec improves to $M\ge D^{4/9+\varepsilon}$.

For square $D=n^2$ the geodesics run from cusp to cusp, and Hecke theory reduces them to the single geodesic from $0$ to $\infty$. For every $n$ and both parities of $x_3$ we obtain an asymptotic formula with error $O(n^{1+\theta+\varepsilon}M^{-1/2})$, which is non-trivial for $M\ge n^{71/96+\varepsilon}$. This improves the range $n^{5/6+\varepsilon}$ of the elementary method and confirms the corrected main term of the author's earlier work.

This paper was written in collaboration with Claude (Opus~5.5, High; Fable~5.2 in the final stage), an AI model developed by Anthropic; see Subsection~\ref{sec:contrib} for the respective contributions.
\end{abstract}
\maketitle

\tableofcontents

\section{Introduction}

\subsection{The problem}\label{sec:problem}
Let
\[
Q(x)=x_1^2+x_2^2-x_3^2 .
\]
For an integer $D\ge1$ the equation $Q(x)=D$ describes a one-sheeted hyperboloid, and it has infinitely many integral points. We ask for those integral points whose third coordinate $x_3$ is small compared with $\sqrt D$.

To make this precise, we fix a smooth even function $\Phi:\R\to\C$ with support in $[-2,2]$, which serves as a smooth cut-off. For a parameter $M$ with $1\le M\le\sqrt D$ and for $e\in\{0,1\}$ we put
\[
\Rr_e(D)=\Rr_e(D;M)=\sum_{\substack{x\in\Z^3,\ Q(x)=D\\ x_3\equiv e\bmod 2}}\Phi\Big(\frac{x_3}{M}\Big),\qquad \Rr(D)=\Rr_0(D)+\Rr_1(D).
\]
Thus $\Rr(D)$ is a smoothly weighted count of the solutions with $|x_3|\le2M$. Since $D+x_3^2$ is a sum of two squares $x_1^2+x_2^2$ in a bounded number of ways on average, one expects $\Rr(D)$ to be roughly proportional to $\int_\R\Phi(x/M)\dd x$, i.e.\ to $M$, with an arithmetic factor depending on $D$. The question is how small $M$ can be taken, as a power of $D$, while this asymptotic behaviour persists. In particular, an asymptotic formula for $M=D^{\eta}$ with $\eta<1/2$ shows that $Q(x)=D$ has solutions with $|x_3|\le 2D^{\eta}$.

The solutions with $x_2\equiv x_3\bmod 2$ play a special role, because they correspond to binary quadratic forms (Section~\ref{sec:count}). We denote their weighted count by
\[
N(D)=N(D;M)=\sum_{\substack{x\in\Z^3,\ Q(x)=D\\ x_2\equiv x_3\bmod 2}}\Phi\Big(\frac{x_3}{M}\Big).
\]
The remaining solutions are obtained from these by interchanging $x_1$ and $x_2$, and a short parity argument expresses $\Rr_e(D)$ in terms of $N(D)$ and of one further count. For instance, if $D\equiv5\bmod 8$, then all solutions have $x_3$ even and $\Rr(D)=2N(D)$. All cases are listed in Proposition~\ref{prop:par}.

\subsection{Previous work and related problems}\label{sec:previous}
Small solutions of $Q(x)=D$ were studied by Friedlander and Iwaniec \cite{FrIw}, in their paper on small representations by indefinite ternary quadratic forms. They treated square-free $D\equiv5\bmod 8$ and obtained an asymptotic formula for $\Rr(D)$ in the range $M\ge D^{1/2-1/1332}$. Their method rests on the equidistribution of the roots of quadratic congruences, in the strong form given by the spectral theory of Weyl sums for quadratic roots \cite{DFI}.

For square discriminants $D=n^2$ the author \cite{B24} gave an elementary treatment. In the revised version \cite{B26} an oversight in the main term of \cite{B24} is corrected, and the range of validity is improved to $M\ge n^{5/6+\varepsilon}$.

\medskip
\emph{Linnik's problems and Duke's theorem.} Our question belongs to a family of three classical problems about integral points on quadrics.
\begin{itemize}
\item \emph{The sphere} $x_1^2+x_2^2+x_3^2=n$.
\item \emph{The two-sheeted hyperboloid} $B^2-4AC=-n<0$. Its integral points are the definite binary quadratic forms $[A,B,C]$ of discriminant $-n$. A positive definite form $q$ has a unique root $z_q$ of $q(z,1)=0$ in the upper half-plane $\Hh$, and the classes of forms give finitely many points on the modular surface $\PSL\backslash\Hh$, the \emph{Heegner points} (or CM points) of discriminant $-n$.
\item \emph{The one-sheeted hyperboloid} $B^2-4AC=D>0$. This is our equation $Q(x)=D$ in the coordinates $x=(B,A-C,A+C)$. Its integral points are the indefinite binary forms of discriminant $D$. Such a form has two real roots, and these are joined by a geodesic of $\Hh$. For non-square $D$ the classes of forms give finitely many \emph{closed geodesics} on the modular surface.
\end{itemize}
Linnik \cite{Li} (and Skubenko, for closed geodesics) proved, by Linnik's ergodic method and under a congruence condition on $n$ (resp.\ $D$), that in all three cases the points become equidistributed: the points on the sphere after projection to the unit sphere, the Heegner points and the closed geodesics on the modular surface. Duke \cite{Du} removed the congruence condition for Heegner points and for closed geodesics of fundamental discriminants, and obtained a power saving; his proof uses Maass forms of half-integral weight and Iwaniec's bound for their Fourier coefficients. For the sphere the corresponding result is due to Duke and Schulze-Pillot \cite{DSP}.

\medskip
\emph{Small scales.} Equidistribution is a statement about a fixed region. To ask for the integral points with \emph{one small coordinate} is to ask for equidistribution in a region which shrinks as $n$ (or $D$) grows. On the sphere, $|x_3|\le M$ describes a thin band around the equator; this is Linnik's problem on sums of two squares and a ``microsquare'', treated by Humphries and Radziwiłł \cite{HR}. In the two hyperbolic cases the third coordinate $x_3=A+C$ has a clean geometric meaning (Lemma~\ref{dist} and Remark~\ref{rem:definite}):
\[
A+C=\begin{cases}\sqrt D\,\sinh\dist(i,\ell_q),& \text{$q$ indefinite of discriminant $D>0$, with geodesic $\ell_q$},\\[3pt]
\sqrt n\,\cosh d(i,z_q),& \text{$q$ positive definite of discriminant $-n$, with Heegner point $z_q$}.\end{cases}
\]
Here $\dist(i,\ell_q)$ is the signed hyperbolic distance from the point $i$ to the geodesic, and $d(i,z_q)$ is the hyperbolic distance between two points. So in the indefinite case, which is the subject of this paper, $x_3$ is small exactly when the geodesic passes close to $i$; in the definite case $x_3\ge\sqrt n$ always, and $x_3$ is close to $\sqrt n$ exactly when the Heegner point is close to $i$. Both are questions about a small ball around $i$: one counts geodesics passing through it, the other counts points lying in it.

For Heegner points in shrinking sets, see Young \cite{Yo}. For closed geodesics of odd fundamental discriminant $D$ and balls around a fixed point, Humphries \cite[Thm.~1.24]{Hu} proved equidistribution for radii $D^{-\eta}$ with $\eta<1/18$, using Young's bounds \cite{Yo} for twisted $L$-functions. This corresponds to the exponent $4/9=1/2-1/18$ in Theorem~\ref{thm:ns} below, but with an error term $o(1)$ instead of a power saving. Humphries and Radziwiłł \cite{HR}, and de Faveri \cite{dF}, studied almost all balls, and the variance, at smaller scales.

\medskip
\emph{Growing regions.} Duke, Rudnick and Sarnak \cite{DRS} counted integral points on affine homogeneous varieties, among them the level sets of indefinite ternary quadratic forms, in growing balls. There $D$ is fixed and the size of the region tends to infinity, whereas here $D$ tends to infinity and the region is small.

\medskip
\emph{Theta functions.} Finally, the counting function is a Fourier coefficient of Siegel's indefinite theta function attached to $Q$. This connection was the starting point of this work. It is explained in Section~\ref{sec:theta}, together with the reason why we take a different route.

\subsection{Main results}\label{sec:results}
Let $\theta=7/64$ be the exponent in the Kim--Sarnak bound for Hecke eigenvalues of Maass forms (Fact~\ref{F:hecke}), and put
\[
\Psi_0=\Psi_0(M)=\int_\R\Phi\Big(\frac{x}{M}\Big)\dd x .
\]
We use the standard arithmetic functions: $\sigma(m)=\sum_{d\mid m}d$ is the sum of the divisors of $m$, $\tau(m)$ is the number of divisors, $\mu$ is the Möbius function and $\Lambda$ is the von Mangoldt function; $v_2(n)$ is the exponent of $2$ in $n$. As usual, $\varepsilon$ denotes an arbitrarily small positive number, which may change from line to line.

Every non-square discriminant $D>0$ can be written uniquely as $D=f^2D_0$ with a fundamental discriminant $D_0>1$ (Subsection~\ref{sec:forms}). Let $\chi_{D_0}=\kron{D_0}{\cdot}$ be the Kronecker symbol and
\begin{equation}\label{hdef}
\h_{D_0}(f)=\frac{1}{f}\sum_{d\mid f}\mu(d)\chi_{D_0}(d)\,\sigma\Big(\frac{f}{d}\Big)=\prod_{p^a\parallel f}\Big(\sum_{i=0}^ap^{-i}-\chi_{D_0}(p)\sum_{i=1}^ap^{-i}\Big).
\end{equation}
Each local factor lies in $[1,3]$, so $1\le\h_{D_0}(f)\ll_\varepsilon f^{\varepsilon}$.

\begin{theorem}[Non-square discriminants]\label{thm:ns}
Let $D=f^2D_0$ be a non-square discriminant, and let $1\le M\le\sqrt D$. Then for every $\varepsilon>0$,
\[
N(D)=\frac{6}{\pi}\,L(1,\chi_{D_0})\,\h_{D_0}(f)\,\Psi_0+O_{\Phi,\varepsilon}\big(D^{\varepsilon}f^{1+\theta}D_0^{2/3}M^{-1/2}\big).
\]
\end{theorem}

Since $L(1,\chi_{D_0})\gg_\varepsilon D_0^{-\varepsilon}$ (Siegel) and $\h_{D_0}(f)\ge1$, Theorem~\ref{thm:ns} is an asymptotic formula as soon as $M\ge f^{71/96}D_0^{4/9}D^{\varepsilon}$. The two extreme cases are as follows. For fundamental discriminants ($f=1$) the range is $M\ge D^{4/9+\varepsilon}$. For bounded $D_0$ it is $M\ge D^{71/192+\varepsilon}$, which matches the square case below.

For $D\equiv5\bmod 8$ the count $N(D)$ determines $\Rr(D)$, and we obtain the following.

\begin{corollary}[All solutions]\label{cor:tern}
Let $D=4^kD'$ with $k\ge0$ and $D'\equiv5\bmod 8$, write $D'=f'^2D_0$ and $f=2^kf'$. Then for $1\le M\le\sqrt D$,
\[
\Rr(D)=\frac{6}{\pi}\big(4-2^{1-k}\big)L(1,\chi_{D_0})\,\h_{D_0}(f')\,\Psi_0+O\big(D^{\varepsilon}f^{1+\theta}D_0^{2/3}M^{-1/2}\big).
\]
In particular, for $D\equiv5\bmod8$ all solutions have $x_3$ even, and
\[
\Rr(D)=\frac{12}\pi L(1,\chi_{D_0})\,\h_{D_0}(f)\,\Psi_0+O\big(D^\varepsilon f^{1+\theta}D_0^{2/3}M^{-1/2}\big).
\]
\end{corollary}

For square-free $D\equiv5\bmod8$ this is the setting of \cite{FrIw}. The range improves from $M\ge D^{1/2-1/1332}$ to $M\ge D^{4/9+\varepsilon}$, and the assumption that $D$ is square-free is no longer needed.

For square discriminants we state the result here in its simplest form. A more precise version, with secondary main terms and the error term $O(n^{1+\theta+\varepsilon}M^{-1/2})$ for all $1\le M\le n$, is Theorem~\ref{thm:sq} in Section~\ref{sec:square}.

\begin{theorem}[Square discriminants]\label{thm:sqsimple}
Let $n\ge1$, $e\in\{0,1\}$ and $\varepsilon>0$. If $n^{71/96+\varepsilon}\le M\le n$, then
\[
\Rr_e(n^2)=\frac{4}{\pi}(\log n)\Psi_0+O(M\log\log n).
\]
Under the Ramanujan conjecture for Maass forms of level $1$ and $2$ (i.e.\ with $\theta=0$), the same holds for $M\ge n^{2/3+\varepsilon}$.
\end{theorem}

The main term agrees with the corrected main term of \cite{B26}, which was obtained by a completely different method, and the range $M\ge n^{71/96+\varepsilon}$ improves the range $M\ge n^{5/6+\varepsilon}$ there. Note that $71/96=(2/3)(1+\theta)$.

\subsection{The strategy in a nutshell}\label{sec:strategy}
The proofs follow one line of thought, which we now describe in simple terms. Figure~\ref{fig:flow} summarises it. Precise statements follow in Sections~\ref{sec:count}--\ref{sec:square}.

\begin{figure}[t]
\centering
\begin{tikzpicture}[
  box/.style={draw,rounded corners=2pt,align=center,font=\small,minimum height=1.05cm,text width=6.6cm,fill=white},
  arr/.style={-{Stealth[length=2.2mm]},thick},
  lab/.style={font=\footnotesize\itshape,align=left,text width=4.6cm}]
\node[box] (a) {solutions of $Q(x)=D$\\ with $|x_3|\le 2M$};
\node[box,below=0.55cm of a] (b) {binary forms $[A,B,C]$ of\\ discriminant $D$ with $|A+C|\le2M$};
\node[box,below=0.55cm of b] (c) {geodesics $\ell_q$ passing within\\ hyperbolic distance $\approx 2M/\sqrt D$ of $i$};
\node[box,below=0.55cm of c] (d) {integral of an automorphic function $F$\\ over the geodesics of discriminant $D$};
\node[box,below=0.55cm of d] (e) {spectral expansion of $F$:\\ constant $+$ Maass forms $+$ Eisenstein series};
\node[box,below left=0.7cm and -2.2cm of e,text width=3.4cm] (f) {constant term:\\ \textbf{main term}};
\node[box,below right=0.7cm and -2.2cm of e,text width=5.6cm] (g) {Maass forms, Eisenstein series:\\ periods $=$ $L$-values; moments\\ \textbf{error term}};
\draw[arr] (a)--(b); \draw[arr] (b)--(c); \draw[arr] (c)--(d); \draw[arr] (d)--(e); \draw[arr] (e)--(f); \draw[arr] (e)--(g);
\node[lab,right=0.4cm of a] {Step 0: parity reduction};
\node[lab,right=0.4cm of b] {Step 1: Gauss's correspondence\\ $x=(B,A-C,A+C)$};
\node[lab,right=0.4cm of c] {Step 2: $x_3=A+C=\sqrt D\sinh r$,\\ $r=$ distance from $i$ to $\ell_q$};
\node[lab,right=0.4cm of d] {Step 3: a kernel $K(d(z,w))$ whose integral along a geodesic is the weight};
\node[lab,right=0.4cm of e] {Step 4: pre-trace formula};
\end{tikzpicture}
\caption{The structure of the argument.}\label{fig:flow}
\end{figure}

\emph{Step 0: parity.} As explained in Subsection~\ref{sec:problem}, it suffices to count the solutions with $x_2\equiv x_3\bmod2$. For $D\equiv1\bmod 8$ one also needs the solutions with $x_2,x_3$ both even, which leads to the theta group $\Gt$ in the square case.

\emph{Step 1: binary forms.} The map $[A,B,C]\mapsto x=(B,A-C,A+C)$ is a bijection between the binary quadratic forms $Au^2+Buv+Cv^2$ of discriminant $B^2-4AC=D$ and the solutions of $Q(x)=D$ with $x_2\equiv x_3\bmod 2$. Under this map $x_3=A+C$. So $N(D)$ counts binary forms of discriminant $D$ with $|A+C|$ small.

\emph{Step 2: geodesics near $i$.} A form $q=[A,B,C]$ of positive discriminant determines the geodesic $\ell_q$ in the upper half-plane $\Hh$ whose endpoints are the two real roots of $q(x,1)$. Equivalent forms have equivalent geodesics, so the classes of forms correspond to geodesics in the quotient $\Gamma\backslash\Hh$, where $\Gamma=\PSL$.
\begin{itemize}
\item If $D$ is not a square, the roots are conjugate quadratic irrationals, and in $\Gamma\backslash\Hh$ the geodesic closes up: it is a \emph{closed geodesic} of finite length.
\item If $D=n^2$, the roots are rational, i.e.\ they are cusps. In $\Hh$ the geodesic joins two different cusps; in $\Gamma\backslash\Hh$, which has only one cusp, both of its ends run into that cusp, and its length is infinite.
\end{itemize}
The key identity (Lemma~\ref{dist}) is
\[
x_3=A+C=\sqrt D\,\sinh r,
\]
where $r$ is the signed hyperbolic distance from $i$ to $\ell_q$. Hence $|x_3|\le 2M$ means that $\ell_q$ meets the hyperbolic ball of radius $r_0=\operatorname{arcsinh}(2M/\sqrt D)\approx 2M/\sqrt D$ around $i$. \emph{Counting small solutions is counting geodesics through a small ball around $i$.} Figure~\ref{fig:geod} shows this for $D=249$, and Figure~\ref{fig:sqgeod} in Section~\ref{sec:square} for $D=7^2$.

\begin{figure}[t]
\centering
\includegraphics[width=\textwidth]{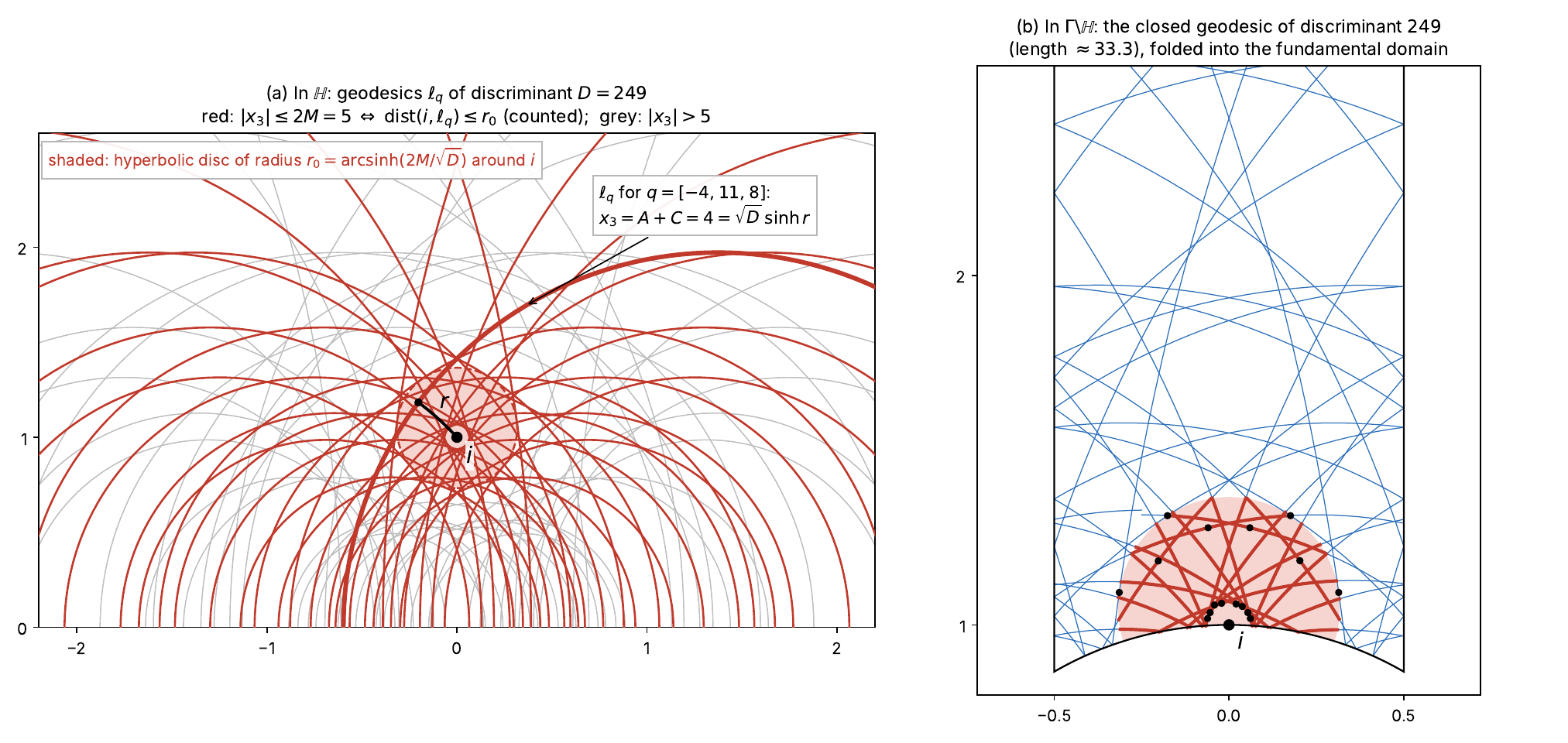}
\caption{The geodesics of discriminant $D=249$ and the ball of radius $r_0=\operatorname{arcsinh}(2M/\sqrt D)$ around $i$, for $M=2.5$. (a) In $\Hh$: the geodesics $\ell_q$ that meet the ball (red) are exactly those with $|x_3|=|A+C|\le 5$. One of them is shown with its distance $r$ to $i$. (b) In $\Gamma\backslash\Hh$ all these geodesics become one closed geodesic of length $2\log\varepsilon^+_{249}\approx33.3$, drawn in the fundamental domain. It passes $16$ times through the ball (black dots: closest approaches). Each passage accounts for $4$ forms ($q\circ\gamma$ and $q\circ\gamma S$, since $S$ fixes $i$, and their negatives), so $N(249)$ with the sharp cut-off $|x_3|\le5$ equals $64$.}\label{fig:geod}
\end{figure}

\emph{Step 3: from geodesics to an automorphic function.} By Step~2,
\[
N(D)=\sum_{q}k\big(\dist(i,\ell_q)\big),\qquad k(r)=\Phi\Big(\frac{\sqrt D\sinh r}{M}\Big),
\]
the sum running over all forms of discriminant $D$. We now choose a kernel $K(d)$, a function of the hyperbolic distance $d=d(z,w)$ between two points, such that
\[
\int_\ell K\big(d(z,w)\big)\dd s(w)=k\big(\dist(z,\ell)\big)
\]
for every geodesic $\ell$ and every point $z$ (Lemma~\ref{kernel}). Here $\dd s$ is the hyperbolic arc length. Such a $K$ exists by Abel inversion, but we never need it explicitly. It is supported in $d\le r_0$, so $K(d(z,\cdot))$ is a smooth bump on the ball of radius $r_0$ around $z$.

With this kernel the count becomes an integral over geodesics. We write this out for non-square $D$, in three lines:
\begin{align*}
N(D)&=\sum_{q}\ \int_{\ell_q}K\big(d(i,w)\big)\dd s(w)\\
&=\sum_{[q]}\ \sum_{\gamma\in\Gamma_q\backslash\Gamma}\ \int_{\ell_q}K\big(d(\gamma i,w)\big)\dd s(w)\\
&=\sum_{[q]}\ \int_{C_q}\ \sum_{\gamma\in\Gamma}K\big(d(\gamma i,w)\big)\dd s(w).
\end{align*}
\begin{itemize}
\item \emph{First line.} This is the definition of $K$, applied with $z=i$ and $\ell=\ell_q$.
\item \emph{Second line.} We group the forms into classes $[q]$. The forms in the class of $q$ are the forms $q\circ\gamma$, where $\gamma$ runs over $\Gamma$ modulo the stabiliser $\Gamma_q$ of $q$. The geodesic of $q\circ\gamma$ is $\gamma^{-1}\ell_q$, and moving it back by the isometry $\gamma$ replaces the point $i$ by $\gamma i$.
\item \emph{Third line.} The stabiliser $\Gamma_q$ is generated by one motion, which shifts $\ell_q$ along itself by a fixed length $L$. A piece of $\ell_q$ of length $L$ is one full turn of the closed geodesic $C_q$ on the modular surface. The integral over all of $\ell_q$ of a sum over $\Gamma_q\backslash\Gamma$ is the same as the integral over one such piece of the sum over all of $\Gamma$. This is the principle $\int_\R f(x)\dd x=\int_0^1\sum_{m\in\Z}f(x+m)\dd x$; this manipulation is called \emph{unfolding}.
\end{itemize}
The inner sum in the third line is the function
\[
F(w)=\sum_{\gamma\in\Gamma}K\big(d(\gamma i,w)\big).
\]
It is $\Gamma$-invariant (``automorphic''): it is a sum of bumps, one on the ball of radius $r_0$ around each point $\gamma i$ of the orbit of $i$. So
\[
N(D)=\sum_{[q]}\int_{C_q}F\dd s=:\Tr_D(F),
\]
the integral of $F$ over the closed geodesics of discriminant $D$. We call it the \emph{trace} of $F$. It measures how much of these closed geodesics lies in the ball around $i$, exactly as in Figure~\ref{fig:geod}(b). In the square case one integrates over the infinite geodesics instead, and Hecke operators reduce them all to the single geodesic from $0$ to $\infty$.

\emph{Step 4: spectral expansion.} The function $F$ is $\Gamma$-invariant, so it can be expanded into the eigenfunctions of the Laplacian: the constant function, the Maass cusp forms $u_j$ (discrete spectrum), and the Eisenstein series $E(\cdot,s)$ (continuous spectrum). Because $K$ depends only on the distance, every eigenfunction $u$ with eigenvalue $1/4+t^2$ enters with coefficient $h(t)u(i)$, where $h$ is a single function, the Selberg transform of $K$. Integrating over the geodesics gives
\[
N(D)=\text{constant term}+\sum_jh(t_j)\,u_j(i)\Tr_D(u_j)+\text{Eisenstein term}.
\]

\emph{Step 5: evaluation.} We now evaluate the three terms.
\begin{itemize}
\item \emph{The Selberg transform.} Instead of computing $K$, we compute $h$ directly from $k$, using coordinates adapted to a geodesic (``position along the geodesic, distance from it''). Two facts matter: $h$ at the constant eigenfunction equals $(M/\sqrt D)\int\Phi$, and $h(t)$ is small for $t\gg\sqrt D/M$. So only the Maass forms with $t_j\lesssim\sqrt D/M$ matter.
\item \emph{The main term.} The constant function contributes $(3/\pi)\cdot(M/\sqrt D)\int\Phi\cdot\Tr_D(1)$, where $\Tr_D(1)$ is the total length of the closed geodesics of discriminant $D$. This is what one expects if the geodesics were randomly distributed: a strip of width about $2r_0$ around geodesics of total length $\Tr_D(1)$ has area about $2r_0\Tr_D(1)$, and $\Gamma\backslash\Hh$ has area $\pi/3$. Dirichlet's class number formula, together with the Hecke relation below, gives $\Tr_D(1)=2\sqrt D\,L(1,\chi_{D_0})\h_{D_0}(f)$, and this yields the main term of Theorem~\ref{thm:ns}.
\item \emph{Reduction by Hecke operators.} For $D=f^2D_0$ the trace $\Tr_D$ is a combination of Hecke operators applied to $\Tr_{D_0}$ (Theorem~\ref{thm:heckerel}). So it suffices to know the traces at the fundamental discriminant $D_0$; the Hecke eigenvalues at $f$ are bounded by Kim--Sarnak.
\item \emph{Periods are $L$-values.} The integrals of Maass forms over the geodesics are central values of $L$-functions. Over the closed geodesics, $|\Tr_{D_0}(u_j)|^2$ is a multiple of $\sqrt{D_0}\,L(1/2,u_j)\,L(1/2,u_j\otimes\chi_{D_0})$ (Duke--Imamo\={g}lu--T\'oth, Baruch--Mao). Along the geodesic from $0$ to $\infty$, $\int_0^\infty u_j(iy)\dd y/y$ is a multiple of $L(1/2,u_j)$ (Hecke). For the Eisenstein series one obtains products of $\zeta$ and Dirichlet $L$-functions.
\item \emph{The error term.} By Cauchy--Schwarz we need two averages over $t_j\lesssim\sqrt D/M$: of the values $u_j(i)$ (local Weyl law) and of the periods. For the latter we use, in the non-square case, Motohashi's second moment of $L(1/2,u_j)$ and the cubic moment of $L(1/2,u_j\otimes\chi_{D_0})$ due to Young and Andersen--Wu; in the square case, an approximate functional equation and the spectral large sieve of Deshouillers and Iwaniec. The Hecke eigenvalues at $f$ (or at the divisors of $n$) are bounded for each $u_j$ separately by Kim--Sarnak. This is the bottleneck, and it produces the exponent $71/96=(2/3)(1+\theta)$.
\end{itemize}
Comparing the error term with the main term, which is of size $M$ (up to $D^{\varepsilon}$), gives the ranges stated above.

\emph{What changes for square discriminants.} For $D=n^2$ the geodesics have infinite length. Then the integral of the constant function along them, $\int_0^\infty\dd y/y$, diverges, and so does that of the Eisenstein series, which grows like $y^s$ in the cusp. The function $F$ itself vanishes high up in the cusp, so we may cut the geodesic off at heights $Y$ and $1/Y$ with $Y$ large. The constant term and the Eisenstein term then each contain a term of size $\log Y$, and these two terms cancel. What remains is finite; we call it the \emph{regularised} integral. It produces the factor $\log n$ in Theorem~\ref{thm:sqsimple}, together with explicit secondary terms. This is the only point where the square case is genuinely different (Subsection~\ref{sec:sqoverview}).

\subsection{What is new}\label{sec:new}
Relative to the works quoted in Subsection~\ref{sec:previous}, the new points of this paper are the following. First, the ternary problem is treated through geodesics near $i$, with explicit main terms, and this gives the improvement of \cite{FrIw} for $D\equiv5\bmod8$ in Corollary~\ref{cor:tern}. Second, the power saving is uniform in the conductor $f$, so that non-fundamental discriminants are included; this rests on the Hecke relation for traces (Theorem~\ref{thm:heckerel}), which we prove by an elementary sublattice count. Third, the square case $D=n^2$, where the geodesics are not closed, is treated by a regularised spectral expansion, with the theta group for the parity of $x_3$, explicit secondary terms, and all $n$.

\subsection{Plan of the paper}
Section~\ref{sec:prelim} collects the background from four fields: binary quadratic forms, hyperbolic geometry, the spectral theory of automorphic forms (including Hecke operators), and $L$-functions. It also fixes the convention on the two groups which are used (Subsection~\ref{sec:Gconv}). We have tried to make it accessible to readers coming from classical analytic number theory; readers who are familiar with the spectral theory of automorphic forms may skip most of it. Section~\ref{sec:theta} explains the relation with theta functions. Section~\ref{sec:count} carries out Steps~0--3: it turns the count into a sum over geodesics and constructs the kernel. Section~\ref{sec:nonsquare} treats non-square discriminants and proves Theorem~\ref{thm:ns} and Corollary~\ref{cor:tern}. Section~\ref{sec:square} treats square discriminants and proves Theorem~\ref{thm:sqsimple}. Section~\ref{sec:num} reports numerical checks and lists open problems. An index of notation is at the end of this section.

\subsection{Contributions of the author and of Claude}\label{sec:contrib}
This paper is the outcome of a collaboration between the author and Claude, an AI model developed by Anthropic. The mathematical work and the first three versions of the text were done with Claude Opus~5.5 (High). The final revision of the presentation, which led to the present version, was done with Claude Fable~5.2.

The author posed the problem and wrote the original paper \cite{B24} on the square case. He raised the doubts about its main term and asked whether averages of Kloosterman sums could improve its error term, which led to the corrected and improved version \cite{B26}. He insisted that the connection with non-holomorphic theta functions be explored (Section~\ref{sec:theta}), which led to the spectral approach of this paper. He required that all $n$ and the full lattice $\Z^3$ be covered in the square case, and asked for the extension to general discriminants. He fixed the standards (every claim proved, or explicitly flagged as cited or open) and the scope, evaluated the successive drafts, and asked for the present presentation.

Claude proposed the spectral approach: the counting function as a sum over geodesics, reduced by Hecke operators to a single geodesic (square case) or to the fundamental discriminant (non-square case). Claude proved the exact formulas (Theorems~\ref{thm:nsexact} and~\ref{thm:sqexact}), the Hecke relation (Theorem~\ref{thm:heckerel}) including the sublattice count behind it, the parity reductions, and the results for the theta group (its Hecke operators and its Eisenstein series); derived the main theorems; and proved the mean-square bounds for the periods, identifying and combining the cited inputs \cite{DI,Pa,Kim,DIT,BM,Yo,AW,Mo}. Claude also carried out the numerical checks (Section~\ref{sec:num}), checked the cited statements against the sources available, wrote the \LaTeX{} text, and made a line-by-line check of the final version.

Results from the literature are stated as Facts, with references, and are not re-proved here. The author takes responsibility for the final version.

\subsection{Index of notation}\label{sec:index}\mbox{}\par\smallskip
{\small\renewcommand{\arraystretch}{1.55}\noindent
\begin{tabularx}{\textwidth}{@{}>{\raggedright}p{0.25\textwidth}>{\raggedright}X>{\raggedright\arraybackslash}p{0.13\textwidth}@{}}
\toprule
Symbol & Meaning & Where\\
\midrule
$Q$, $\Phi$, $M$ & ternary form, smooth weight, size of $x_3$ & \S\ref{sec:problem}\\
$\sigma$, $\tau$, $\mu$, $\Lambda$, $v_2$ & sum of divisors, number of divisors, Möbius, von Mangoldt, $2$-adic valuation & \S\ref{sec:results}\\
$\Rr_e(D)$, $\Rr(D)$, $N(D)$ & counts of solutions: $x_3\equiv e$; all; $x_2\equiv x_3\bmod 2$ & \S\ref{sec:problem}\\
$\Psi_0$, $\theta$ & $\int\Phi(x/M)\dd x$; Kim--Sarnak exponent $7/64$ & \S\ref{sec:results}\\
$D=f^2D_0$, $\chi_{D_0}$, $\h_{D_0}(f)$ & conductor and fundamental discriminant, character, \eqref{hdef} & \S\ref{sec:results}\\
$[A,B,C]$, $q$, $\Qf(D)$ & binary form $Au^2+Buv+Cv^2$; set of forms of discriminant $D$ & \S\ref{sec:forms}\\
$\Hh$, $d(z,w)$, $\dd\mu$, $\Delta$ & upper half-plane, distance, measure, Laplacian & \S\ref{sec:hyp}\\
$\Gamma$, $\Gt$, $S$, $T$, $R$ & $\PSL$, theta group, $z\mapsto-1/z$, $z\mapsto z+1$, $z\mapsto-\bar z$ & \S\ref{sec:hyp}--\ref{sec:Gconv}\\
$\ell_q$, $Q_q(z)$, $C_q$, $\Gamma_q$ & geodesic of $q$, $(A|z|^2+Bx+C)/y$, closed geodesic, stabiliser & \S\ref{sec:quot}, \S\ref{sec:dist}\\
$\Tr_D(F)$ & sum of the integrals of $F$ over the closed geodesics of discriminant $D$ & \S\ref{sec:quot}\\
$u_j$, $t_j$, $\rho_j(m)$, $\lambda_j(d)$, $\alpha_j$ & Maass forms, spectral parameters, coefficients, Hecke eigenvalues, $|\rho_j(1)|^2/\cosh\pi t_j$ & \S\ref{sec:spec}\\
$E(z,s)$, $\eta_d(s)$ & Eisenstein series, its Hecke eigenvalues $\sum_{ab=d}a^sb^{1-s}$ & \S\ref{sec:spec}\\
$G$, $w_G$, $\vol_G$, $\mathcal A_G$, $T_d^G$ & $\Gamma$ or $\Gt$ (Table~\ref{tab:G}), cusp width, volume, admissible $n$, Hecke operators & \S\ref{sec:Gconv}, \S\ref{sec:hecke}\\
$\delta$, $k$, $r_0$, $K$, $h$, $H$ & $M/\sqrt D$, weight in the distance, its support, kernel, Selberg transform ($H(1/2+it)=h(t)$) & \S\ref{sec:kernel}\\
$A_j(f)$, $A_E(f,s)$; $A_j(n)$, $A(n,t)$ & Hecke factors, non-square and square case & \S\ref{sec:heckerel}, \S\ref{sec:sqoverview}\\
$\Pp_j$, $\Pp_E(s)$ & periods along $\ell_0=i\R_{>0}$ of $u_j$ and (regularised) of $E(\cdot,s)$ & \S\ref{sec:sqper}, \S\ref{sec:sqeisper}\\
$\Psi_1$, $\varpi(n)$, $\vpo(n)$, $c_1$, $\beta_e(n)$ & square case: main-term data & \S\ref{sec:sqoverview}\\
\bottomrule
\end{tabularx}}

\section{Preliminaries}\label{sec:prelim}
In this section we collect the background. It comes from four fields, and we treat them in turn: binary quadratic forms (Subsection~\ref{sec:forms}), hyperbolic geometry and the two groups that we need (Subsections~\ref{sec:hyp}, \ref{sec:quot} and~\ref{sec:Gconv}), the spectral theory of automorphic functions and Hecke operators (Subsections~\ref{sec:spec} and~\ref{sec:hecke}), and $L$-functions (Subsection~\ref{sec:Lfun}). Statements taken from the literature are called Facts. Our main reference for the spectral theory is Iwaniec's book \cite{Iw}; for binary quadratic forms we refer to Zagier \cite{Za}.

\subsection{Binary quadratic forms}\label{sec:forms}
A binary quadratic form is a polynomial $q(u,v)=Au^2+Buv+Cv^2$; we write $q=[A,B,C]$. Its discriminant is $B^2-4AC$, an integer $\equiv0,1\bmod 4$ if $q$ is integral. For a discriminant $D$ let
\[
\Qf(D)=\{[A,B,C]\in\Z^3:\ B^2-4AC=D\}
\]
be the set of all integral forms of discriminant $D$, primitive or not. The form is primitive if $\gcd(A,B,C)=1$.

The group $\mathrm{SL}_2(\R)$ acts on real forms from the right by substitution, $(q\circ\gamma)(u,v)=q(\gamma(u,v)^T)$, and this action preserves the discriminant. Since $q\circ(-I)=q$, the modular group $\Gamma=\PSL$ acts on $\Qf(D)$. Two forms in the same orbit are called equivalent, and the orbits are the classes of forms.

A discriminant $D_0$ is fundamental if every form of discriminant $D_0$ is primitive, i.e.\ if $D_0\equiv1\bmod4$ is square-free, or $D_0=4m$ with $m\equiv2,3\bmod4$ square-free. Every non-square discriminant $D>0$ can be written uniquely as $D=f^2D_0$ with $D_0>1$ fundamental; $f$ is the conductor. For fundamental $D_0$, the Kronecker symbol $\chi_{D_0}=\kron{D_0}{\cdot}$ is a real primitive character modulo $|D_0|$. For a general discriminant $D$ we also write $\chi_D=\kron{D}{\cdot}$; for an odd prime $p$, $\chi_D(p)$ is the Legendre symbol, and $\chi_D(2)=1,-1,0$ according as $D\equiv1\bmod8$, $D\equiv5\bmod8$, or $D$ is even.

For positive non-square $D$ the theory of forms is governed by the Pell equation. The following facts go back to Gauss and Dirichlet.

\begin{fact}[Gauss, Dirichlet; see {\cite{Za}}]\label{F:gauss}
Let $D>0$ be a non-square discriminant.
\begin{enumerate}[label=(\roman*)]
\item For $q=[A,B,C]\in\Qf(D)$ primitive, the stabiliser $\Gamma_q=\{\gamma\in\Gamma:q\circ\gamma=q\}$ consists of the matrices
\[
\pm\begin{pmatrix}\frac{t-Bu}{2}&-Cu\\ Au&\frac{t+Bu}{2}\end{pmatrix},\qquad t^2-Du^2=4,
\]
and is infinite cyclic. For non-primitive $q$, $\Gamma_q=\Gamma_{q/g}$ with $g$ the content of $q$.
\item The number of classes $\Gamma\backslash\Qf(D)$ is finite.
\item For primitive $q$, the group $\Gamma_q$ is generated by the element with $(t,u)$ the least positive solution of $t^2-Du^2=4$. We put $\varepsilon_D^+=(1/2)(t+u\sqrt D)$.
\item For fundamental $D_0>1$, the number of classes of forms of discriminant $D_0$ is the narrow class number $h^+(D_0)$, and Dirichlet's class number formula reads
\[
h^+(D_0)\log\varepsilon^+_{D_0}=2h(D_0)\log\varepsilon_{D_0}=\sqrt{D_0}L(1,\chi_{D_0}).
\]
\end{enumerate}
\end{fact}

For square discriminants $D=n^2$ the Pell equation $t^2-n^2u^2=4$, i.e.\ $(t-nu)(t+nu)=4$, has only the solutions with $u=0$. Accordingly, the forms of square discriminant $n^2$ have trivial stabiliser in $\Gamma$; there are exactly $n$ classes, described explicitly in Lemma~\ref{orbits}.

\subsection{The hyperbolic plane}\label{sec:hyp}
Let $\Hh=\{z=x+iy:y>0\}$ be the upper half-plane with the hyperbolic metric and measure
\[
\dd s=\frac{|\dd z|}{y},\qquad \dd\mu=\frac{\dd x\dd y}{y^2}.
\]
The hyperbolic distance $d(z,w)$ is given by \cite[\S1.1]{Iw}
\begin{equation}\label{hypdist}
\cosh d(z,w)=1+\frac{|z-w|^2}{2\,\Im z\,\Im w}.
\end{equation}
The geodesics are the vertical half-lines and the semicircles centred on the real axis. Their endpoints lie on the boundary $\R\cup\{\infty\}$, and every geodesic has infinite length. Real matrices of positive determinant act on $\Hh$ by Möbius transformations $z\mapsto(az+b)/(cz+d)$; these preserve $\dd s$ and $\dd\mu$, so $\mathrm{PSL}_2(\R)$ acts by isometries. The reflection $R:z\mapsto-\bar z$ is an orientation-reversing isometry. We also write
\[
S=\begin{pmatrix}0&-1\\1&0\end{pmatrix}:z\mapsto-\frac{1}{z},\qquad T=\begin{pmatrix}1&1\\0&1\end{pmatrix}:z\mapsto z+1.
\]
A motion $\gamma\ne\pm I$ is \emph{elliptic} if $|\mathrm{tr}\gamma|<2$ (it rotates around a fixed point in $\Hh$), \emph{parabolic} if $|\mathrm{tr}\gamma|=2$ (it has one fixed point on the boundary), and \emph{hyperbolic} if $|\mathrm{tr}\gamma|>2$. A hyperbolic motion fixes two boundary points and translates the geodesic joining them along itself by the distance $2\operatorname{arccosh}(|\mathrm{tr}\gamma|/2)$ \cite[\S1.5]{Iw}.

\emph{Distance to the imaginary axis.} Let $\ell_0=i\R_{>0}$. For $w=x+iy$, the point of $\ell_0$ nearest to $w$ is $i|w|$, and \eqref{hypdist} gives $\cosh d(w,i|w|)=|w|/y$, i.e.\ $\sinh d(w,\ell_0)=|x|/y$. We define the \emph{signed} distance $r$ from $w$ to $\ell_0$ by
\begin{equation}\label{signeddist}
\sinh r=\frac{\Re w}{\Im w}.
\end{equation}

\emph{Fermi coordinates.} It is convenient to describe a point $w$ by its position along $\ell_0$ and its signed distance from $\ell_0$. Put
\[
w(\sigma,r)=e^{\sigma}\,\frac{\sinh r+i}{\cosh r}\qquad(\sigma,r\in\R).
\]
Then $|w|=e^\sigma$, so the nearest point of $\ell_0$ is $ie^\sigma$, at arc length $\sigma$ from $i$; and $\Re w/\Im w=\sinh r$, so $r$ is the signed distance to $\ell_0$. The curves $r=\mathrm{const}$ are the rays from $0$ (the points at fixed distance from $\ell_0$), and the curves $\sigma=\mathrm{const}$ are the semicircles centred at $0$ (the geodesics perpendicular to $\ell_0$); see Figure~\ref{fig:fermi}. In these coordinates
\begin{equation}\label{fermi}
\dd\mu=\cosh r\dd\sigma\dd r,\qquad \Delta=-\big(\partial_r^2+\tanh r\,\partial_r+\cosh^{-2}r\,\partial_\sigma^2\big),
\end{equation}
where $\Delta=-y^2(\partial_x^2+\partial_y^2)$ is the hyperbolic Laplacian. The Laplacian commutes with all isometries, and $\Delta y^s=s(1-s)y^s$.

\begin{figure}[t]
\centering
\begin{tikzpicture}[scale=1.55]
\draw[->] (-0.4,0)--(3.3,0) node[right] {$\Re$};
\draw[very thick] (0,0)--(0,3.6) node[above] {$\ell_0=i\R_{>0}$};
\foreach \a in {25,40,53.13,70} {\draw[gray!60] (0,0)--(\a:3.9);}
\draw[gray!60] (53.13:3.9) node[above right,black,font=\small] {$r=\mathrm{const}$};
\foreach \rr in {0.9,2,3.2} {\draw[gray!60] (\rr,0) arc (0:90:\rr);}
\draw[gray!60] (3.2,0) node[below,black,font=\small] {$\sigma=\mathrm{const}$};
\draw[thick,blue] (2,0) arc (0:90:2);
\draw[thick,red] (0,2) arc (90:53.13:2);
\node[red,font=\small] at (0.65,2.15) {$r$};
\fill (1.2,1.6) circle (1.2pt) node[right,font=\small] {$z=w(\sigma_0,r)$};
\fill (0,2) circle (1.2pt) node[left,font=\small] {$ie^{\sigma_0}$};
\fill (0,0.9) circle (1.2pt) node[left,font=\small] {$w$};
\fill (0,3.2) circle (1.2pt) node[left,font=\small] {$w'$};
\draw[dashed] (1.2,1.6) arc (94.6:145.9:1.605);
\draw[dashed] (1.2,1.6) arc (22.83:50.9:4.123);
\node[font=\small] at (0.55,1.05) {$d(z,w)$};
\node[font=\small] at (0.75,2.85) {$d(z,w')$};
\end{tikzpicture}
\caption{Fermi coordinates around $\ell_0=i\R_{>0}$: the rays from $0$ are the curves at fixed distance $r$ from $\ell_0$, the semicircles centred at $0$ are the geodesics perpendicular to $\ell_0$. The kernel of Lemma~\ref{kernel} is chosen so that, for every point $z$ at distance $r$ from $\ell_0$, the integral of $K(d(z,w))$ over all $w\in\ell_0$ equals the prescribed weight $k(r)$.}\label{fig:fermi}
\end{figure}

\subsection{The modular group and the modular surface}\label{sec:quot}
Let $\Gamma=\PSL$. It is generated by $S$ and $T$. The set
\[
\mathcal F=\{z\in\Hh:\ |\Re z|\le\tfrac{1}{2},\ |z|\ge1\}
\]
is a fundamental domain: every orbit $\Gamma z$ meets $\mathcal F$, and in exactly one point unless it meets the boundary, whose edges are glued by $T$ and $S$ \cite[\S2.3]{Iw}. The quotient $\Gamma\backslash\Hh$, the set of orbits, is the modular surface; it has area $\vol(\Gamma\backslash\Hh)=\pi/3$. Functions on $\Gamma\backslash\Hh$ are the same as $\Gamma$-invariant functions on $\Hh$, which we call automorphic. The point $i$ is fixed by $S$, so its stabiliser in $\Gamma$ is $\{1,S\}$.

\emph{Cusps.} The rational points of the boundary, together with $\infty$, are the cusps. They form a single $\Gamma$-orbit, so $\Gamma\backslash\Hh$ has one cusp, the region $y\to\infty$ of $\mathcal F$. The stabiliser of $\infty$ is $\Gamma_\infty=\langle T\rangle$. Two elements of $\Gamma$ lie in the same right coset $\Gamma_\infty\gamma$ if and only if they have the same bottom row (up to sign, since $\Gamma=\PSL$). Indeed,
\[
T^m\begin{pmatrix}a&b\\c&d\end{pmatrix}=\begin{pmatrix}a+mc&b+md\\c&d\end{pmatrix},\qquad
\begin{pmatrix}a'&b'\\c&d\end{pmatrix}\begin{pmatrix}a&b\\c&d\end{pmatrix}^{-1}=\begin{pmatrix}1&ab'-a'b\\0&1\end{pmatrix}.
\]
The first identity shows that $T^m$ does not change the bottom row. The second shows that two matrices with the same bottom row differ by an element of $\Gamma_\infty$: for fixed $(c,d)$, the solutions of $ad-bc=1$ are exactly $(a,b)+m(c,d)$, $m\in\Z$. Every coprime pair $(c,d)$ occurs as a bottom row, since $ad-bc=1$ is solvable. Hence $\Gamma_\infty\backslash\Gamma$ is in bijection with the coprime pairs $(c,d)$ modulo $\pm1$.

\emph{Averaging and unfolding.} If a function $f$ on $\Hh$ is invariant under a subgroup $\Gamma'$ of $\Gamma$, then the distinct translates $f\circ\gamma$ correspond to the cosets $\Gamma'\gamma$, and
\[
F(z)=\sum_{\gamma\in\Gamma'\backslash\Gamma}f(\gamma z)
\]
is $\Gamma$-invariant (when the sum converges absolutely). Each translate is counted once; summing over all of $\Gamma$ would count each one $|\Gamma'|$ times. Integrating over a fundamental domain, the sum can be ``unfolded'':
\begin{equation}\label{unfold}
\int_{\Gamma\backslash\Hh}F\dd\mu=\int_{\Gamma'\backslash\Hh}f\dd\mu ,
\end{equation}
because the translates $\gamma\mathcal F$, $\gamma\in\Gamma'\backslash\Gamma$, tile a fundamental domain of $\Gamma'$. The same holds for integrals over curves instead of areas. This principle is used for the Eisenstein series below ($\Gamma'=\Gamma_\infty$) and for the classes of forms ($\Gamma'=\Gamma_q$, Theorem~\ref{thm:nsexact}). For the kernel of Step~3 we sum over all of $\Gamma$; since the stabiliser of $i$ is $\{1,S\}$, each translate is then counted twice.

\emph{Closed geodesics.} Let $\gamma\in\Gamma$ be hyperbolic, and let $\ell$ be the geodesic joining its two fixed points. Then $\gamma$ maps $\ell$ to itself and shifts every point of $\ell$ along $\ell$ by the same hyperbolic length $L>0$. If we describe the points of $\ell$ by the arc length $\sigma\in\R$, measured from some base point, then $\gamma$ acts on $\ell$ as $\sigma\mapsto\sigma+L$.

On the modular surface the points $w$ and $\gamma w$ are the same point. So when $\sigma$ runs through $\R$, the image of $\ell$ in $\Gamma\backslash\Hh$ repeats itself with period $L$: it is a closed curve of length $L$, a \emph{closed geodesic}. In the notation of quotients, $\langle\gamma\rangle\backslash\ell$ is the set of points of $\ell$ up to the shifts $\sigma\mapsto\sigma+mL$ ($m\in\Z$), that is, the circle $\R/L\Z$. A $\Gamma$-invariant function $F$, restricted to $\ell$, is an $L$-periodic function of $\sigma$, and its integral over the closed geodesic is the integral over one period,
\[
\int_{\langle\gamma\rangle\backslash\ell}F\dd s=\int_{0}^{L}F\big(w(\sigma)\big)\dd\sigma ,
\]
where $w(\sigma)$ is the point of $\ell$ at arc length $\sigma$.

\emph{The geodesic of a form.} For a real form $q=[A,B,C]$ of positive discriminant let $\ell_q$ be the geodesic of $\Hh$ whose endpoints are the two zeros of $q$ on $\R\cup\{\infty\}$, i.e.\ the roots of $q(x,1)$, with $\infty$ in place of a root if $A=0$. Then $\ell_{q\circ\gamma}=\gamma^{-1}\ell_q$ (Lemma~\ref{dist}). Now let $D>0$ be a non-square discriminant and $q\in\Qf(D)$, the set of integral forms of discriminant $D$. Then the roots are conjugate quadratic irrationals, and the stabiliser $\Gamma_q$ of $q$ in $\Gamma$ is generated by a hyperbolic element which fixes them (Fact~\ref{F:gauss}(i)). Hence
\[
C_q=\Gamma_q\backslash\ell_q
\]
is a closed geodesic in $\Gamma\backslash\Hh$. For primitive $q$ its length is $2\log\varepsilon_D^+$, since the generator of $\Gamma_q$ has trace $t$ and $2\operatorname{arccosh}(t/2)=2\log\varepsilon^+_D$ (note that $\varepsilon_D^+\cdot(1/2)(t-u\sqrt D)=1$). Equivalent forms give the same closed geodesic. For a continuous $\Gamma$-invariant function $F$ we define the trace
\begin{equation}\label{trdef}
\Tr_D(F)=\sum_{[q]\in\Gamma\backslash\Qf(D)}\int_{C_q}F\dd s ,
\end{equation}
with $\dd s$ the arc length, and $\Tr_D:=0$ if $D$ is not a discriminant. The sum is finite by Fact~\ref{F:gauss}(ii). In particular, $\Tr_D(1)$ is the total length of the closed geodesics of discriminant $D$.

The name is taken from the definite case. For a negative discriminant, the classes of positive definite forms give finitely many points of $\Gamma\backslash\Hh$, the Heegner points, and the sum of the values of a modular function at these points is called a trace (Zagier's ``traces of singular moduli''). For positive discriminants the points are replaced by closed geodesics, and the values by integrals.

For a square discriminant the roots of $q(x,1)$ are rational (or $\infty$), so $\ell_q$ joins two cusps. Since $\Gamma_q$ is trivial, nothing is identified along $\ell_q$, and its image in $\Gamma\backslash\Hh$ has infinite length, with both ends in the cusp.

\subsection{The theta group, and the convention on \texorpdfstring{$G$}{G}}\label{sec:Gconv}
Besides the modular group $\Gamma=\PSL$ we need one further group, a subgroup of $\Gamma$ called the \emph{theta group}:
\[
\Gt=\Big\{\gamma\in\PSL:\gamma\equiv\begin{pmatrix}1&0\\0&1\end{pmatrix}\ \text{or}\ \begin{pmatrix}0&1\\1&0\end{pmatrix}\bmod2\Big\}.
\]
It is the group under which Jacobi's theta function $\sum_ne^{\pi in^2z}$ transforms as a modular form of weight $1/2$, whence the name; in this paper only the group matters. It enters for the following reason: for $D\equiv1\bmod8$, the forms $[A,B,C]$ of discriminant $D$ with $A$ and $C$ even, which we need in order to separate the parities of $x_3$, are permuted by $\Gt$ but not by all of $\Gamma$ (Lemma~\ref{orbits}).

\emph{Basic properties.} The group $\Gt$ has index $3$ in $\Gamma$: reduction modulo $2$ maps $\Gamma$ onto $\mathrm{SL}_2(\Z/2)$, which has six elements, and $\Gt$ is the preimage of a subgroup of order $2$. It contains $S$ and $T^2$, but not $T$. It is normalised by the reflection $R$, since $(\begin{smallmatrix}a&-b\\-c&d\end{smallmatrix})\equiv(\begin{smallmatrix}a&b\\c&d\end{smallmatrix})\bmod2$. Its covolume is $3\cdot\pi/3=\pi$.

\emph{Fundamental domain.} A fundamental domain for $\Gt$ is
\[
\mathcal F_\vartheta=\{z\in\Hh:\ |\Re z|\le1,\ |z|\ge1\},
\]
see Figure~\ref{fig:theta}. To see this, note that $\Gamma$ is the disjoint union of the cosets $\Gt$, $\Gt T$ and $\Gt TS$ (modulo $2$ these are three different cosets, of two elements each). Hence $\mathcal F\cup T\mathcal F\cup TS\mathcal F$ is a fundamental domain for $\Gt$. Here
\[
T\mathcal F=\{\tfrac{1}{2}\le x\le\tfrac{3}{2},\ |z-1|\ge1\},\qquad TS\mathcal F=\{|z-1|\le1,\ |z|\ge1,\ |z-2|\ge1\}.
\]
We cut these two tiles along the line $x=1$ and move the parts to the right of it by $T^{-2}\in\Gt$. The parts with $x\le1$ fill up $\{1/2\le x\le1,\ |z|\ge1\}$, and the shifted parts fill up $\{-1\le x\le-1/2,\ |z|\ge1\}$. Together with $\mathcal F$ this is $\mathcal F_\vartheta$. The vertical sides of $\mathcal F_\vartheta$ are glued by $T^2$, and the two halves of the arc by $S$.

\begin{figure}[t]
\centering
\begin{tikzpicture}[scale=2.3]
\fill[gray!18] (-1,2.55)--(-1,0) arc (180:0:1)--(1,2.55)--cycle;
\fill[gray!38] (-0.5,2.55)--(-0.5,0.866) arc (120:60:1)--(0.5,2.55)--cycle;
\draw[->] (-1.7,0)--(1.8,0) node[right] {$\Re z$};
\draw[thick] (-1,0)--(-1,2.55); \draw[thick] (1,0)--(1,2.55);
\draw[thick] (1,0) arc (0:180:1);
\draw[dashed] (-0.5,0.866)--(-0.5,2.55); \draw[dashed] (0.5,0.866)--(0.5,2.55);
\draw[dashed] (0.5,0.866) arc (120:90:1);
\draw[dashed] (-0.5,0.866) arc (60:90:1);
\fill (0,1) circle (0.7pt) node[above,font=\small] {$i$};
\node[font=\small] at (0,1.9) {$\mathcal F$};
\node[font=\small] at (0.76,1.9) {$T\mathcal F$};
\node[font=\small] at (-0.76,1.9) {$T^{-1}\mathcal F$};
\node[font=\scriptsize] at (0.79,0.86) {$TS\mathcal F$};
\node[font=\scriptsize] at (-0.77,0.86) {$T^{-1}\!S\mathcal F$};
\foreach \x/\l in {-1/{-1},-0.5/{-\frac12},0/0,0.5/{\frac12},1/1} {\draw (\x,0.03)--(\x,-0.03) node[below,font=\small] {$\l$};}
\node[font=\small] at (0,2.75) {cusp $\infty$ (width $2$)};
\node[font=\small,align=center] at (1.55,0.35) {cusp $1$\\(width $1$)};
\node[font=\small,align=center] at (-1.6,0.35) {$-1=T^{-2}(1)$};
\draw[<->,gray!70!black] (-0.95,2.3)--(0.95,2.3); \node[font=\scriptsize,fill=gray!38,inner sep=1pt] at (0,2.3) {$T^2$};
\draw[<->,gray!70!black] (-0.62,0.62) to[bend left=35] (0.62,0.62); \node[font=\scriptsize] at (0,0.48) {$S$};
\end{tikzpicture}
\caption{The fundamental domain $\mathcal F_\vartheta=\{|\Re z|\le1,\ |z|\ge1\}$ of the theta group (shaded). It consists of the fundamental domain $\mathcal F$ of $\Gamma$ (dark) and two further copies of it, each cut into two pieces. The domain reaches the boundary at $\infty$ and at $\pm1$; the points $1$ and $-1$ are equivalent under $T^2$, so $\Gt$ has two cusps, $\infty$ and $1$.}\label{fig:theta}
\end{figure}

\emph{Cusps.} The group $\Gt$ has two cusps, $\infty$ and $1$:
\begin{itemize}
\item the $\Gt$-orbit of $\infty$ is $\{p/q:p\not\equiv q\bmod 2\}$, with stabiliser $\langle T^2\rangle$. The cusp has width $2$, and the scaling matrix is $\sigma_\infty z=2z$;
\item the orbit of $1$ is $\{p/q:p\equiv q\equiv1\bmod2\}$, with stabiliser $\Gamma_1=\sigma_1\langle T\rangle\sigma_1^{-1}$, where $\sigma_1=(\begin{smallmatrix}1&0\\1&1\end{smallmatrix})$. The cusp has width $1$; indeed $\sigma_1T^k\sigma_1^{-1}\equiv(\begin{smallmatrix}0&1\\1&0\end{smallmatrix})^k\bmod2$.
\end{itemize}

\medskip
\noindent\fbox{\begin{minipage}{0.97\textwidth}\smallskip
\textbf{Convention (the group $G$).} \emph{Throughout the paper, $G$ denotes one of the two groups $\Gamma=\PSL$ or $\Gt$.} Objects which depend on the group carry an index $G$, which is specialised to $\Gamma$ or $\vartheta$; they are listed in Table~\ref{tab:G}. The theta group is needed only in the square case (Section~\ref{sec:square}), and there only to separate the two parities of $x_3$ when $n$ is odd. \emph{On first reading one may take $G=\Gamma$ everywhere and ignore the index $G$}: then $w_G=1$, $\vol_G=\pi/3$, and the Hecke operators are the classical ones.\smallskip
\end{minipage}}

\begin{table}[ht]
\centering\small\renewcommand{\arraystretch}{1.45}
\begin{tabular}{llll}
\toprule
 & notation & $G=\Gamma$ & $G=\Gt$\\
\midrule
area of $G\backslash\Hh$ & $\vol_G$ & $\pi/3$ & $\pi$\\
cusps & & $\infty$ & $\infty$ and $1$\\
width of the cusp $\infty$ & $w_G$ & $1$ & $2$\\
contains & & $S$, $T$ & $S$, $T^2$\\
forms permuted by $G$ & $\Qf_G(D)$ & $\Qf(D)$ & $\Qf^\vartheta(D)$: $A$, $C$ even ($D\equiv1\bmod8$)\\
weighted count of these forms & $N_G(D)$ & $N(D)$ & $N_\vartheta(D)$\\
admissible $n$ in the square case & $\mathcal A_G$ & all $n\ge1$ & odd $n$\\
Hecke operators $T^G_d$ defined for & & all $d\ge1$ & odd $d$\\
\bottomrule
\end{tabular}
\caption{The two groups and the objects attached to them. The sets $\Qf^\vartheta(D)$ and the counts $N_\vartheta(D)$ are defined in Subsection~\ref{sec:parity}, the Hecke operators in Subsection~\ref{sec:hecke}.}\label{tab:G}
\end{table}

\emph{Conjugation to $\Gamma_0(2)$.} The theta group is conjugate to the Hecke congruence group $\Gamma_0(2)$. This allows us to import results which are stated in the literature for $\Gamma_0(2)$. For $d\ge1$ let $M_d$ be the set of integral $2\times2$ matrices of determinant $d$ modulo $\pm1$, and for odd $d$ let
\[
M_d^\vartheta=\{\mu\in M_d:\mu\equiv I\ \text{or}\ (\begin{smallmatrix}0&1\\1&0\end{smallmatrix})\bmod2\},\qquad M^{(0)}_d=\{\mu\in M_d:\mu_{21}\equiv0\bmod2\},
\]
so that $M_1^\vartheta=\Gt$ and $M_1^{(0)}=\Gamma_0(2)$. Put $M_d^\Gamma=M_d$ and $g=(\begin{smallmatrix}2&1\\0&1\end{smallmatrix})$, which acts by $gz=2z+1$.

\begin{lemma}\label{conj}
For odd $d$, $gM_d^{(0)}g^{-1}=M^\vartheta_d$. In particular $g\Gamma_0(2)g^{-1}=\Gt$.
\end{lemma}

\begin{proof}
\emph{$gM^{(0)}_dg^{-1}\subset M^\vartheta_d$.} Let $\mu=(\begin{smallmatrix}a&b\\2c'&d_1\end{smallmatrix})\in M^{(0)}_d$; then $ad_1$ is odd. A direct computation gives
\[
g\mu g^{-1}=\begin{pmatrix}a+c'&2b+d_1-a-c'\\c'&d_1-c'\end{pmatrix},
\]
which is integral and is $\equiv I$ for $c'$ even and $\equiv(\begin{smallmatrix}0&1\\1&0\end{smallmatrix})$ for $c'$ odd.

\emph{$g^{-1}M^\vartheta_dg\subset M^{(0)}_d$.} For $\nu=(\begin{smallmatrix}\alpha&\beta\\\gamma&\delta\end{smallmatrix})\in M_d^\vartheta$,
\[
g^{-1}\nu g=\begin{pmatrix}\alpha-\gamma&(\alpha-\gamma+\beta-\delta)/2\\2\gamma&\gamma+\delta\end{pmatrix},
\]
and $\alpha-\gamma+\beta-\delta$ is even in both residue classes of $\nu$.
\end{proof}

\subsection{Spectral theory of automorphic functions}\label{sec:spec}
The Laplacian $\Delta$ acts on the Hilbert space $L^2(G\backslash\Hh)$ with the inner product $\langle F_1,F_2\rangle=\int_{G\backslash\Hh}F_1\overline{F_2}\dd\mu$. It is a non-negative self-adjoint operator, and its spectral decomposition is the analogue, for $G\backslash\Hh$, of the Fourier expansion of periodic functions \cite[Ch.~4 and~7]{Iw}. The eigenfunctions are of three kinds.

\emph{The constant function} (eigenvalue $0$).

\emph{Maass cusp forms} (the discrete spectrum). These are eigenfunctions $\Delta u=(1/4+t^2)u$ in $L^2(G\backslash\Hh)$ which decay in all cusps. For $G=\Gamma$ such a $u$ has a Fourier expansion
\begin{equation}\label{fourier}
u(z)=\sqrt y\sum_{m\ne0}\rho(m)K_{it}(2\pi|m|y)e(mx),
\end{equation}
where $K_{it}$ is the $K$-Bessel function and $e(x)=e^{2\pi ix}$; for $G=\Gt$ the same holds for $u(2z)$, since the cusp $\infty$ has width $2$. The form is even or odd according as $u\circ R=u$ or $u\circ R=-u$, i.e.\ $\rho(-m)=\pm\rho(m)$. The parameter $t$ is positive, or, for at most finitely many exceptional eigenvalues $1/4+t^2<1/4$, lies in $i(0,1/2)$. For $\Gamma$ there are no exceptional eigenvalues; in fact $1/4+t_1^2\ge3\pi^2/2$ by Roelcke's bound \cite[Thm.~11.4]{Iw}, and numerically $1/4+t_1^2=91.14\ldots$, i.e.\ $t_1=9.53\ldots$ \cite[p.~167]{Iw}.

\emph{Eisenstein series} (the continuous spectrum). For $\Gamma$ the Eisenstein series is
\[
E(z,s)=\sum_{\gamma\in\Gamma_\infty\backslash\Gamma}\Im(\gamma z)^s=\frac{1}{2}\sum_{(c,d)=1}\frac{y^s}{|cz+d|^{2s}}\qquad(\Re s>1).
\]
Here $\gamma$ runs over the cosets of $\Gamma_\infty=\langle T\rangle$, which correspond to the coprime pairs $(c,d)$ modulo $\pm1$ (Subsection~\ref{sec:quot}). The series is the average of the function $y^s$, which is $\Gamma_\infty$-invariant and satisfies $\Delta y^s=s(1-s)y^s$. Hence $E(\cdot,s)$ is automorphic, with $\Delta E(\cdot,s)=s(1-s)E(\cdot,s)$.

The series converges only for $\Re s>1$. However, $E(z,s)$ has a meromorphic continuation in $s$ to the whole complex plane (Fact~\ref{F:E}), and the continued function is still automorphic with eigenvalue $s(1-s)$. This continuation is essential, for the following reason. The continuous spectrum of $\Delta$ on $L^2(\Gamma\backslash\Hh)$ is the half-line $[1/4,\infty)$. An eigenvalue $\lambda\ge1/4$ can be written as
\[
\lambda=\tfrac{1}{4}+t^2=s(1-s)\qquad\text{with}\qquad s=\tfrac{1}{2}+it,\quad t\in\R .
\]
So the Eisenstein series which occur in the spectral decomposition are the functions $E(\cdot,1/2+it)$, $t\in\R$. The line $\Re s=1/2$ lies outside the region of convergence, and there $E$ is defined by the continuation. These functions are not square-integrable, since they grow like $\sqrt y$ in the cusp. They play the same role as the exponentials $e^{itx}$ in the Fourier integral on $\R$: a general automorphic function is a \emph{sum} over the cusp forms plus an \emph{integral} over $t$ of the Eisenstein series $E(\cdot,1/2+it)$.

For $\Gt$ there is one Eisenstein series $E^\vartheta_{\mathfrak a}$ for each cusp $\mathfrak a\in\{\infty,1\}$; they are described in Lemma~\ref{lem:Etheta}.

\begin{fact}[Eisenstein series for $\Gamma$; {\cite[\S3.4, (3.20)--(3.29), and (8.33)]{Iw}}]\label{F:E}
$E(z,s)$ has the Fourier expansion
\[
E(z,s)=y^s+\varphi(s)y^{1-s}+\frac{2\pi^s\sqrt y}{\Gamma(s)\zeta(2s)}\sum_{m\ne0}|m|^{s-1/2}\sigma_{1-2s}(|m|)K_{s-1/2}(2\pi|m|y)e(mx),
\]
with $\varphi(s)=\sqrt\pi\,\Gamma(s-1/2)\zeta(2s-1)/(\Gamma(s)\zeta(2s))$. It satisfies $E(z,s)=\varphi(s)E(z,1-s)$, is holomorphic in $\Re s\ge1/2$ except for a simple pole at $s=1$ with residue $3/\pi$, and $\overline{E(z,s)}=E(z,\bar s)$. At $z=i$,
\[
E(i,s)=\frac{1}{2}\sum_{(c,d)=1}(c^2+d^2)^{-s}=\frac{2\zeta(s)L(s,\chi_{-4})}{\zeta(2s)},
\]
since $\sum_{(m,n)\ne(0,0)}(m^2+n^2)^{-s}=4\zeta(s)L(s,\chi_{-4})$. Hence, by the convexity bounds for $\zeta$ and $L(s,\chi_{-4})$ and Fact~\ref{F:zeta}, $E(i,s)$ has polynomial growth in $|\Im s|$ for $1/2\le\Re s\le2$, $|s-1|\ge1/4$.
\end{fact}

\emph{Point-pair invariants and the Selberg transform.} A function $K(z,w)=K(d(z,w))$ on $\Hh\times\Hh$ which depends only on the hyperbolic distance is called a \emph{point-pair invariant}; then $K(gz,gw)=K(z,w)$ for every isometry $g$. We always assume that $K$ is smooth and compactly supported in $d$. Consider the integral operator
\[
(L_Kf)(z)=\int_\Hh K\big(d(z,w)\big)f(w)\dd\mu(w).
\]
It replaces $f(z)$ by a weighted average of $f$ over a ball around $z$, with the same weight at every point $z$. Selberg's basic observation is that every eigenfunction of the Laplacian is also an eigenfunction of $L_K$, with an eigenvalue which depends only on the Laplace eigenvalue \cite[Thm.~1.14]{Iw}:
\begin{equation}\label{selberg}
\Delta f=(\tfrac{1}{4}+t^2)f\ \text{ on }\Hh\qquad\Longrightarrow\qquad L_Kf=h(t)\,f .
\end{equation}
The function $h$ depends only on $K$, not on $f$. It is called the \emph{Selberg transform} (or Selberg--Harish-Chandra transform) of $K$.

The reason behind \eqref{selberg} is the following. Fix $z$ and replace $f$ by its average over all rotations around $z$. This changes neither $f(z)$ nor $(L_Kf)(z)$, because $K(d(z,\cdot))$ is invariant under these rotations. The averaged function is still an eigenfunction of $\Delta$, and it depends only on the distance to $z$. Such a function satisfies an ordinary differential equation in the distance, whose solution that is regular at distance $0$ is unique up to a factor. So the average equals $f(z)\,\omega_t(d(z,w))$ with a fixed function $\omega_t$, $\omega_t(0)=1$, and $(L_Kf)(z)=f(z)\int_\Hh K(d(z,w))\omega_t(d(z,w))\dd\mu(w)$. The last integral does not depend on $z$ or on $f$; it is $h(t)$.

Since \eqref{selberg} holds for \emph{every} eigenfunction, $h$ can be computed from any convenient one. For instance, $f(w)=(\Im w)^{1/2+it}$ and $z=i$ give
\[
h(t)=\int_\Hh K\big(d(i,w)\big)(\Im w)^{1/2+it}\dd\mu(w).
\]
Thus $h$ is an entire function of $t$, and it is even, since it depends only on the eigenvalue $1/4+t^2$. The constant function has eigenvalue $0$, which corresponds to $t=i/2$; so
\[
h(i/2)=\int_\Hh K\big(d(i,w)\big)\dd\mu(w)
\]
is the total mass of $K$. In Subsection~\ref{sec:kernel} we shall compute $h$ for our kernel from another eigenfunction, one which is adapted to a geodesic.

\emph{The automorphic kernel.} If $K$ is averaged over the group,
\[
K_G(z,w)=\sum_{\gamma\in G}K(\gamma z,w),
\]
the result is $G$-invariant in $z$ and in $w$. Its expansion into eigenfunctions is the pre-trace formula (Fact~\ref{F:pretrace}). The coefficients of this expansion follow at once from \eqref{selberg}: for a cusp form $u$ with $\Delta u=(1/4+t^2)u$, unfolding the sum over $G$ gives
\[
\int_{G\backslash\Hh}K_G(z,w)u(w)\dd\mu(w)=\int_\Hh K(z,w)u(w)\dd\mu(w)=h(t)\,u(z).
\]
So in the expansion of $K_G(z,\cdot)$, each eigenfunction $u$ appears with the coefficient $h(t)u(z)$: its value at the point $z$, times one number which depends only on the eigenvalue. This is the reason why point-pair invariants are so useful.

\begin{fact}[Pre-trace formula; {\cite[Thms.~1.14, 1.16, 7.4]{Iw}}]\label{F:pretrace}
Let $K(z,w)=K(d(z,w))$ be smooth and compactly supported in $d$, with Selberg transform $h$. (In \cite{Iw} the kernel is written as a function $k(u)$ of $u=\sinh^2(d/2)$; the kernel of Lemma~\ref{kernel} is a smooth function of $\cosh d=1+2u$, so $k\in C_0^\infty(\R_{\ge0})$.)
\begin{enumerate}[label=(\alph*)]
\item \eqref{selberg} holds: for every $F$ with $\Delta F=(1/4+t^2)F$, $\int_\Hh K(z,w)F(w)\dd\mu(w)=h(t)F(z)$.
\item The automorphic kernel $K_G(z,w)=\sum_{\gamma\in G}K(\gamma z,w)$ has the expansion
\[
K_G(z,w)=\sum_jh(t_j)u_j(z)u_j(w)+\frac{h(i/2)}{\vol_G}+\frac{1}{4\pi}\sum_{\mathfrak a}\int_\R h(t)E^G_{\mathfrak a}(z,\tfrac{1}{2}+it)\overline{E^G_{\mathfrak a}(w,\tfrac{1}{2}+it)}\dd t,
\]
which converges absolutely and uniformly on compacta. Here $\{u_j\}$ is an orthonormal basis of real-valued Maass cusp forms for $G$, with $\Delta u_j=(1/4+t_j^2)u_j$. The second term is the contribution of the normalised constant function $\vol_G^{-1/2}$, whose eigenvalue $0$ corresponds to $t=i/2$. The sum over $\mathfrak a$ runs over the cusps of $G$. The residual spectrum consists of the constants (for $\Gt$ by Lemma~\ref{lem:Etheta}(i)).
\end{enumerate}
\end{fact}

The three terms in (b) are the three kinds of eigenfunctions: cusp forms, the constant, and Eisenstein series.

The last general tool is a bound for the average size of the eigenfunctions at a point.

\begin{fact}[Local Weyl law; {\cite[Prop.~7.2]{Iw}}]\label{F:weyl}
For $T\ge1$ and $z=x+iy$ with $y\ge1/2$,
\[
\sum_{|t_j|\le T}|u_j(z)|^2+\frac{1}{4\pi}\sum_{\mathfrak a}\int_{-T}^T|E^G_{\mathfrak a}(z,\tfrac{1}{2}+it)|^2\dd t\ll T^2+Ty .
\]
The same bound holds for $\int_{-T}^T|E(z,1/2+it)|^2\dd t$. Indeed, \cite[Prop.~7.2]{Iw} gives the bound $T^2+Ty_G(z)$ with the invariant height $y_G(z)=\max_{\mathfrak a}\max_{\gamma\in G}\Im(\sigma_{\mathfrak a}^{-1}\gamma z)$. Each $\sigma_{\mathfrak a}^{-1}\gamma$ is an element of $\PSL$, followed by $z\mapsto z/2$ if $\mathfrak a=\infty$ and $G=\Gt$; hence $y_G(z)\le\max_{\gamma\in\PSL}\Im(\gamma z)\le\max(y,1/y)\le4y$ for $y\ge1/2$.
\end{fact}

\subsection{Hecke operators}\label{sec:hecke}
Recall that $M_d$ is the set of integral $2\times2$ matrices of determinant $d$, modulo $\pm1$. The group $\Gamma$ acts on $M_d$ by multiplication from the left, and there are finitely many orbits. For a $\Gamma$-invariant function $F$ and $d\ge1$, the classical Hecke operator is
\[
T_dF(z)=\sum_{\mu\in\Gamma\backslash M_d}F(\mu z)=\sum_{ad'=d}\ \sum_{b\bmod d'}F\Big(\frac{az+b}{d'}\Big).
\]
The second expression uses the upper triangular matrices $(\begin{smallmatrix}a&b\\0&d'\end{smallmatrix})$ as representatives of the orbits. In words: $T_dF(z)$ is the sum of the values of $F$ at the $\sigma(d)$ points $\mu z$. For a prime $p$, for example,
\[
T_pF(z)=F(pz)+\sum_{b\bmod p}F\Big(\frac{z+b}{p}\Big).
\]
The Hecke operators map $\Gamma$-invariant functions to $\Gamma$-invariant functions. They commute with the Laplacian (each $\mu$ acts on $\Hh$ as an isometry) and with each other. Therefore the cusp forms can be chosen as eigenfunctions of all Hecke operators at once.

For the theta group we need a variant. Let $B_\Gamma(d')$ be a complete residue system modulo $d'$, and let $B_\vartheta(d')=\{0,2,\dots,2d'-2\}$ be the set of even residues modulo $2d'$. For $G$-invariant $F$ and $d$ coprime to $w_G$ put
\[
(T_d^GF)(z)=\sum_{ad'=d}\ \sum_{b\in B_G(d')}F\Big(\frac{az+b}{d'}\Big),\qquad T_d=T_d^\Gamma .
\]

\begin{lemma}\label{hecke}
Let $d$ be coprime to $w_G$.
\begin{enumerate}[label=(\roman*)]
\item The matrices $(\begin{smallmatrix}a&b\\0&d'\end{smallmatrix})$, $ad'=d$, $b\in B_G(d')$, form a system of representatives of $G\backslash M^G_d$, and $GM_d^GG=M^G_d$.
\item $T^G_d$ maps $G$-invariant functions to $G$-invariant functions, $T_d^G1=\sigma(d)$, and $T^\vartheta_dF=T_dF$ for $\Gamma$-invariant $F$.
\item $T^\vartheta_dF=\big(T^{(0)}_d(F\circ g)\big)\circ g^{-1}$, where $T_d^{(0)}F(z)=\sum_{\nu\in\Gamma_0(2)\backslash M^{(0)}_d}F(\nu z)$ is the Hecke operator of $\Gamma_0(2)$.
\item $T^G_d$ commutes with $R$, and $T^G_dF(w)=\sum_{\mu\in M_d^G}K(d(z_0,\mu w))$ if $F(w)=\sum_{\gamma\in G}K(d(\gamma z_0,w))$.
\end{enumerate}
\end{lemma}

\begin{proof}
\emph{(i), $G=\Gamma$.} This is classical \cite[\S8.5, (8.30)]{Iw}.

\emph{(i), $G=\Gt$.} For $\mu\in M_d$ ($d$ odd) we have $\gamma\mu\in M^\vartheta_d$ if and only if $\gamma\equiv\nu\bar\mu^{-1}\bmod 2$ with $\nu\in\{I,(\begin{smallmatrix}0&1\\1&0\end{smallmatrix})\}$. Hence $\Gamma\mu\cap M^\vartheta_d$ is a single $\Gt$-coset $\Gt\gamma_\mu\mu$. For $\mu=(\begin{smallmatrix}a&b\\0&d'\end{smallmatrix})$ we have $\mu\equiv T^b$, and we may take $\gamma_\mu=I$ if $b$ is even and $\gamma_\mu=T^{-1}$ if $b$ is odd. Then $T^{-1}\mu=(\begin{smallmatrix}a&b-d'\\0&d'\end{smallmatrix})$ with $b-d'$ even. Thus each residue class modulo $d'$ is represented exactly once by an even $b$ modulo $2d'$. Finally $GM_d^GG=M_d^G$ because $\{I,(\begin{smallmatrix}0&1\\1&0\end{smallmatrix})\}$ is a group.

\emph{(ii).} By (i), $T^G_dF(z)=\sum_{\mu\in G\backslash M_d^G}F(\mu z)$, and right multiplication by $\gamma\in G$ permutes $G\backslash M_d^G$. There are $\sigma(d)$ cosets. For $\Gamma$-invariant $F$, the sum over $b\in B_\vartheta(d')$ equals the sum over $b\bmod d'$, because $F$ has period $1$.

\emph{(iii).} By Lemma~\ref{conj}, $\nu\mapsto g\nu g^{-1}$ maps $\Gamma_0(2)\backslash M^{(0)}_d$ bijectively onto $\Gt\backslash M^\vartheta_d$, and $F(g\nu z)=F((g\nu g^{-1})gz)$.

\emph{(iv).} $R\mu R^{-1}=(\begin{smallmatrix}a&-b\\0&d'\end{smallmatrix})$, and $-b$ runs over $B_G(d')$ modulo $d'$ (resp.\ $2d'$) as $b$ does. The formula follows from the bijection $G\times(G\backslash M^G_d)\to M^G_d$, $(\gamma,\mu)\mapsto\mu\gamma$, and $d(\gamma z_0,\mu w)=d(z_0,\gamma^{-1}\mu w)$.
\end{proof}

\begin{fact}[Hecke eigenforms and the Kim--Sarnak bound]\label{F:hecke}
The space of Maass cusp forms on $G$ has an orthonormal basis $\{u_j\}$ of real-valued functions with the following properties:
\begin{itemize}
\item $\Delta u_j=(1/4+t_j^2)u_j$, where $t_j>0$ or $t_j\in i(0,1/2)$;
\item $u_j\circ R=\pm u_j$;
\item $T^G_du_j=\sqrt d\,\lambda_j(d)u_j$ for all $d$ coprime to $w_G$, with $\lambda_j(d)$ real;
\item $|\lambda_j(d)|\le\tau(d)d^{\theta}$, with $\theta=7/64$.
\end{itemize}
\emph{References.} For $\Gamma$ and for $\Gamma_0(2)$ (with $d$ odd), the $T_d$ are self-adjoint, commute with $\Delta$ and with each other, and satisfy the Hecke relations \cite[\S8.5, (8.31)--(8.32)]{Iw}; they commute with $R$ by Lemma~\ref{hecke}(iv). Lemma~\ref{hecke}(iii) transfers this to $\Gt$, since $F\mapsto F\circ g$ is unitary from $L^2(\Gt\backslash\Hh)$ to $L^2(\Gamma_0(2)\backslash\Hh)$. For a prime $p\nmid w_G$, $\lambda_j(p)$ is the normalised Hecke eigenvalue at $p$ of a cuspidal automorphic representation of $\mathrm{GL}_2(\mathbb A_\Q)$ unramified at $p$. Hence $|\lambda_j(p)|\le p^\theta+p^{-\theta}$ by Kim--Sarnak \cite[Appendix~2]{Kim}. The Hecke relations then give $|\lambda_j(p^k)|\le(k+1)p^{k\theta}$.
\end{fact}

The Ramanujan conjecture asserts $|\lambda_j(p)|\le2$, i.e.\ $\theta=0$. The Eisenstein series are Hecke eigenfunctions as well: $T_dE(\cdot,s)=\eta_d(s)E(\cdot,s)$ \cite[(8.33)]{Iw}, with
\[
\eta_d(s)=\sum_{ab=d}a^sb^{1-s}.
\]
Note that $\eta_d(1)=\eta_d(0)=\sigma(d)$, and that $|\eta_d(1/2+it)|\le\tau(d)\sqrt d$.

\subsection{\texorpdfstring{$L$}{L}-functions and periods}\label{sec:Lfun}
Let $u_j$ be a Hecke eigenform on $\Gamma$ with Fourier coefficients $\rho_j(m)$ as in \eqref{fourier}. Then $\rho_j(m)=\rho_j(1)\lambda_j(m)$ for $m\ge1$, and its $L$-function and its twist by a real character $\chi$ are
\[
L(s,u_j)=\sum_{m\ge1}\frac{\lambda_j(m)}{m^s},\qquad L(s,u_j\otimes\chi)=\sum_{m\ge1}\frac{\lambda_j(m)\chi(m)}{m^s}.
\]
They have analytic continuation and functional equations, and the central values $L(1/2,u_j)$ and $L(1/2,u_j\otimes\chi_{D_0})$ are real. We also need the symmetric square $L$-function $L(s,\mathrm{sym}^2u_j)$, and the ``harmonic weight''
\[
\alpha_j=\frac{|\rho_j(1)|^2}{\cosh(\pi t_j)},
\]
which is, up to a constant, $1/L(1,\mathrm{sym}^2u_j)$ (Fact~\ref{F:L}(c)).

\emph{Periods.} Integrals of automorphic forms over the natural cycles of $\Gamma\backslash\Hh$ are called periods. The name comes from algebraic geometry, where periods are integrals of differential forms over closed cycles, such as the periods $\oint\dd x/y$ of an elliptic curve. Two kinds of cycles occur here, and in both cases the periods turn out to be central $L$-values; this is what makes them computable.
\begin{itemize}
\item \emph{The geodesic from $0$ to $\infty$.} For even $u_j$, Hecke's integral $\int_0^\infty u_j(iy)\dd y/y$ is a multiple of $L(1/2,u_j)$ (Lemma~\ref{mellin}). This is the square case.
\item \emph{Closed geodesics.} The traces $\Tr_{D_0}(u_j)$ are governed by $L(1/2,u_j)L(1/2,u_j\otimes\chi_{D_0})$. This goes back to Katok--Sarnak \cite{KS93}, in the spirit of Waldspurger's theorem; the precise identity is Fact~\ref{F:DIT}. This is the non-square case.
\end{itemize}

Recall from \eqref{trdef} that $\Tr_{D_0}(u_j)=\sum_{[q]}\int_{C_q}u_j\dd s$ is the sum of the integrals of $u_j$ over the closed geodesics of discriminant $D_0$. Such integrals are also called cycle integrals.

\begin{fact}[Cycle integrals of cusp forms]\label{F:DIT}
Let $D_0>1$ be fundamental. For even $u_j$ on $\Gamma$,
\[
|\Tr_{D_0}(u_j)|^2=\frac{|\rho_j(1)|^2}{4\pi}\sqrt{D_0}\,\big|\Gamma(\tfrac{1}{4}+\tfrac{it_j}{2})\big|^4\,L(\tfrac{1}{2},u_j)\,L(\tfrac{1}{2},u_j\otimes\chi_{D_0}),
\]
and both central values are non-negative.

\emph{Reference and derivation.} Let $\phi=2\sqrt y\sum_{m\ne0}\lambda_j(m)K_{it_j}(2\pi|m|y)e(mx)$ be the Hecke-normalised form of \cite[(5.7)]{DIT}. Then $u_j=\rho_j(1)\phi/2$, so $\langle\phi,\phi\rangle=4|\rho_j(1)|^{-2}$; in \cite{DIT} the inner product is the plain integral over a fundamental domain of $\Gamma$, as here.
\begin{itemize}
\item \cite[Theorem~4]{DIT} with $d'=1$ and $d=D_0$ (trivial genus character; this case occurs in \cite[\S5]{DIT}, e.g.\ ``$D=12=(1)(12)$'') gives
\[
12\sqrt\pi\,D_0^{3/4}\,b(1)b(D_0)=\langle\phi,\phi\rangle^{-1}\Tr_{D_0}(\phi).
\]
Here $b(d)$ are the Fourier coefficients of a weight-$1/2$ Maass form attached to $\phi$; the cycle integrals in \cite{DIT} run over curves $\partial F_A$ whose images in $\Gamma\backslash\Hh$ are the closed geodesics $C_A$, and the narrow ideal classes $A$ correspond to $\Gamma\backslash\Qf(D_0)$.
\item The Baruch--Mao formula \cite{BM}, in the form \cite[(5.17)]{DIT}, gives, for $d\in\{1,D_0\}$,
\[
12\pi d\,|b(d)|^2=\langle\phi,\phi\rangle^{-1}\big|\Gamma(\tfrac14+\tfrac{it_j}2)\big|^2L(\tfrac12,\phi\otimes\chi_d).
\]
\end{itemize}
Eliminating $b(1)b(D_0)$ gives
\[
|\Tr_{D_0}(\phi)|^2=\frac{\sqrt{D_0}}\pi\,\big|\Gamma(\tfrac14+\tfrac{it_j}2)\big|^4L(\tfrac12,u_j)L(\tfrac12,u_j\otimes\chi_{D_0}),
\]
and multiplying by $|\rho_j(1)|^2/4$ gives the identity. Non-negativity: the left-hand side of \cite[(5.17)]{DIT} is $\ge0$ and the Gamma factor is $>0$. Only the bound
\[
|\Tr_{D_0}(u_j)|^2\ll\sqrt{D_0}\,\alpha_j(1+t_j)^{-1}L(\tfrac{1}{2},u_j)L(\tfrac{1}{2},u_j\otimes\chi_{D_0})
\]
is used below; it follows from the identity by Stirling's formula, $\cosh(\pi t)|\Gamma(1/4+it/2)|^4\ll(1+|t|)^{-1}$. For $D_0=1$ the identity becomes the formula for $|\Pp_j|^2$ in Lemma~\ref{mellin}, which is a consistency check. Both identities were also checked numerically for one Maass form (Subsection~\ref{sec:numchecks}).
\end{fact}

\begin{fact}[Cycle integrals of Eisenstein series; Hecke, see {\cite{Za}}]\label{F:Eis}
Let $D_0>1$ be fundamental. For $\Re s>1$, and by continuation for all $s$,
\[
\Tr_{D_0}\big(E(\cdot,s)\big)=2\,\frac{\sqrt\pi\,\Gamma(s/2)}{\Gamma((s+1)/2)}\Big(\frac{D_0}{4}\Big)^{s/2}\frac{\zeta(s)L(s,\chi_{D_0})}{\zeta(2s)} .
\]
\emph{Derivation and checks.} Unfolding $E$ gives
\[
\int_{C_q}E(\cdot,s)\dd s=\frac{\sqrt\pi\,\Gamma(s/2)}{\Gamma((s+1)/2)}\Big(\frac{\sqrt{D_0}}2\Big)^s\sum_v|q(v)|^{-s},
\]
with $v$ running over primitive vectors modulo $\pm1$ and $\Gamma_q$ (on the left, $\dd s$ is the arc length). Summing over classes gives the Dirichlet series of the number of representations, which is $2\zeta(s)L(s,\chi_{D_0})/\zeta(2s)$ \cite{Za}.
\begin{itemize}
\item The residue at $s=1$ is $(6/\pi)\sqrt{D_0}L(1,\chi_{D_0})=(3/\pi)\Tr_{D_0}(1)$, as it must be.
\item We verified the identity numerically at $s=2$ for $18$ fundamental discriminants $D_0\le65$, odd and even, to at least $10$ digits, by integrating the Fourier expansion of $E$ along the geodesics (Subsection~\ref{sec:numchecks}).
\end{itemize}
For $D_0=1$ the analogue is the regularised period $\Pp_E(s)$ in \eqref{Jdef}, with $\zeta(s)^2$ in place of $\zeta(s)L(s,\chi_{D_0})$.
\end{fact}

To bound sums of these periods over the spectrum we need moments of $L$-functions, a large sieve inequality for Fourier coefficients, and classical facts about $\zeta$.

\begin{fact}[Inputs on $L$-functions]\label{F:L}
Let $D_0>1$ be a fundamental discriminant, and let $L(s,\mathrm{sym}^2u_j)$ be the symmetric square $L$-function of a Hecke eigenform $u_j$ on $\Gamma$ (the adjoint $L$-function in \cite{AW}).
\begin{enumerate}[label=(\alph*)]
\item \emph{Cubic moment, Weyl bound} (Andersen--Wu \cite[Theorem~4.1]{AW}, with $N=1$ and $q=D_0$; Petrow--Young \cite[Theorem~1.1]{PY}): for $T\ge1$,
\[
\sum_{T\le t_j<T+1}\frac{L(1/2,u_j\otimes\chi_{D_0})^3}{L(1,\mathrm{sym}^2u_j)}\ll_\varepsilon(D_0T)^{1+\varepsilon},\qquad |L(\tfrac{1}{2}+it,\chi_{D_0})|\ll_\varepsilon(D_0(1+|t|))^{1/6+\varepsilon}.
\]
The sum runs over all Hecke--Maass cusp forms for $\Gamma$, even and odd. All its terms are non-negative: $L(1,\mathrm{sym}^2u_j)>0$ by (c), and $L(1/2,u_j\otimes\chi_{D_0})\ge0$ by Guo's theorem \cite{Gu}, since the automorphic representation attached to $u_j\otimes\chi_{D_0}$ has trivial central character. In \cite{AW} the theorem is stated for $T\gg1$, i.e.\ for $T\ge T_0$ with an absolute constant $T_0$; for the finitely many $u_j$ with $t_j<T_0+1$, \cite[Corollary~4.2]{AW} gives $L(1/2,u_j\otimes\chi_{D_0})\ll_\varepsilon D_0^{1/3+\varepsilon}$, so the first bound holds for all $T\ge1$. For odd $D_0$ both bounds are due to Young \cite[Theorem~1.1, (1.7), (1.8)]{Yo}.
\item \emph{Second moment} (Motohashi \cite{Mo}): $\sum_{t_j\le T}\alpha_jL(1/2,u_j)^2\ll T^2\log T$.
\item \emph{Rankin--Selberg}: $\alpha_j=c/L(1,\mathrm{sym}^2u_j)$ with an absolute constant $c>0$. This is the remark after \cite[(8.46)]{Iw}, p.~120, with $N=1$; it follows from the residue (8.12) of the Rankin--Selberg integral (8.10) in \cite{Iw}. In Iwaniec's normalisation $|\nu_j(1)|^2=4\pi|\rho^{\mathrm{Iw}}_j(1)|^2/\cosh(\pi t_j)$ and $u_j=\sum\rho^{\mathrm{Iw}}_j(n)\,2\sqrt{|n|y}\,K_{it_j}(2\pi|n|y)e(nx)$, so $\rho_j(1)=2\rho^{\mathrm{Iw}}_j(1)$ and $\alpha_j=|\nu_j(1)|^2/\pi$. (The computation gives $c=2$; we checked this numerically for the first even form, to the accuracy $1.2\%$ of a truncated Euler product.)
\item \emph{Deshouillers--Iwaniec}: Fact~\ref{F:LS} with $q=1$, $N=1/2$ and the single coefficient $m=1$ gives $\sum_{t_j\le T}\alpha_j\ll T^2$.
\end{enumerate}
\end{fact}

\begin{fact}[Large sieve; {\cite[Thm.~2]{DI}}, in the form {\cite[Lemma~3.9]{Pa}}, which allows $N\ge1/2$]\label{F:LS}
Let $q\in\{1,2\}$ and let $\{\phi_j\}$ be an orthonormal basis of Maass cusp forms for $\Gamma_0(q)$, with $\Delta\phi_j=(1/4+t_j^2)\phi_j$ and $\phi_j(z)=\sqrt y\sum_{m\ne0}\rho_j(m)K_{it_j}(2\pi|m|y)e(mx)$. Then for $K\ge1$, $N\ge1/2$ and complex $a_m$,
\[
\sum_{|t_j|\le K}\frac{1}{\cosh(\pi t_j)}\Big|\sum_{N<m\le2N}a_m\rho_j(m)\Big|^2\ll_\varepsilon(K^2+N^{1+\varepsilon})\sum_m|a_m|^2 .
\]
\end{fact}

The large sieve says that the Fourier coefficients of the different cusp forms behave, on average, like orthogonal vectors. It is the spectral analogue of the classical large sieve for Dirichlet characters.

\begin{fact}[Riemann zeta function; {\cite[Ch.~3, 7]{Ti}}]\label{F:zeta}
$\int_T^{2T}|\zeta(1/2+it)|^2\dd t\ll T\log T$, $\int_T^{2T}|\zeta(1/2+it)|^4\dd t\ll T\log^4T$, and $|\zeta(\sigma+it)|^{-1}\ll\log(2+|t|)$ uniformly for $\sigma\ge1$, $|t|\ge1$.
\end{fact}

\section{Theta functions and the route taken here}\label{sec:theta}
This work started from the question whether the counting function $N(D)$ can be studied through non-holomorphic theta functions. In this section we make the relation precise. We first show that $N(D)$ is a Fourier coefficient of Siegel's theta function (Subsection~\ref{sec:siegel}). Then we describe what the theta approach and the approach of this paper have in common (Subsection~\ref{sec:parallel}), and where they part (Subsection~\ref{sec:deviate}). Table~\ref{tab:theta} summarises the comparison.

Nothing in this section is used in the proofs. We use some notation from Section~\ref{sec:count}, which is recalled where it occurs.

\subsection{The counting function as a theta coefficient}\label{sec:siegel}
By Gauss's correspondence $x=(B,A-C,A+C)$ (Lemma~\ref{corresp}), the lattice
\[
L=\{x\in\Z^3:x_2\equiv x_3\bmod 2\}
\]
is the set of all integral binary forms $q=[A,B,C]$, and $Q(x)=B^2-4AC$ is the discriminant. To a form $q$ and a point $z=x+iy\in\Hh$ we attach the number
\[
Q_q(z)=\frac{A|z|^2+Bx+C}{y}.
\]
This function is studied in Subsection~\ref{sec:dist}. We need two of its properties here: $Q_q(i)=A+C=x_3$, and $Q_{q\circ\gamma}(z)=Q_q(\gamma z)$ for $\gamma\in\mathrm{SL}_2(\R)$ (Lemma~\ref{dist}). For $z\in\Hh$ and $x\in L$ corresponding to $q$, put
\[
Q_z(x)=Q(x)+2Q_q(z)^2 .
\]
At $z=i$ this is $Q_i(x)=x_1^2+x_2^2-x_3^2+2x_3^2=x_1^2+x_2^2+x_3^2$. For general $z$, the second property shows that $Q_z$ is obtained from $Q_i$ by a linear change of variables which preserves $Q$. So $Q_z$ is a positive definite quadratic form, a \emph{majorant} of $Q$ in the sense of Siegel. Siegel's indefinite theta function is
\[
\vartheta_L(\tau,z)=\sum_{x\in L}e\big(Q(x)u\big)\,e^{-2\pi vQ_z(x)},\qquad \tau=u+iv\in\Hh .
\]
It depends on two variables.
\begin{itemize}
\item In the variable $z$ it is $\Gamma$-invariant, because $\Gamma$ acts on $L$ preserving $Q$, and $Q_{\gamma z}$ corresponds to $Q_z$ under this action.
\item In the variable $\tau$ it transforms like a non-holomorphic modular form of half-integral weight, by Siegel's theory of indefinite theta series \cite{Si51}.
\end{itemize}
Collecting the terms with $Q(x)=D$, we see that its $D$-th Fourier coefficient in $u$ is
\begin{equation}\label{thetacoeff}
e^{-2\pi vD}\sum_{q\in\Qf(D)}e^{-4\pi vQ_q(z)^2}.
\end{equation}
At $z=i$ this is
\[
e^{-2\pi vD}\sum_{\substack{x\in L\\ Q(x)=D}}e^{-4\pi vx_3^2}=e^{-2\pi vD}\,N(D),
\]
where $N(D)$ is taken with the Gaussian weight $\Phi(t)=e^{-t^2}$ and $M=(4\pi v)^{-1/2}$. \emph{So the counting function of this paper, with a Gaussian weight, is a Fourier coefficient of Siegel's theta function at the point $z=i$.}

\subsection{What the two approaches have in common}\label{sec:parallel}
Both approaches study the same function of $z$. Let $k$ be an even function on $\R$, and for $D>0$ put
\begin{equation}\label{NDz}
N_D(z;k)=\sum_{q\in\Qf(D)}k\big(\dist(z,\ell_q)\big),
\end{equation}
where $\ell_q$ is the geodesic of the form $q$ and $\dist(z,\ell_q)$ is the signed hyperbolic distance from $z$ to $\ell_q$. This function is $\Gamma$-invariant. By Lemma~\ref{dist}, $Q_q(z)=\sqrt D\sinh\dist(z,\ell_q)$. Hence:
\begin{itemize}
\item the theta coefficient \eqref{thetacoeff} is $e^{-2\pi vD}N_D(z;k)$ with the Gaussian weight $k(r)=e^{-4\pi vD\sinh^2r}$;
\item the count of this paper is $N(D)=N_D(i;k)$ with $k(r)=\Phi(\sqrt D\sinh r/M)$ (Corollary~\ref{geosum}).
\end{itemize}
In both approaches, $N_D(z;k)$ is expanded into eigenfunctions of the Laplacian in the variable $z$. The basic fact, common to both, is that the coefficient of a cusp form factorises: it is a transform of the weight $k$ times the trace of the cusp form over the closed geodesics of discriminant $D$.

\begin{proposition}\label{prop:coeff}
Let $D>0$ be a non-square discriminant, and let $k$ be smooth, even and compactly supported. Let $u$ be a Maass cusp form for $\Gamma$ with $\Delta u=\lambda u$. Then
\[
\int_{\Gamma\backslash\Hh}N_D(z;k)\,u(z)\dd\mu(z)=\hat h(\lambda)\,\Tr_D(u),\qquad \hat h(\lambda)=\int_\R k(r)P_\lambda(r)\cosh r\dd r ,
\]
where $P_\lambda$ is the even solution of $P''+\tanh r\,P'+\lambda P=0$ with $P_\lambda(0)=1$. The same holds for the constant function $u=1$, with $\lambda=0$ and $P_0=1$:
\[
\int_{\Gamma\backslash\Hh}N_D(z;k)\dd\mu(z)=\Tr_D(1)\int_\R k(r)\cosh r\dd r .
\]
\end{proposition}

\begin{proof}
\emph{Unfolding.} Group the forms into classes. The forms in the class of $q$ are $q\circ\gamma$ with $\gamma\in\Gamma_q\backslash\Gamma$, and $\dist(z,\ell_{q\circ\gamma})=\dist(\gamma z,\ell_q)$. Hence, by \eqref{unfold},
\[
\int_{\Gamma\backslash\Hh}N_D(z;k)u(z)\dd\mu(z)=\sum_{[q]}\int_{\Gamma_q\backslash\Hh}k\big(\dist(z,\ell_q)\big)u(z)\dd\mu(z).
\]

\emph{Fermi coordinates.} Describe $z$ by its signed distance $r$ to $\ell_q$ and by the arc length $\sigma$ of the nearest point of $\ell_q$ (Subsection~\ref{sec:hyp}, transported to $\ell_q$ by an isometry). The group $\Gamma_q$ acts by $\sigma\mapsto\sigma+L$, where $L$ is the length of the closed geodesic $C_q$, and $\dd\mu=\cosh r\dd\sigma\dd r$. So the integral over $\Gamma_q\backslash\Hh$ equals
\[
\int_\R k(r)\,\Psi(r)\cosh r\dd r,\qquad \Psi(r)=\int_0^Lu\big(z(\sigma,r)\big)\dd\sigma .
\]

\emph{The differential equation.} Integrating $\Delta u=\lambda u$ over one period in $\sigma$, and using the form \eqref{fermi} of $\Delta$ in Fermi coordinates, we get $\Psi''+\tanh r\,\Psi'+\lambda\Psi=0$; the term with $\partial_\sigma^2$ integrates to zero by periodicity. The solutions of this equation are $\Psi=\Psi(0)P_\lambda+c\,P^{\mathrm{odd}}_\lambda$ with an odd function $P_\lambda^{\mathrm{odd}}$. Since $k$ is even, the odd part does not contribute, and $\Psi(0)=\int_{C_q}u\dd s$. Summing over the classes gives the claim. For $u=1$ we have $\Psi(r)=L$.
\end{proof}

In Section~\ref{sec:count} we shall see that $h(t)=\hat h(1/4+t^2)$ is the Selberg transform of the kernel $K$ which is used in this paper (Lemma~\ref{kernel}). So Proposition~\ref{prop:coeff} says that the coefficient of a cusp form $u_j$ in $N_D(\cdot;k)$ is $h(t_j)\Tr_D(u_j)$. This is exactly the coefficient of $u_j(i)$ in our exact formula (Theorem~\ref{thm:nsexact}).

\emph{The theta side.} For the Gaussian weight, the same factorisation is the content of the theorem of Katok and Sarnak \cite{KS93}. They integrate a theta function of the type $\vartheta_L(\tau,z)$ against $u_j(z)$ over $\Gamma\backslash\Hh$. The result is a function of $\tau$, a Maass form of weight $1/2$, called the Shintani--Katok--Sarnak lift of $u_j$. Its $D$-th Fourier coefficient is, by \eqref{thetacoeff}, the integral of $N_D(\cdot;k)u_j$ with the Gaussian $k$. Katok and Sarnak show that this coefficient is, up to normalising factors, the trace $\Tr_D(u_j)$ for $D>0$, and the sum of the values of $u_j$ at the Heegner points of discriminant $D$ for $D<0$.

\emph{The main term.} The second formula of Proposition~\ref{prop:coeff} gives the mean value of $N_D$:
\[
\frac{1}{\vol(\Gamma\backslash\Hh)}\int_{\Gamma\backslash\Hh}N_D(z;k)\dd\mu(z)=\frac{3}{\pi}\,\hat h(0)\,\Tr_D(1).
\]
This is exactly the main term of Theorem~\ref{thm:nsexact}. In the language of theta functions, it is the integral of $\vartheta_L$ over $z$, which is the setting of the Siegel--Weil formula. In our situation the elementary computation above is all that is needed.

\emph{Summary.} The exact formulas of this paper (Theorems~\ref{thm:nsexact} and~\ref{thm:sqexact}) are the spectral expansion in $z$ of the $D$-th Fourier coefficient of a theta function, evaluated at $z=i$, but with a general weight $k$ in place of the Gaussian.

\begin{table}[ht]
\centering\small\renewcommand{\arraystretch}{1.5}
\begin{tabularx}{\textwidth}{@{}>{\raggedright}p{0.2\textwidth}>{\raggedright}X>{\raggedright\arraybackslash}X@{}}
\toprule
 & Theta approach & This paper\\
\midrule
Basic object & $\vartheta_L(\tau,z)$: all discriminants $D$ at once & $N_D(z;k)$: one discriminant $D$\\
Weight in $x_3$ & Gaussian $e^{-4\pi vx_3^2}$, width tied to $v$ & any smooth $\Phi$ with compact support\\
Function of $z$ & $D$-th coefficient $=N_D(z;k)$, $k$ Gaussian & $N_D(z;k)$, $k(r)=\Phi(\sqrt D\sinh r/M)$\\
Coefficient of $u_j$ & $D$-th coefficient of the lift of $u_j$: a transform of the Gaussian times $\Tr_D(u_j)$ & $h(t_j)\,\Tr_D(u_j)$ (Proposition~\ref{prop:coeff})\\
Main term & integral over $z$ (Siegel--Weil) & $(3/\pi) h(i/2)\Tr_D(1)$\\
\midrule
Arithmetic input & modularity in $\tau$: $\Tr_D(u_j)$ as Fourier coefficients of forms of weight $1/2$, and bounds for these & $\Tr_{D_0}(u_j)$ as central $L$-values (Fact~\ref{F:DIT}), and moments of $L$-functions\\
From $D_0$ to $f^2D_0$ & Hecke operators $T_{p^2}$ in weight $1/2$ & Hecke relation for traces (Theorem~\ref{thm:heckerel})\\
Evaluation at $z=i$ & to be justified separately & built in: pre-trace formula for the kernel $K$\\
Square $D$ & divergent integrals with Gaussian tails & cut-off at a fixed height (Section~\ref{sec:square})\\
\bottomrule
\end{tabularx}
\caption{The theta approach and the approach of this paper. Above the middle line the two are parallel; below it they differ.}\label{tab:theta}
\end{table}

\subsection{Where the two approaches part}\label{sec:deviate}
Up to this point the two approaches are parallel. They differ in the following respects.

\emph{(i) Which variable carries the arithmetic.} The strength of the theta function is its modularity in the second variable $\tau$: the traces $\Tr_D(u_j)$, for all $D$ together, are the Fourier coefficients of one Maass form of weight $1/2$. Duke's equidistribution theorem \cite{Du} is proved in this way, by bounding these Fourier coefficients. We never use the variable $\tau$. We fix $D$ and use instead that the traces at a fundamental discriminant are central values of $L$-functions (Fact~\ref{F:DIT}). This identity is itself a consequence of the theta correspondence, through the work of Katok--Sarnak, Baruch--Mao and Duke--Imamo\={g}lu--T\'oth; so the theta correspondence enters our proof, but only through this one result, which we quote. For our problem this is the natural form of the input. The radius of the ball around $i$ is about $M/\sqrt D$, and therefore all cusp forms with $t_j$ up to about $\sqrt D/M$ contribute. What is needed is a bound \emph{on average over $u_j$}, and this is what moments of $L$-functions provide.

\emph{(ii) The weight.} The theta function forces the Gaussian weight $e^{-4\pi vx_3^2}$, which is not compactly supported and whose width is tied to $v$. Our method allows any smooth $\Phi$ with compact support. Then only the geodesics which actually pass through the ball of radius $r_0$ around $i$ contribute, and all sums and integrals in the unfolding steps are finite.

\emph{(iii) Evaluation at the point $i$.} Proposition~\ref{prop:coeff} gives the spectral coefficients of $N_D(\cdot;k)$, but we need the value of this function at the single point $z=i$. So one needs the spectral expansion to converge at a point, with control of the values $u_j(i)$. In our approach this is built in. We write the weight as an integral of a point-pair invariant $K$ along the geodesic (Lemma~\ref{kernel}); then $N_D(z;k)=\Tr_D(K_\Gamma(z,\cdot))$, and the pre-trace formula (Fact~\ref{F:pretrace}) gives an expansion which converges absolutely and uniformly. The values $u_j(i)$ are controlled on average by the local Weyl law (Fact~\ref{F:weyl}). This is the reason for working with a kernel.

\emph{(iv) Square discriminants.} For $D=n^2$ the geodesics run into the cusp, and the integrals of the constant function and of the Eisenstein series along them diverge. With a compactly supported weight the automorphic function $F$ vanishes high up in the cusp, and the regularisation of Section~\ref{sec:square} reduces to cutting off at a fixed height. With the Gaussian weight this would be more delicate.

\emph{(v) Changing the discriminant.} In the theta picture, the passage from $D_0$ to $f^2D_0$ is the action of the Hecke operators $T_{p^2}$ in weight $1/2$. Theorem~\ref{thm:heckerel} is the corresponding relation for the traces; our proof is an elementary count of sublattices and does not use half-integral weight.

In short: the method of this paper is the theta correspondence ``in kernel form'', for a single Fourier coefficient and for a weight with compact support, combined with the description of periods as $L$-values.

\section{From solutions to geodesics near \texorpdfstring{$i$}{i}}\label{sec:count}
In this section we carry out Steps~0--3 of Subsection~\ref{sec:strategy}, for all discriminants at once. The section has four parts, and each of them takes one step along the following chain:
\begin{align*}
\Rr_e(D)\ &\xrightarrow{\ \S\ref{sec:parity}\ }\ N(D),\ N_\vartheta(D)\ \xrightarrow{\ \S\ref{sec:dist}\ }\ \sum_qk\big(\dist(i,\ell_q)\big)\\
&\xrightarrow{\ \S\ref{sec:kernel}\ }\ \text{integrals of $F_G$ along geodesics}\ \xrightarrow{\ \S\ref{sec:F}\ }\ \text{spectral expansion}.
\end{align*}
\begin{itemize}
\item Subsection~\ref{sec:parity} reduces the count $\Rr_e(D)$ of all solutions to counts of binary forms.
\item Subsection~\ref{sec:dist} shows that $x_3$ measures the distance from the point $i$ to the geodesic of the form. So the count becomes a sum over geodesics, with a weight $k$ which depends on the distance to $i$.
\item Subsection~\ref{sec:kernel} constructs a kernel $K$ which turns this weight into an integral along the geodesic, and computes its Selberg transform $h$.
\item Subsection~\ref{sec:F} introduces the automorphic function $F_G$ built from $K$ and states its spectral expansion.
\end{itemize}
At the end of the section the problem is reduced to integrating eigenfunctions of the Laplacian along geodesics. This is done in Section~\ref{sec:nonsquare} for non-square $D$ and in Section~\ref{sec:square} for square $D$.

Recall that $Q(x)=x_1^2+x_2^2-x_3^2$, that $\Phi$ is a smooth even function with support in $[-2,2]$, and that $\Rr_e(D)$ is the sum of $\Phi(x_3/M)$ over the solutions of $Q(x)=D$ with $x_3\equiv e\bmod2$, while $N(D)$ is the same sum over the solutions with $x_2\equiv x_3\bmod2$.

\subsection{Binary forms and parity}\label{sec:parity}
We first identify the solutions with $x_2\equiv x_3\bmod2$ with binary quadratic forms. Recall that $\Qf(D)$ is the set of all integral forms $[A,B,C]$ of discriminant $B^2-4AC=D$.

\begin{lemma}[Gauss's correspondence]\label{corresp}
Let $D\equiv0,1\bmod4$. The map $(A,B,C)\mapsto x=(B,A-C,A+C)$ is a bijection from $\Qf(D)$ onto $\{x\in\Z^3:Q(x)=D,\ x_2\equiv x_3\bmod2\}$, with $x_3=A+C$. For $D\equiv1\bmod8$ it maps
\[
\Qf^\vartheta(D)=\{[A,B,C]\in\Qf(D):A\equiv C\equiv0\bmod 2\}
\]
onto the subset with $x_2\equiv x_3\equiv0\bmod 2$.
\end{lemma}

\begin{proof}
\emph{Bijection.} We have $Q(B,A-C,A+C)=B^2+(A-C)^2-(A+C)^2=B^2-4AC$. The inverse is $A=(x_3+x_2)/2$, $C=(x_3-x_2)/2$, $B=x_1$, which is integral if and only if $x_2\equiv x_3\bmod2$.

\emph{The case $D\equiv1\bmod 8$.} Here $B$ is odd, so $B^2\equiv1\equiv D\bmod 8$, hence $8\mid4AC$ and $AC$ is even. Therefore $A\equiv C\bmod 2$ (i.e.\ $x_2,x_3$ even) is equivalent to $A\equiv C\equiv0\bmod2$.
\end{proof}

In other words, $Q$ is the discriminant form $B^2-4AC$ written in the coordinates $(B,A-C,A+C)$. By Lemma~\ref{corresp},
\[
N(D)=\sum_{[A,B,C]\in\Qf(D)}\Phi\Big(\frac{A+C}{M}\Big),
\]
and for $D\equiv1\bmod 8$ we let $N_\vartheta(D)$ be the same sum over $\Qf^\vartheta(D)$. The solutions with $x_2\not\equiv x_3\bmod 2$ are not binary forms, but they are obtained from binary forms by interchanging $x_1$ and $x_2$. The following proposition makes this precise; Table~\ref{tab:parity} summarises it.

\begin{proposition}[Parity reduction]\label{prop:par}
Let $1\le M\le\sqrt D$.
\begin{enumerate}[label=(\alph*)]
\item If $D\equiv1\bmod 8$, then $\Rr_0(D)=2N_\vartheta(D)$ and $\Rr_1(D)=N(D)-N_\vartheta(D)$.
\item If $D\equiv5\bmod 8$, then $\Rr_1(D)=0$ and $\Rr(D)=\Rr_0(D)=2N(D)$.
\item If $4\mid D$, then $\Rr_0(D;M)=\Rr(D/4;M/2)$ and $\Rr_1(D;M)=2N(D;M)-2\Rr(D/4;M/2)$.
\end{enumerate}
\end{proposition}

\begin{table}[ht]
\centering\footnotesize\renewcommand{\arraystretch}{1.4}
\begin{tabular}{@{}llll@{}}
\toprule
$D$ & parity types of $x$ & types counted by $N(D)$ & $\Rr_e(D)$\\
\midrule
$\equiv1\bmod 8$ & $(1,0,0)$, $(0,1,0)$, $(1,1,1)$ & $(1,0,0)$, $(1,1,1)$ & $\Rr_0=2N_\vartheta$, $\Rr_1=N-N_\vartheta$\\
$\equiv5\bmod 8$ & $(1,0,0)$, $(0,1,0)$ & $(1,0,0)$ & $\Rr_0=2N$, $\Rr_1=0$\\
$\equiv0\bmod 4$ & $(0,0,0)$, $(1,0,1)$, $(0,1,1)$ & $(0,0,0)$, $(0,1,1)$ & $\Rr_0=\Rr(D/4;M/2)$, $\Rr_1=2N-2\Rr(D/4;M/2)$\\
\bottomrule
\end{tabular}
\caption{The parity reduction. A type is the vector of parities of $(x_1,x_2,x_3)$; the swap $x_1\leftrightarrow x_2$ exchanges the types in pairs, and in each pair exactly one type has $x_2\equiv x_3\bmod 2$.}\label{tab:parity}
\end{table}

\begin{proof}
\emph{Parity types.} Let $Q(x)=D$ and let $\pi(x)\in(\Z/2)^3$ be the vector of parities of $x$. Since squares are $\equiv0,1\bmod4$,
\[
Q(x)\equiv\#\{i\le2:x_i\ \mathrm{odd}\}-[x_3\ \mathrm{odd}]\bmod4 .
\]
Hence, for $D\equiv1\bmod4$, $\pi(x)\in\{(1,0,0),(0,1,0),(1,1,1)\}$, and for $D\equiv0\bmod 4$, $\pi(x)\in\{(0,0,0),(1,0,1),(0,1,1)\}$. Since odd squares are $\equiv1\bmod8$, type $(1,1,1)$ has $Q(x)\equiv1\bmod 8$. The swap $(x_1,x_2,x_3)\mapsto(x_2,x_1,x_3)$ preserves $Q$ and $x_3$, and exchanges $(1,0,0)\leftrightarrow(0,1,0)$ and $(1,0,1)\leftrightarrow(0,1,1)$. By Lemma~\ref{corresp}, $N(D)$ is the weighted count of the types with $x_2\equiv x_3$, and for $D\equiv1\bmod 8$, $N_\vartheta(D)$ is that of type $(1,0,0)$.

\emph{(a)} $\Rr_0$ counts types $(1,0,0)$ and $(0,1,0)$, which by the swap give $2N_\vartheta$. $\Rr_1$ counts type $(1,1,1)$, which is $N-N_\vartheta$.

\emph{(b)} Type $(1,1,1)$ does not occur, so $\Rr_1=0$, and $N(D)$ counts type $(1,0,0)$ only. Hence $\Rr_0=2N$.

\emph{(c)} $\Rr_0$ counts type $(0,0,0)$, i.e.\ $x=2y$ with $Q(y)=D/4$ and weight $\Phi(y_3/(M/2))$. $\Rr_1$ counts types $(1,0,1)$ and $(0,1,1)$, i.e.\ twice type $(0,1,1)$, and type $(0,1,1)$ is $N(D;M)-\Rr(D/4;M/2)$.
\end{proof}

For $D\equiv2,3\bmod4$ no solution has $x_2\equiv x_3\bmod2$; these $D$ are not treated in this paper. For $4\mid D$ the reduction can be iterated, until the discriminant is $\equiv1\bmod4$.

\subsection{The form, its geodesic and the point \texorpdfstring{$i$}{i}}\label{sec:dist}
To a real form $q=[A,B,C]$ we attach the function
\[
Q_q(z)=\frac{A|z|^2+Bx+C}{y}=\frac{q(x,1)}{y}+Ay\qquad(z=x+iy\in\Hh).
\]
It is a way of evaluating $q$ at a point of $\Hh$ instead of a pair of integers, and $Q_q(i)=A+C=x_3$. It has three properties, which are proved in Lemma~\ref{dist}.
\begin{itemize}
\item It is compatible with the group action: $Q_{q\circ\gamma}(z)=Q_q(\gamma z)$.
\item Its zero set in $\Hh$ is the geodesic $\ell_q$. Indeed, for $A\ne0$ the equation $A|z|^2+Bx+C=0$ says $|z+(B/(2A))|^2=D/(4A^2)$, which is the semicircle with centre $-B/(2A)$ and radius $\sqrt D/(2|A|)$, whose endpoints are the roots of $q(x,1)$; for $A=0$ it is the vertical line $Bx+C=0$.
\item Away from $\ell_q$ it measures the distance to $\ell_q$: $Q_q(z)=\sqrt D\sinh d_\pm(z,\ell_q)$.
\end{itemize}
In particular, $x_3=A+C$ is small exactly when $\ell_q$ passes close to $i$. The second property also explains the definition of $\ell_q$ in Subsection~\ref{sec:quot}: the roots of $q(x,1)$ are real, so they lie on the boundary of $\Hh$, and $\ell_q$ is the curve inside $\Hh$ which joins them.

\begin{lemma}[The distance formula]\label{dist}
Let $q(u,v)=(au+bv)(cu+dv)$ with $\mu=\left(\begin{smallmatrix}a&b\\c&d\end{smallmatrix}\right)\in\mathrm{GL}_2(\R)$ and $\det\mu=m>0$. Then:
\begin{enumerate}[label=(\roman*)]
\item $q$ has discriminant $m^2$, and $Q_q(z)=m\,\Re(\mu z)/\Im(\mu z)$;
\item $Q_{q\circ\gamma}(z)=Q_q(\gamma z)$ and $\ell_{q\circ\gamma}=\gamma^{-1}\ell_q$ for $\gamma\in\mathrm{SL}_2(\R)$;
\item $Q_q(z)=m\sinh d_\pm(z,\ell_q)$. Here $\ell_q=\mu^{-1}(i\R_{>0})\subset\Hh$ is the geodesic of $\Hh$ whose endpoints are the two zeros $-b/a$ and $-d/c$ of $q$ on $\R\cup\{\infty\}$, i.e.\ the roots of $q(x,1)$, with $\infty$ in place of a root when $A=0$. The signed hyperbolic distance $d_\pm(z,\ell_q)$ is positive on the side of $\ell_q$ where $\Re(\mu z)>0$.
\end{enumerate}
\end{lemma}

\begin{proof}
\emph{(i)} We have $A=ac$, $B=ad+bc$, $C=bd$, so $B^2-4AC=m^2$. Moreover $(az+b)(c\bar z+d)=A|z|^2+Bx+C+imy=|cz+d|^2\mu z$; compare real and imaginary parts.

\emph{(ii)} The form $q\circ\gamma$ corresponds to $\mu\gamma$.

\emph{(iii)} By \eqref{signeddist}, the signed distance $r$ from $w$ to $i\R_{>0}$ satisfies $\sinh r=\Re w/\Im w$. Apply the isometry $\mu$, scaled to determinant $1$. Since $\det\mu>0$, $\mu$ maps $\Hh$ to itself, so $\mu^{-1}(i\R_{>0})$ is a geodesic of $\Hh$. Its endpoints are $\mu^{-1}0=-b/a$ and $\mu^{-1}\infty=-d/c$ in $\R\cup\{\infty\}$ (read as $\infty$ when $a=0$, resp.\ $c=0$).
\end{proof}

\begin{remark}[The definite case, for comparison]\label{rem:definite}
Let $q=[A,B,C]$ be a positive definite real form of discriminant $-n<0$, and let $z_q\in\Hh$ be the root of $q(z,1)=0$ in the upper half-plane. Then
\[
Q_q(z)=\sqrt n\,\cosh d(z,z_q)\qquad(z\in\Hh).
\]
In particular $A+C=Q_q(i)=\sqrt n\cosh d(i,z_q)\ge\sqrt n$. So where an indefinite form has a geodesic and the function $\sinh$ of the distance to it, a definite form has a point (for integral forms, a Heegner point) and the function $\cosh$ of the distance to it. Small values of $A+C$ correspond to geodesics passing close to $i$ in the first case, and to Heegner points close to $i$ in the second.

\emph{Proof.} The rule $Q_{q\circ\gamma}(z)=Q_q(\gamma z)$ holds for every real form $q$ and every $\gamma=(\begin{smallmatrix}a&b\\c&d\end{smallmatrix})\in\mathrm{SL}_2(\R)$. Indeed, $A|z|^2+Bx+C$ is the value at $(z,1)$ of the Hermitian form $A|v_1|^2+B\,\Re(v_1\bar v_2)+C|v_2|^2$; we have $\gamma(z,1)^T=(cz+d)(\gamma z,1)^T$, and $\Im(\gamma z)=y/|cz+d|^2$. Moreover $z_{q\circ\gamma}=\gamma^{-1}z_q$. Choose $\gamma$ with $\gamma i=z_q$. Then $q\circ\gamma$ has the root $i$, so $q\circ\gamma=a_0(u^2+v^2)$ with $a_0>0$ and $4a_0^2=n$. By \eqref{hypdist}, $Q_{q\circ\gamma}(z)=a_0(|z|^2+1)/y=2a_0\cosh d(z,i)$. Hence $Q_q(\gamma z)=\sqrt n\cosh d(z,i)=\sqrt n\cosh d(\gamma z,z_q)$. \qed
\end{remark}

We return to indefinite forms. Every real form of positive discriminant $D$ factors over $\R$ as in Lemma~\ref{dist}, with $m=\sqrt D$. We now fix the notation
\[
\delta=\frac{M}{\sqrt D}\in(0,1],\qquad k(r)=\Phi\Big(\frac{\sinh r}{\delta}\Big),\qquad r_0=\operatorname{arcsinh}(2\delta).
\]
The function $k$ is the weight $\Phi(x_3/M)$, expressed through the distance $r$ instead of $x_3=\sqrt D\sinh r$. Since $\Phi$ has support in $[-2,2]$, $k$ is supported in $|r|\le r_0\le\operatorname{arcsinh}2$; and $k$ is even.

\begin{cor}[The count as a sum over geodesics]\label{geosum}
For every discriminant $D>0$, and for $\Qf^\vartheta(D)$ when $D\equiv1\bmod8$,
\[
N(D)=\sum_{q\in\Qf(D)}k\big(\dist(i,\ell_q)\big),\qquad N_\vartheta(D)=\sum_{q\in\Qf^\vartheta(D)}k\big(\dist(i,\ell_q)\big).
\]
\end{cor}

\begin{proof}
By Lemma~\ref{dist}(iii), $A+C=Q_q(i)=\sqrt D\sinh d_\pm(i,\ell_q)$, and $k$ is even.
\end{proof}

\upshot{$N(D)$ counts, with the weight $k$, the geodesics of discriminant $D$ which pass through the hyperbolic ball of radius $r_0$ around $i$; see Figures~\ref{fig:geod} and~\ref{fig:sqgeod}. The radius $r_0\approx2\delta=2M/\sqrt D$ shrinks as $M/\sqrt D\to0$, and the difficulty of the problem lies there.}

\subsection{The kernel and its Selberg transform}\label{sec:kernel}
The sum in Corollary~\ref{geosum} runs over geodesics, while spectral theory deals with functions of points. The bridge between the two is a point-pair invariant $K$ whose integral along any geodesic reproduces the weight $k$:
\begin{equation}\label{kernelidentity}
\int_\ell K\big(d(z,w)\big)\dd s(w)=k\big(\dist(z,\ell)\big)\qquad\text{for every geodesic $\ell$ and every $z\in\Hh$}.
\end{equation}
Here $k$ is given by the problem, while $K$ has to be constructed. We shall never need $K$ explicitly. All we need is its Selberg transform $h$ (Subsection~\ref{sec:spec}), and it turns out that $h$ can be written down directly in terms of $k$.

This subsection is organised as follows. We first state the result (Lemma~\ref{kernel}) and the two properties of $h$ which are used later (Lemma~\ref{hbounds}). The proofs follow at the end of the subsection.

\emph{The transform.} For $\lambda\in\C$ let $P_\lambda$ be the even solution of
\[
P''+\tanh r\,P'+\lambda P=0,\qquad P_\lambda(0)=1 .
\]
It is entire in $\lambda$, and $P_0=1$. (This is the equation satisfied by an eigenfunction of $\Delta$ with eigenvalue $\lambda$ which depends only on the distance $r$ to a geodesic, see \eqref{fermi}.) We put
\[
\hat h(\lambda)=\int_\R k(r)P_\lambda(r)\cosh r\dd r,\qquad h(t)=\hat h(\tfrac{1}{4}+t^2),\qquad H(s)=\hat h(s(1-s)).
\]
So $h$ and $H$ are the same function, written in the two usual variables: $H(1/2+it)=h(t)$, and $H(0)=H(1)=h(i/2)=\hat h(0)$. The function $\hat h$ is entire.

\begin{lemma}[The kernel]\label{kernel}
There is a smooth function $K$ on $[0,\infty)$ with $K(d)=0$ for $d>r_0$, such that \eqref{kernelidentity} holds for every geodesic $\ell$ and every $z\in\Hh$. Its Selberg transform is the function $h(t)=\hat h(1/4+t^2)$ defined above.
\end{lemma}

\begin{lemma}[Size of the transform]\label{hbounds}
\begin{enumerate}[label=(\roman*)]
\item For $\lambda=\mu^2$ with $\mu\in\C$,
\[
|\hat h(\lambda)|\le e^{|\Im\mu|r_0}\,\delta\,\|\Phi\|_1\qquad\text{and}\qquad \hat h(\lambda)\ll_Ae^{|\Im\mu|r_0}\,\delta\,(1+\delta^2|\lambda|)^{-A}\quad\text{for every $A\ge0$}.
\]
\item $H(0)=\delta\int_\R\Phi(x)\dd x$ and $H'(0)=-\int_\R k(r)\cosh r\log\cosh r\dd r$.
\item If $s=1/2+i\nu$ with $\nu\in\C$, then $s(1-s)=\mu^2$ with $|\Im\mu|\le|\Im\nu|$.
\end{enumerate}
\end{lemma}

Two consequences of Lemma~\ref{hbounds} are used throughout.
\begin{itemize}
\item \emph{The value at the constant function.} By (ii),
\[
h(i/2)=H(0)=\delta\int_\R\Phi(x)\dd x=\frac{\Psi_0}{\sqrt D}.
\]
This determines the main term.
\item \emph{Decay.} Cusp forms have $1/4+t_j^2>0$, so $\mu$ is real in (i). Hence
\begin{equation}\label{hbound}
|h(t_j)|\ll\delta\min\big(1,(\delta|t_j|)^{-2}\big)
\end{equation}
for all $j$, exceptional or not, and the same bound holds for $h(t)$, $t\in\R$. So $h(t)$ has size $\delta$ for $|t|\le1/\delta$ and decays beyond. Only the spectral parameters up to about $1/\delta=\sqrt D/M$ matter.
\end{itemize}
Part (iii) and the second half of (ii) are needed only in the square case, where a contour is shifted and a derivative of $H$ appears.

\medskip
We now turn to the proofs. Let $\ell_0=i\R_{>0}$, with the Fermi coordinates $w(\sigma,r)$ of Subsection~\ref{sec:hyp}: $r$ is the signed distance from $w$ to $\ell_0$, and $\sigma$ is the position along $\ell_0$. The first lemma explains where the function $P_\lambda$ comes from: averaging an eigenfunction of $\Delta$ along $\ell_0$ gives a function of the distance $r$ alone, and this function satisfies the differential equation of $P_\lambda$.

\begin{lemma}\label{average}
Let $\Delta F=\lambda F$. Assume that $F(w(\sigma,r))$ and its derivatives up to order $2$ are integrable in $\sigma$, uniformly for $r$ in compact sets. Then $\Psi(r)=\int_\R F(w(\sigma,r))\dd\sigma$ satisfies $\Psi''+\tanh r\,\Psi'+\lambda\Psi=0$.
\end{lemma}

\begin{proof}
Integrate $\Delta F=\lambda F$ over $\sigma$, using the formula \eqref{fermi} for $\Delta$ in Fermi coordinates; the term $\int\partial_\sigma^2F\dd\sigma$ vanishes.
\end{proof}

The second lemma bounds $P_\lambda$.

\begin{lemma}\label{Pbound}
If $\lambda=\mu^2$ with $\mu\in\C$, then $|P_\lambda(r)|\le e^{|\Im\mu||r|}$.
\end{lemma}

\begin{proof}
For $\mu=0$, $P\equiv1$. Otherwise, for $r\ge0$ let $\mathcal E=|P'|^2+|\mu|^2|P|^2$. By the equation,
\[
\mathcal E'=-2\tanh r|P'|^2+2\Re\big(\bar P'P(|\mu|^2-\mu^2)\big)\le2|\Im\mu|\,\mathcal E,
\]
since $||\mu|^2-\mu^2|=2|\mu||\Im\mu|$. Hence $|\mu|^2|P|^2\le\mathcal E(r)\le|\mu|^2e^{2|\Im\mu|r}$. Use evenness for $r<0$.
\end{proof}

\begin{proof}[Proof of Lemma~\ref{kernel}]
\emph{Construction of $K$.} Both sides of \eqref{kernelidentity} are invariant under isometries, so it suffices to take $\ell=\ell_0$ and $z$ on the geodesic through $i$ perpendicular to $\ell_0$. If $\dist(z,\ell_0)=r$, then
\[
\cosh d(z,ie^\sigma)=\cosh r\cosh\sigma
\]
(the hyperbolic theorem of Pythagoras, or directly from \eqref{hypdist}). So \eqref{kernelidentity} reads
\[
\int_\R K\big(\operatorname{arccosh}(\cosh r\cosh\sigma)\big)\dd\sigma=k(r).
\]
Put $X=\cosh^2r$ and $G(w)=K(\operatorname{arccosh}\sqrt w)/\sqrt w$. The substitutions $v=\cosh r\cosh\sigma$, $w=v^2$ turn this into the Abel integral equation
\[
\tilde k(X)=\int_X^\infty G(w)(w-X)^{-1/2}\dd w,\qquad \tilde k(\cosh^2r)=k(r).
\]
Since $\Phi$ is even and smooth, $\tilde k(X)=\Phi_0((X-1)/\delta^2)$ with $\Phi_0(u)=\Phi(\sqrt u)$ smooth, so $\tilde k$ is smooth with support in $[1,1+4\delta^2]$. Abel's inversion formula
\[
G(w)=-\frac{1}{\pi}\int_w^\infty\tilde k'(X)(X-w)^{-1/2}\dd X
\]
gives a smooth $G$ with the same support. Put $K(d)=\cosh d\,G(\cosh^2d)$.

\emph{The Selberg transform: idea.} By \eqref{selberg}, the transform can be read off from the action of $K$ on any eigenfunction. We take an eigenfunction $F$ and compute the integral
\[
I=\int_\Hh k\big(\dist(w,\ell_0)\big)F(w)\dd\mu(w)
\]
in two ways: once in Fermi coordinates, which produces $\hat h$, and once by writing $k$ as an integral of $K$, which produces the Selberg transform.

\emph{The Selberg transform: computation.} Let $s=1/2+it$ and $F(w)=(\Im w/|w+1|^2)^s$. This function satisfies $\Delta F=(1/4+t^2)F$, since it is $y^s$ composed with an isometry. Moreover $F(ie^\sigma)=(2\cosh\sigma)^{-s}$, and $F$ decays exponentially along the curves $r=\mathrm{const}$.
\begin{itemize}
\item \emph{First way.} In Fermi coordinates $\dd\mu=\cosh r\dd\sigma\dd r$, so $I=\int_\R k(r)\Psi(r)\cosh r\dd r$ with $\Psi(r)=\int_\R F(w(\sigma,r))\dd\sigma$. By Lemma~\ref{average}, $\Psi$ solves the equation of $P_\lambda$ with $\lambda=1/4+t^2$, so $\Psi=\Psi(0)P_\lambda$ plus an odd function. Since $k$ is even,
\[
I=\hat h(\tfrac{1}{4}+t^2)\,\Psi(0)=\hat h(\tfrac{1}{4}+t^2)\int_\R(2\cosh\sigma)^{-s}\dd\sigma .
\]
\item \emph{Second way.} By \eqref{kernelidentity}, $k(\dist(w,\ell_0))=\int_{\ell_0}K(d(w,v))\dd s(v)$. By Fubini and \eqref{selberg}, applied at the point $v$,
\[
I=\int_{\ell_0}\int_\Hh K(d(v,w))F(w)\dd\mu(w)\dd s(v)=h_K(t)\int_{\ell_0}F(v)\dd s(v)=h_K(t)\int_\R(2\cosh\sigma)^{-s}\dd\sigma ,
\]
where $h_K$ denotes the Selberg transform of $K$.
\end{itemize}
Finally $\int_\R(\cosh\sigma)^{-s}\dd\sigma=\sqrt\pi\,\Gamma(s/2)/\Gamma((s+1)/2)\ne0$. Hence $h_K(t)=\hat h(1/4+t^2)$ for real $t$, and then for all $t$, since both sides are entire.
\end{proof}

\begin{proof}[Proof of Lemma~\ref{hbounds}]
\emph{(i)} The first bound follows from Lemma~\ref{Pbound} and $\int k(r)\cosh r\dd r=\delta\int\Phi$ (substitute $\sinh r=\delta x$), applied to $|k|$. For the second, let $\mathcal Lg=(\cosh r\,g')'/\cosh r$. Then $\mathcal LP_\lambda=-\lambda P_\lambda$, and two integrations by parts give $\hat h(\lambda;k)=-\lambda^{-1}\hat h(\lambda;\mathcal Lk)$, where $\hat h(\lambda;g)$ denotes the transform with $g$ in place of $k$. In the variable $u=\sinh r$ we have $\cosh r\dd r=\dd u$ and
\[
\mathcal L=\frac{\dd}{\dd u}(1+u^2)\frac{\dd}{\dd u}.
\] Since $|u|\le2\delta\le2$ on the support,
\[
\int|\mathcal L^Ak|\cosh r\dd r\ll_A\delta^{1-2A}.
\]
Now use Lemma~\ref{Pbound} again.

\emph{(ii)} $P_0=1$, so $H(0)=\int k\cosh r\dd r=\delta\int\Phi$. Moreover $\lambda=s(1-s)$ has derivative $1$ at $s=0$, and $R_0=\partial_\lambda P_\lambda|_{\lambda=0}$ solves $(\cosh r R_0')'=-\cosh r$ with $R_0(0)=R_0'(0)=0$. Hence $R_0=-\log\cosh r$.

\emph{(iii)} $s(1-s)=1/4+\nu^2=:z$. For any $z\in\C$, $|\Im\sqrt z|^2=(1/2)(|z|-\Re z)$, and
\[
|z|-\Re z\le\tfrac{1}{4}+|\nu|^2-\tfrac{1}{4}-(\Re\nu)^2+(\Im\nu)^2=2(\Im\nu)^2 .\qedhere
\]
\end{proof}

\subsection{The automorphic function and its spectral expansion}\label{sec:F}
Let $G$ be one of the two groups $\Gamma$, $\Gt$ (Subsection~\ref{sec:Gconv}; on first reading, $G=\Gamma$). We average the kernel over $G$, with the first variable fixed at the point $i$:
\[
F_G(w)=K_G(i,w)=\sum_{\gamma\in G}K\big(d(\gamma i,w)\big).
\]
The sum is finite for each $w$, since $K$ has compact support, and $F_G$ is $G$-invariant. It is a sum of bumps, one on the ball of radius $r_0$ around each point $\gamma i$.

\emph{The count as an integral of $F_G$.} By Corollary~\ref{geosum} and \eqref{kernelidentity},
\[
N(D)=\sum_{q\in\Qf(D)}\int_{\ell_q}K\big(d(i,w)\big)\dd s(w).
\]
When the forms are grouped into classes, the sum over a class becomes the integral of $F_G$ along one geodesic, because $\dist(i,\ell_{q\circ\gamma})=\dist(\gamma i,\ell_q)$. This is the unfolding which was described in Step~3 of Subsection~\ref{sec:strategy}. It is carried out
\begin{itemize}
\item in Theorem~\ref{thm:nsexact} for non-square $D$, where the geodesics are closed and the result is $N(D)=\Tr_D(F_\Gamma)$;
\item in Proposition~\ref{reduction} for square $D$, where the geodesics are infinite.
\end{itemize}

\emph{The spectral expansion.} The pre-trace formula (Fact~\ref{F:pretrace}(b)) with $z=i$ expands $F_G$ into eigenfunctions:
\begin{equation}\label{Fexp}
F_G(w)=\frac{H(0)}{\vol_G}+\sum_jh(t_j)u_j(i)u_j(w)+\frac{1}{4\pi}\sum_{\mathfrak a}\int_\R h(t)E^G_{\mathfrak a}(i,\tfrac{1}{2}+it)\overline{E^G_{\mathfrak a}(w,\tfrac{1}{2}+it)}\dd t .
\end{equation}
The three terms are the constant function, the Maass cusp forms $u_j$ of $G$ and the Eisenstein series of $G$, one for each cusp $\mathfrak a$. The expansion converges absolutely and uniformly on compact sets.

\upshot{The count is an integral of $F_G$ along the geodesics of discriminant $D$, and $F_G$ is an explicit combination of eigenfunctions, with the coefficients $h(t)\cdot(\text{value at }i)$. The coefficient $h$ has size $\delta=M/\sqrt D$ for $|t|\le1/\delta$ and is negligible beyond, by \eqref{hbound}. What remains is to integrate each eigenfunction along the geodesics, that is, to compute and to bound the \emph{periods}. The constant function will give the main term, the cusp forms and Eisenstein series the error term.}

\section{Non-square discriminants}\label{sec:nonsquare}
In this section we prove Theorem~\ref{thm:ns} and Corollary~\ref{cor:tern}. Throughout the section, $D=f^2D_0>0$ is a non-square discriminant with fundamental discriminant $D_0$ and conductor $f$, the group is $G=\Gamma$, and $\{u_j\}$ is the basis of Hecke eigenforms for $\Gamma$ (Fact~\ref{F:hecke}).

\recall{From Section~\ref{sec:count} we use the following. The count is $N(D)=\sum_{q\in\Qf(D)}k(\dist(i,\ell_q))$ (Corollary~\ref{geosum}). The function $F=F_\Gamma$ is a sum of bumps around the points $\gamma i$, with the spectral expansion \eqref{Fexp}. Its coefficients involve the transform $h$, for which $h(i/2)=H(0)=\Psi_0/\sqrt D$ and $|h(t)|\ll\delta\min(1,(\delta|t|)^{-2})$, where $\delta=M/\sqrt D$. From Subsection~\ref{sec:quot}: $C_q$ is the closed geodesic of the form $q$, and $\Tr_D(F)=\sum_{[q]}\int_{C_q}F\dd s$ is the sum of the integrals of $F$ over the closed geodesics of discriminant $D$.}

The argument has four steps, which correspond to the subsections.
\begin{enumerate}[leftmargin=*]
\item \emph{Exact formula} (Subsection~\ref{sec:exactns}). Unfolding gives $N(D)=\Tr_D(F)$. The spectral expansion of $F$ then expresses $N(D)$ through the traces of the eigenfunctions: $\Tr_D(1)$, $\Tr_D(u_j)$ and $\Tr_D(E(\cdot,s))$.
\item \emph{Hecke relation} (Subsection~\ref{sec:heckerel}). The traces at $D=f^2D_0$ are the traces at $D_0$ times explicit factors made of Hecke eigenvalues. For the constant function this gives the main term.
\item \emph{Periods} (Subsection~\ref{sec:fundper}). At the fundamental discriminant $D_0$ the traces of cusp forms are central $L$-values, and moments of $L$-functions bound their mean square.
\item \emph{Conclusion} (Subsections~\ref{sec:nsproof} and~\ref{sec:tern}). The Cauchy--Schwarz inequality and the local Weyl law give the error term.
\end{enumerate}

\subsection{The exact formula}\label{sec:exactns}
We begin by writing the count as the trace of $F$ and inserting the spectral expansion of $F$. The result is an exact formula for $N(D)$.

\begin{theorem}[Non-square discriminants, exact formula]\label{thm:nsexact}
Let $D>0$ be a non-square discriminant and $1\le M\le\sqrt D$. Then
\[
N(D)=\frac{3}{\pi} H(0)\Tr_D(1)+\sum_{j:\ u_j\ \mathrm{even}}h(t_j)u_j(i)\Tr_D(u_j)+\frac{1}{4\pi}\int_\R h(t)E(i,\tfrac{1}{2}+it)\Tr_D\big(E(\cdot,\tfrac{1}{2}-it)\big)\dd t .
\]
Moreover $\Tr_D(u_j)=0$ for odd $u_j$.
\end{theorem}

The three terms come from the three parts of the spectrum: the constant function, the cusp forms and the Eisenstein series. Each term is a product of three factors of different origin:
\[
\underbrace{h(t)}_{\text{the weight }\Phi}\ \times\ \underbrace{\text{value of the eigenfunction at }i}_{\text{the centre of the ball}}\ \times\ \underbrace{\text{trace of the eigenfunction}}_{\text{the geodesics of discriminant }D} .
\]
The formula is exact; no estimate has been made so far.

\begin{proof}
Let $K$ be as in Lemma~\ref{kernel}, and $F=F_\Gamma$ as in Subsection~\ref{sec:F}.

\emph{Step 1: unfolding.} By Corollary~\ref{geosum}, $N(D)=\sum_{q\in\Qf(D)}k(\dist(i,\ell_q))$. We group the forms into classes. The map $\Gamma_q\gamma\mapsto q\circ\gamma$ is a bijection from $\Gamma_q\backslash\Gamma$ onto the class of $q$, and $\dist(i,\ell_{q\circ\gamma})=\dist(\gamma i,\ell_q)$, because $\ell_{q\circ\gamma}=\gamma^{-1}\ell_q$ (Lemma~\ref{dist}(ii)). We now write the weight as an integral of the kernel, $k(\dist(z,\ell))=\int_\ell K(d(z,w))\dd s(w)$ (this is \eqref{kernelidentity}). Hence
\[
N(D)=\sum_{[q]}\sum_{\gamma\in\Gamma_q\backslash\Gamma}\int_{\ell_q}K(d(\gamma i,w))\dd s(w)=\sum_{[q]}\int_{\Gamma_q\backslash\ell_q}\sum_{\gamma\in\Gamma}K(d(\gamma i,w))\dd s(w)=\Tr_D(F).
\]
In the middle step the integral over $\ell_q$ of a sum over $\Gamma_q\backslash\Gamma$ is folded into an integral over $\Gamma_q\backslash\ell_q$ of a sum over $\Gamma$, as in \eqref{unfold}. All sums are finite by Fact~\ref{F:gauss}(ii) and the compact support of $K$.

\emph{Step 2: spectral expansion.} By \eqref{Fexp} with $G=\Gamma$ and $\vol_\Gamma=\pi/3$,
\[
F(w)=\frac{3}{\pi} H(0)+\sum_jh(t_j)u_j(i)u_j(w)+\frac{1}{4\pi}\int_\R h(t)E(i,\tfrac{1}{2}+it)\overline{E(w,\tfrac{1}{2}+it)}\dd t .
\]
The expansion converges uniformly on the compact set $\bigcup_{[q]}C_q$, so it may be integrated termwise along the closed geodesics. This replaces each eigenfunction of $w$ by its trace. For the Eisenstein term use $\overline{E(w,s)}=E(w,\bar s)$.

\emph{Step 3: odd forms.} Recall that a cusp form $u$ is odd if $u\circ R=-u$, where $R:z\mapsto-\bar z$ is the reflection in the imaginary axis. The reflection $R$ maps $\ell_q$ onto $\ell_{q'}$ with $q'=[A,-B,C]$, preserves arc length, and permutes $\Gamma\backslash\Qf(D)$. Hence $\Tr_D(u\circ R)=\Tr_D(u)$, which vanishes for odd $u$.
\end{proof}

\subsection{The Hecke relation and the main term}\label{sec:heckerel}
The exact formula involves the traces at the discriminant $D=f^2D_0$. The aim of this subsection is to reduce them to the traces at the fundamental discriminant $D_0$. The tool is a relation between the traces at different discriminants and the Hecke operators.

\recall{For a prime $p$, $M_p$ is the set of integral matrices of determinant $p$, and $T_pF(z)=\sum_{\nu\in\Gamma\backslash M_p}F(\nu z)$ is the Hecke operator (Subsection~\ref{sec:hecke}). For a discriminant $D'$, $\chi_{D'}(p)=\kron{D'}{p}$ is the Kronecker symbol (Subsection~\ref{sec:forms}). The trace $\Tr_{D''}$ is defined to be $0$ if $D''$ is not a discriminant, for instance if $D''$ is not an integer.}

The underlying observation is simple. Every $\mu\in M_p$ maps a geodesic $\ell_q$ of discriminant $D'$ to the geodesic of the form $q\circ\mathrm{adj}(\mu)$, which has discriminant $p^2D'$. Conversely, a geodesic of discriminant $p^2D'$ arises in this way from a certain number of geodesics of discriminant $D'$, and this number depends only on the power of $p$ dividing the form (Lemma~\ref{lem:nu}). Comparing the two sides gives the following.

\begin{theorem}[Hecke relation]\label{thm:heckerel}
Let $D'>0$ be a non-square discriminant, $p$ a prime, and $F$ a continuous $\Gamma$-invariant function. Then
\[
\Tr_{D'}(T_pF)=\Tr_{p^2D'}(F)+\chi_{D'}(p)\Tr_{D'}(F)+\big(p-\chi_{D'}(p)\big)\Tr_{D'/p^2}(F).
\]
\end{theorem}

If $F$ is an eigenfunction of $T_p$, the left-hand side is a multiple of $\Tr_{D'}(F)$, and the relation becomes a recursion which expresses the traces at $p^2D'$, $p^4D'$, \dots\ through the trace at $D'$. Solving this recursion for all primes dividing $f$ gives the reduction to $D_0$. Recall that the three kinds of eigenfunctions have the Hecke eigenvalues
\[
T_m1=\sigma(m),\qquad T_mu_j=\sqrt m\,\lambda_j(m)\,u_j,\qquad T_mE(\cdot,s)=\eta_m(s)E(\cdot,s),\quad \eta_m(s)=\sum_{ab=m}a^sb^{1-s},
\]
(Subsection~\ref{sec:hecke}), and that $|\lambda_j(m)|\le\tau(m)m^\theta$ by the Kim--Sarnak bound.

\begin{cor}[Reduction to the fundamental discriminant]\label{cor:factor}
Let $D=f^2D_0$ with $D_0>1$ fundamental. Then
\begin{align*}
\Tr_D(1)&=f\,\h_{D_0}(f)\,\Tr_{D_0}(1),\\
\Tr_D(u_j)&=A_j(f)\,\Tr_{D_0}(u_j),\\
\Tr_D\big(E(\cdot,s)\big)&=A_E(f,s)\,\Tr_{D_0}\big(E(\cdot,s)\big),
\end{align*}
where $\h_{D_0}(f)$ is defined in \eqref{hdef} and
\[
A_j(f)=\sum_{d\mid f}\mu(d)\chi_{D_0}(d)\sqrt{\frac{f}{d}}\,\lambda_j\Big(\frac{f}{d}\Big),\qquad A_E(f,s)=\sum_{d\mid f}\mu(d)\chi_{D_0}(d)\,\eta_{f/d}(s).
\]
These factors satisfy
\[
|A_j(f)|\le\tau(f)^2f^{1/2+\theta},\qquad |A_E(f,s)|\le\tau(f)^2f^{1/2}\quad\text{for }\Re s=\tfrac{1}{2} .
\]
\end{cor}

For $D_0=1$ and $f=n$ (that is, with $\chi_{D_0}$ replaced by $1$) the three factors become $\sum_{d\mid n}\mu(d)\sigma(n/d)=n$ and the factors $A_j(n)$, $A(n,t)$ of the square case (Theorem~\ref{thm:sqexact}); there the single geodesic from $0$ to $\infty$ plays the role of the closed geodesics of discriminant $D_0$. Theorem~\ref{thm:heckerel} is the trace version of the classical relation for $T_{p^2}$ in half-integral weight; our proof is elementary and self-contained.

The proof of Theorem~\ref{thm:heckerel} has two parts. Lemma~\ref{lem:unf} unfolds the left-hand side into a sum over the closed geodesics of discriminant $p^2D'$, with multiplicities $\nu_p$. Lemma~\ref{lem:nu} computes these multiplicities. We use that the right cosets $\nu\Gamma$ in $M_p$ correspond to the $p+1$ sublattices $\nu\Z^2$ of index $p$ in $\Z^2$.

\begin{lemma}[Unfolding]\label{lem:unf}
For $\tilde q\in\Qf(p^2D')$ let
\[
\nu_p(\tilde q)=\#\{\nu\Gamma\in M_p/\Gamma:\ \tilde q\circ\nu\equiv0\bmod p^2\}
\]
be the number of cosets $\nu\Gamma$ for which all three coefficients of $\tilde q\circ\nu$ are divisible by $p^2$. Then for continuous $\Gamma$-invariant $F$,
\[
\Tr_{D'}(T_pF)=\sum_{[\tilde q]\in\Gamma\backslash\Qf(p^2D')}\nu_p(\tilde q)\int_{C_{\tilde q}}F\dd s .
\]
\end{lemma}

\begin{proof}
\emph{Reduction to $F=F_\varphi$.} Both sides are integrals of $F$ against finite measures supported on a fixed compact subset of $\Gamma\backslash\Hh$. Every continuous $F$ coincides on that set with some $F_\varphi=\sum_{\gamma\in\Gamma}\varphi\circ\gamma$, $\varphi\in C_c(\Hh)$: take $\varphi=F\chi$ for a cut-off $\chi\in C_c(\Hh)$ with $\sum_\gamma\chi\circ\gamma=1$ there. So we may assume $F=F_\varphi$.

\emph{Unfolding.} Fix representatives $\nu$ of $M_p/\Gamma$. Then $T_pF_\varphi(w)=\sum_{\mu\in M_p}\varphi(\mu w)=\sum_\nu\sum_{\gamma\in\Gamma}\varphi(\nu\gamma w)$. Unfolding as in the proof of Theorem~\ref{thm:nsexact},
\[
\Tr_{D'}(T_pF_\varphi)=\sum_\nu\sum_{q\in\Qf(D')}\int_{\nu\ell_q}\varphi\dd s .
\]

\emph{Identifying the image geodesics.} $\nu\ell_q=\ell_{\tilde q}$ with $\tilde q=q\circ\mathrm{adj}(\nu)\in\Qf(p^2D')$, by Lemma~\ref{dist}(ii) (after scaling to determinant $1$). Since $\mathrm{adj}(\nu)\nu=p$, we get $q=(\tilde q\circ\nu)/p^2$. So $\tilde q$ arises from $\nu$ if and only if $\tilde q\circ\nu\equiv0\bmod p^2$. This condition depends only on $\nu\Gamma$, and replacing $\nu$ by $\nu\gamma$ permutes $\Qf(D')$.

\emph{Folding back.} Hence the right-hand side equals $\sum_{\tilde q\in\Qf(p^2D')}\nu_p(\tilde q)\int_{\ell_{\tilde q}}\varphi$. Folding, with $\nu_p$ being $\Gamma$-invariant, gives the claim.
\end{proof}

\begin{lemma}[The multiplicities]\label{lem:nu}
Let $\tilde q\in\Qf(p^2D')$, and let $p^k$ be the exact power of $p$ dividing all three coefficients of $\tilde q$ (the $p$-content). Then
\[
\nu_p(\tilde q)=1\ (k=0),\qquad \nu_p(\tilde q)=1+\chi_{D'}(p)\ (k=1),\qquad \nu_p(\tilde q)=p+1\ (k\ge2).
\]
\end{lemma}

\begin{proof}
For a lattice $\Lambda\supset p\Z^2$ of index $p$, the condition reads $\tilde q(\Lambda)\subset p^2\Z$; the middle coefficient is covered because it equals $\tilde q(v_1+v_2)-\tilde q(v_1)-\tilde q(v_2)$.

\emph{$k\ge2$.} Every $\Lambda$ works.

\emph{$k=1$.} Write $\tilde q=pq_1$ with $q_1\in\Qf(D')$ primitive at $p$. Since $q_1(v+pw)\equiv q_1(v)\bmod p$, the condition is that the line $\Lambda/p\Z^2$ is isotropic for $q_1\bmod p\ne0$. The number of isotropic lines is:
\begin{itemize}
\item for odd $p$: $1+\kron{D'}{p}$ if $p\nmid D'$, and $1$ if $p\mid D'$ (then the form has rank one);
\item for $p=2$, writing $q_1=[a,b,c]$: two if $b$ is odd and $ac$ is even ($D'\equiv1\bmod8$), none if $a,b,c$ are odd ($D'\equiv5\bmod 8$), and one if $b$ is even (the form is a square mod $2$).
\end{itemize}
In all cases this equals $1+\chi_{D'}(p)$.

\emph{$k=0$: only the radical line can work.} $\tilde q\bmod p$ is non-zero and degenerate, since $p\mid p^2D'$. Choose a basis with $\tilde q\equiv au^2$, $p\nmid a$; then $p\mid b$ and $p\mid c$. From $\tilde q(v+pw)=\tilde q(v)+pB'(v,w)+p^2\tilde q(w)$, with $B'$ the polar form:
\begin{itemize}
\item for odd $p$ the line must lie in the radical $\langle e_2\rangle$;
\item for $p=2$, $B'$ is even since $b$ is even, and $\tilde q(v)\equiv0\bmod2$ again forces $v\in\langle e_2\rangle$.
\end{itemize}

\emph{$k=0$: the radical line works.} It remains to check $p^2\mid c$.
\begin{itemize}
\item For odd $p$: from $p^2\mid b^2-4ac$ and $p\mid b$ we get $p^2\mid ac$, hence $p^2\mid c$.
\item For $p=2$: write $b=2b'$, so $b'^2-ac=D'$. If $D'\equiv1\bmod 4$, then $b'$ is odd (as $c$ is even), so $ac=b'^2-D'\equiv0\bmod4$. If $4\mid D'$, then $b'$ is even and again $4\mid ac$. In both cases $4\mid c$.\qedhere
\end{itemize}
\end{proof}

\begin{proof}[Proof of Theorem~\ref{thm:heckerel}]
Split $\Qf(p^2D')$ according to the $p$-content into three parts:
\begin{itemize}
\item the forms with $k=0$;
\item the forms with $k=1$: these are $p$ times the forms in $\Qf(D')$ which are not divisible by $p$;
\item the forms with $k\ge2$: these are $p^2$ times the forms in $\Qf(D'/p^2)$.
\end{itemize}
A form and its multiples have the same geodesic, $C_{p\tilde q}=C_{\tilde q}$. So the three parts contribute to $\sum_{[\tilde q]}\int_{C_{\tilde q}}F$ the amounts $\Tr_{p^2D'}(F)-\Tr_{D'}(F)$, $\Tr_{D'}(F)-\Tr_{D'/p^2}(F)$ and $\Tr_{D'/p^2}(F)$, respectively. By Lemmas~\ref{lem:unf} and~\ref{lem:nu}, with the multiplicities $1$, $1+\chi_{D'}(p)$ and $p+1$,
\[
\Tr_{D'}(T_pF)=\big[\Tr_{p^2D'}-\Tr_{D'}\big](F)+(1+\chi_{D'}(p))\big[\Tr_{D'}-\Tr_{D'/p^2}\big](F)+(p+1)\Tr_{D'/p^2}(F),
\]
which is the relation.
\end{proof}

\begin{proof}[Proof of Corollary~\ref{cor:factor}]
\emph{One prime.} Let $F\in\{1,u_j,E(\cdot,s)\}$, so $T_mF=\kappa(m)F$ with $\kappa(m)=\sigma(m)$, $\sqrt m\lambda_j(m)$ or $\eta_m(s)$, respectively. These satisfy the Hecke recursion $\kappa(p^{a+1})=\kappa(p)\kappa(p^a)-p\kappa(p^{a-1})$. Write $f=p^af'$ with $p\nmid f'$, $D_1=f'^2D_0$ and $b_i=\Tr_{p^{2i}D_1}(F)$. We apply Theorem~\ref{thm:heckerel} with $D'=p^{2i}D_1$, and note two things.
\begin{itemize}
\item $D_1/p^2$ is not a discriminant: for odd $p$, $p^2\nmid D_1$; for $p=2$, either $D_0$ is odd, or $D_0=4m$ with $m\equiv2,3\bmod4$ and $D_1/4=f'^2m$.
\item $\chi_{D_1}(p)=\chi_{D_0}(p)$, using $f'^2\equiv1\bmod8$ if $p=2$; and $\chi_{p^{2i}D_1}(p)=0$ for $i\ge1$.
\end{itemize}
The relation therefore gives
\[
b_1=(\kappa(p)-\chi_{D_0}(p))b_0,\qquad b_{i+1}=\kappa(p)b_i-pb_{i-1}\ (i\ge1),
\]
whence $b_a=b_0(\kappa(p^a)-\chi_{D_0}(p)\kappa(p^{a-1}))$ by the Hecke recursion.

\emph{All primes.} Induct on the number of prime factors of $f$, using the multiplicativity of $\kappa$ and $\chi_{D_0}$. This gives $\Tr_D(F)=\Tr_{D_0}(F)\sum_{d\mid f}\mu(d)\chi_{D_0}(d)\kappa(f/d)$. For $F=1$ the sum is $f\h_{D_0}(f)$ by \eqref{hdef}; for $F=u_j$ and $F=E(\cdot,s)$ it is $A_j(f)$ and $A_E(f,s)$.

\emph{The bounds.} By the Kim--Sarnak bound, $|\sqrt m\lambda_j(m)|\le\tau(m)m^{1/2+\theta}$, and $|\eta_m(s)|\le\tau(m)\sqrt m$ for $\Re s=1/2$. Each sum has at most $\tau(f)$ terms.
\end{proof}

With the corollary we can evaluate the first term of the exact formula. It is the main term of Theorem~\ref{thm:ns}.

\begin{proposition}[Main term]\label{prop:main}
For every non-square discriminant $D=f^2D_0$,
\[
\Tr_D(1)=2\sqrt D\,L(1,\chi_{D_0})\,\h_{D_0}(f),
\]
and the first term in Theorem~\ref{thm:nsexact} equals
\[
\frac{3}{\pi} H(0)\Tr_D(1)=\frac{6}{\pi} L(1,\chi_{D_0})\,\h_{D_0}(f)\,\Psi_0 .
\]
\end{proposition}

\begin{proof}
\emph{The fundamental discriminant.} All forms in $\Qf(D_0)$ are primitive. There are $h^+(D_0)$ classes, and each closed geodesic has length $2\log\varepsilon^+_{D_0}$ (Subsection~\ref{sec:quot}). So by Dirichlet's class number formula (Fact~\ref{F:gauss}(iv)),
\[
\Tr_{D_0}(1)=h^+(D_0)\cdot2\log\varepsilon^+_{D_0}=2\sqrt{D_0}\,L(1,\chi_{D_0}).
\]

\emph{General $D$.} By Corollary~\ref{cor:factor}, $\Tr_D(1)=f\h_{D_0}(f)\Tr_{D_0}(1)$, and $f\sqrt{D_0}=\sqrt D$.

\emph{The first term.} $H(0)=\Psi_0/\sqrt D$ by Lemma~\ref{hbounds}(ii).
\end{proof}

Inserting Corollary~\ref{cor:factor} and Proposition~\ref{prop:main} into Theorem~\ref{thm:nsexact}, we obtain the exact formula in its final form:
\begin{equation}\label{nsfinal}
\begin{split}
N(D)={}&\frac{6}{\pi} L(1,\chi_{D_0})\,\h_{D_0}(f)\,\Psi_0+\sum_{j:\ u_j\ \mathrm{even}}h(t_j)\,A_j(f)\,u_j(i)\,\Tr_{D_0}(u_j)\\
&+\frac{1}{4\pi}\int_\R h(t)\,A_E(f,\tfrac{1}{2}-it)\,E(i,\tfrac{1}{2}+it)\,\Tr_{D_0}\big(E(\cdot,\tfrac{1}{2}-it)\big)\dd t .
\end{split}
\end{equation}
The first term is the main term of Theorem~\ref{thm:ns}. It remains to bound the other two. The factors $h(t_j)$ and $A_j(f)$ are under control, by \eqref{hbound} and Corollary~\ref{cor:factor}. The values $u_j(i)$ are bounded on average by the local Weyl law. The only missing piece is a bound for the traces $\Tr_{D_0}(u_j)$ at the fundamental discriminant.

\subsection{The traces of cusp forms at fundamental discriminants}\label{sec:fundper}
Let $D_0>1$ be fundamental. By Fact~\ref{F:DIT}, the trace of a cusp form is a product of two central $L$-values:
\begin{equation}\label{DITbound}
|\Tr_{D_0}(u_j)|^2\ll\sqrt{D_0}\,\frac{\alpha_j}{1+t_j}\,L(\tfrac{1}{2},u_j)\,L(\tfrac{1}{2},u_j\otimes\chi_{D_0}),\qquad \alpha_j=\frac{|\rho_j(1)|^2}{\cosh(\pi t_j)},
\end{equation}
where both $L$-values are non-negative. Here $L(s,u_j)$ is the $L$-function of $u_j$, $L(s,u_j\otimes\chi_{D_0})$ is its twist by the character $\chi_{D_0}$, and $\rho_j(1)$ is the first Fourier coefficient of $u_j$ (Subsection~\ref{sec:Lfun}). We bound the mean square of the traces by combining a second moment of the first $L$-value with a cubic moment of the second. Recall from Subsection~\ref{sec:spec} that for $\Gamma$ every $t_j$ is real and $t_j>3$.

\begin{proposition}\label{prop:ms}
Let $D_0>1$ be a fundamental discriminant and $T\ge1$. Then
\[
\sum_{\substack{T<t_j\le2T\\ u_j\ \mathrm{even}}}|\Tr_{D_0}(u_j)|^2\ll_\varepsilon D_0^{5/6+\varepsilon}T^{1+\varepsilon}.
\]
\end{proposition}

\begin{proof}
Write $L_j=L(1/2,u_j)$ and $L_j^\chi=L(1/2,u_j\otimes\chi_{D_0})$; both are $\ge0$. All sums below run over even $u_j$ with $T<t_j\le2T$.

\emph{Step 1: reduction to a mixed moment.} By \eqref{DITbound}, the left-hand side is $\ll\sqrt{D_0}\,T^{-1}\sum\alpha_jL_jL_j^\chi$.

\emph{Step 2: Cauchy--Schwarz.} We have
\[
\sum\alpha_jL_jL^\chi_j\le\Big(\sum\alpha_jL_j^2\Big)^{1/2}\Big(\sum\alpha_j(L_j^\chi)^2\Big)^{1/2}.
\]
The first factor is $\ll(T^2\log T)^{1/2}$ by Motohashi's second moment $\sum_{t_j\le T}\alpha_jL_j^2\ll T^2\log T$ (Fact~\ref{F:L}(b)).

\emph{Step 3: Hölder.} For the second factor we use Hölder's inequality with the exponents $3$ and $3/2$:
\[
\sum\alpha_j(L_j^\chi)^2\le\Big(\sum\alpha_j\Big)^{1/3}\Big(\sum\alpha_j(L_j^\chi)^3\Big)^{2/3}.
\]
The two sums on the right are bounded by Fact~\ref{F:L}:
\begin{itemize}
\item $\sum_{t_j\le2T}\alpha_j\ll T^2$ (Fact~\ref{F:L}(d));
\item $\alpha_j=c/L(1,\mathrm{sym}^2u_j)$ (Fact~\ref{F:L}(c)), so $\alpha_j(L_j^\chi)^3=c\,(L_j^\chi)^3/L(1,\mathrm{sym}^2u_j)$, and the cubic moment (Fact~\ref{F:L}(a)) bounds the sum of these terms over a window $T'\le t_j<T'+1$ by $(D_0T')^{1+\varepsilon}$. We cover $(T,2T]$ by $\lceil T\rceil+1$ such windows; on each window we may add the terms of the odd forms, which are non-negative.
\end{itemize}
Hence
\[
\sum\alpha_j(L_j^\chi)^2\ll(T^2)^{1/3}\big(T\cdot(D_0T)^{1+\varepsilon}\big)^{2/3}\ll D_0^{2/3+\varepsilon}T^{2+\varepsilon}.
\]

\emph{Conclusion.} Hence $\sum\alpha_jL_jL_j^\chi\ll D_0^{1/3+\varepsilon}T^{2+\varepsilon}$, and the claim follows.
\end{proof}

\subsection{Proof of Theorem~\ref{thm:ns}}\label{sec:nsproof}
We start from the exact formula \eqref{nsfinal}. Its first term is the main term of Theorem~\ref{thm:ns}. It remains to bound the cuspidal part (the sum over $j$) and the Eisenstein part (the integral over $t$). We collect what we know about the factors:
\begin{itemize}
\item $|h(t)|\ll\delta\min(1,(\delta|t|)^{-2})$, where $\delta=M/\sqrt D=M/(f\sqrt{D_0})\in(0,1]$; this is \eqref{hbound};
\item $|A_j(f)|\le\tau(f)^2f^{1/2+\theta}$ and $|A_E(f,1/2-it)|\le\tau(f)^2f^{1/2}$ (Corollary~\ref{cor:factor});
\item $\sum_{t_j\le T}|u_j(i)|^2\ll T^2$ and $\int_{-T}^T|E(i,1/2+it)|^2\dd t\ll T^2$ (local Weyl law, Fact~\ref{F:weyl});
\item $\sum_{T<t_j\le2T}|\Tr_{D_0}(u_j)|^2\ll D_0^{5/6+\varepsilon}T^{1+\varepsilon}$, the sum being over even $u_j$ (Proposition~\ref{prop:ms}).
\end{itemize}

\emph{Cuspidal part.} By the Cauchy--Schwarz inequality, the last two items give, for $T\ge1$,
\[
\sum_{\substack{T<t_j\le2T\\ u_j\ \mathrm{even}}}|u_j(i)\Tr_{D_0}(u_j)|\ll T\cdot\big(D_0^{5/6+\varepsilon}T^{1+\varepsilon}\big)^{1/2}=D_0^{5/12+\varepsilon}T^{3/2+\varepsilon}.
\]
We split the sum over $j$ into dyadic ranges $T<t_j\le2T$, $T=2^k$, and use the bounds for $h$ and $A_j(f)$. This gives
\[
\ll\tau(f)^2f^{1/2+\theta}D_0^{5/12+\varepsilon}\,\delta\sum_{T}T^{3/2+\varepsilon}\min(1,(\delta T)^{-2})\ll D^{\varepsilon}f^{1/2+\theta}D_0^{5/12}\delta^{-1/2}=D^\varepsilon f^{1+\theta}D_0^{2/3}M^{-1/2}.
\]
The sum over $T$ is dominated by $T\approx1/\delta$: only the cusp forms with $t_j$ up to about $\sqrt D/M$ matter. In the last step we used $\delta^{-1/2}=f^{1/2}D_0^{1/4}M^{-1/2}$.

\emph{Eisenstein part.} The trace of the Eisenstein series at $D_0$ is given by Fact~\ref{F:Eis}:
\[
\Tr_{D_0}\big(E(\cdot,s)\big)=2\,\frac{\sqrt\pi\,\Gamma(s/2)}{\Gamma((s+1)/2)}\Big(\frac{D_0}{4}\Big)^{s/2}\frac{\zeta(s)L(s,\chi_{D_0})}{\zeta(2s)} .
\]
On the line $s=1/2-it$, Stirling's formula bounds the quotient of Gamma functions by $(1+|t|)^{-1/2}$, the Weyl bound gives $|L(1/2+it,\chi_{D_0})|\ll(D_0(1+|t|))^{1/6+\varepsilon}$ (Fact~\ref{F:L}(a)), and $|\zeta(1-2it)|^{-1}\ll\log(2+|t|)$ (Fact~\ref{F:zeta}). Hence
\[
\big|\Tr_{D_0}(E(\cdot,\tfrac{1}{2}-it))\big|\ll D_0^{1/4}(1+|t|)^{-1/2}|\zeta(\tfrac{1}{2}+it)|(D_0(1+|t|))^{1/6+\varepsilon},
\]
where the factor $\log(2+|t|)$ has been absorbed into $(1+|t|)^\varepsilon$. By the Cauchy--Schwarz inequality, the local Weyl law and the second moment $\int_T^{2T}|\zeta(1/2+it)|^2\dd t\ll T\log T$ (Fact~\ref{F:zeta}),
\[
\int_{T\le|t|\le2T}|E(i,\tfrac{1}{2}+it)\Tr_{D_0}(E(\cdot,\tfrac{1}{2}-it))|\dd t\ll D_0^{5/12+\varepsilon}T^{7/6+\varepsilon},
\]
and the range $|t|\le1$ contributes $O(D_0^{5/12+\varepsilon})$. Dyadic summation against the bounds for $h$ and $A_E$ gives $\ll D^\varepsilon f^{1/2}D_0^{5/12}\delta^{-1/6}$, which is smaller than the cuspidal bound since $\delta\le1$.

\emph{Conclusion.} Both parts are $\ll D^\varepsilon f^{1+\theta}D_0^{2/3}M^{-1/2}$. This proves Theorem~\ref{thm:ns}.\qed

\medskip
The exponents can be traced back as follows. The factor $D_0^{2/3}=D_0^{1/4}\cdot D_0^{5/12}$ comes from $\delta^{-1/2}$ and from the mean square of the traces; in the latter, $D_0^{5/12}$ is the square root of $D_0^{1/2}\cdot D_0^{1/3}$, where $D_0^{1/2}$ is the factor in \eqref{DITbound} and $D_0^{1/3}$ comes from the cubic moment. The factor $f^{1+\theta}=f^{1/2}\cdot f^{1/2+\theta}$ comes from $\delta^{-1/2}$ and from the Hecke factor $A_j(f)$, bounded by Kim--Sarnak.

\subsection{Proof of Corollary~\ref{cor:tern}}\label{sec:tern}
Recall that $D=4^kD'$ with $D'\equiv5\bmod8$, $D'=f'^2D_0$ and $f=2^kf'$. We pass from $N$ to $\Rr$ by the parity reduction.

\emph{Reduction.} By Proposition~\ref{prop:par}(b), $\Rr(D')=2N(D')$. By Proposition~\ref{prop:par}(c), $\Rr_0(D;M)=\Rr(D/4;M/2)$ and $\Rr_1(D;M)=2N(D;M)-2\Rr(D/4;M/2)$ for $4\mid D$; adding these,
\[
\Rr(D;M)=2N(D;M)-\Rr\big(\tfrac{D}{4};\tfrac{M}{2}\big).
\]

\emph{Main terms.} Let $D=4^kD'$ and $f=2^kf'$. Then $D_0\equiv5\bmod8$, so $\chi_{D_0}(2)=-1$ and, for $a\ge1$, $\h_{D_0}(2^af')=(3-2^{1-a})\h_{D_0}(f')$. By Theorem~\ref{thm:ns}, the main term of $\Rr(D;M)$ is $(6/\pi) L(1,\chi_{D_0})\h_{D_0}(f')\Psi_0\cdot r_k$, where $r_0=2$ and $r_k=2(3-2^{1-k})-(1/2)r_{k-1}$. Here $\Psi_0(M/2)=(1/2)\Psi_0(M)$. By induction $r_k=4-2^{1-k}$.

\emph{Error terms.} The error of the $j$-th step is
\[
\ll(D/4^j)^\varepsilon(f/2^j)^{1+\theta}D_0^{2/3}(M/2^j)^{-1/2}=2^{-j(1/2+\theta)}\cdot D^\varepsilon f^{1+\theta}D_0^{2/3}M^{-1/2}\cdot 4^{-j\varepsilon},
\]
which sums geometrically. Here $M/2^j\le\sqrt{D/4^j}$ holds throughout. If $M/2^j<2$, only $|x_3|\le3$ contributes, so the remaining count and its main term are $\ll D^\varepsilon$ by the divisor bound, as in the proof of Theorem~\ref{thm:sq}. This is admissible, since $M\le f\sqrt{D_0}$ gives $f^{1+\theta}D_0^{2/3}M^{-1/2}\ge1$.\qed

\section{Square discriminants}\label{sec:square}
In this section we treat square discriminants $D=n^2$ and prove Theorem~\ref{thm:sqsimple}, in the more precise form of Theorem~\ref{thm:sq}. Throughout the section $\delta=M/\sqrt D=M/n$.

\recall{From Section~\ref{sec:count} we use the following. The count is $N(n^2)=\sum_{q\in\Qf(n^2)}k(\dist(i,\ell_q))$ (Corollary~\ref{geosum}), where $\ell_q$ is the geodesic of the form $q$. The kernel $K$ satisfies $\int_\ell K(d(z,w))\dd s(w)=k(\dist(z,\ell))$ for every geodesic $\ell$; this is \eqref{kernelidentity}. The function $F_G(w)=\sum_{\gamma\in G}K(d(\gamma i,w))$ has the spectral expansion \eqref{Fexp}, with the transform $h$ and $H(1/2+it)=h(t)$. The letter $G$ stands for $\Gamma$ or the theta group $\Gt$ (Subsection~\ref{sec:Gconv}); on first reading, take $G=\Gamma$.}

\subsection{Overview and results}\label{sec:sqoverview}
The proof follows the same line as in the non-square case. Table~\ref{tab:compare} puts the two cases side by side. Three things are different.

\emph{The geodesics are not closed.} The roots of a form of discriminant $n^2$ are rational, so its geodesic runs from cusp to cusp and has trivial stabiliser (Figure~\ref{fig:sqgeod}). Every such form is equivalent to a form $[0,n,C]$, whose geodesic is the vertical line $\Re z=-C/n$ (Lemma~\ref{orbits}). These lines are the images of the single geodesic $\ell_0=i\R_{>0}$ under $z\mapsto z-C/n$, and Hecke operators collapse the whole family onto $\ell_0$ (Proposition~\ref{reduction}). The integral of a Maass form along $\ell_0$ is Hecke's integral, a multiple of $L(1/2,u_j)$. So in the square case one central $L$-value appears, instead of the product of two in the non-square case.

\begin{figure}[t]
\centering
\includegraphics[width=\textwidth]{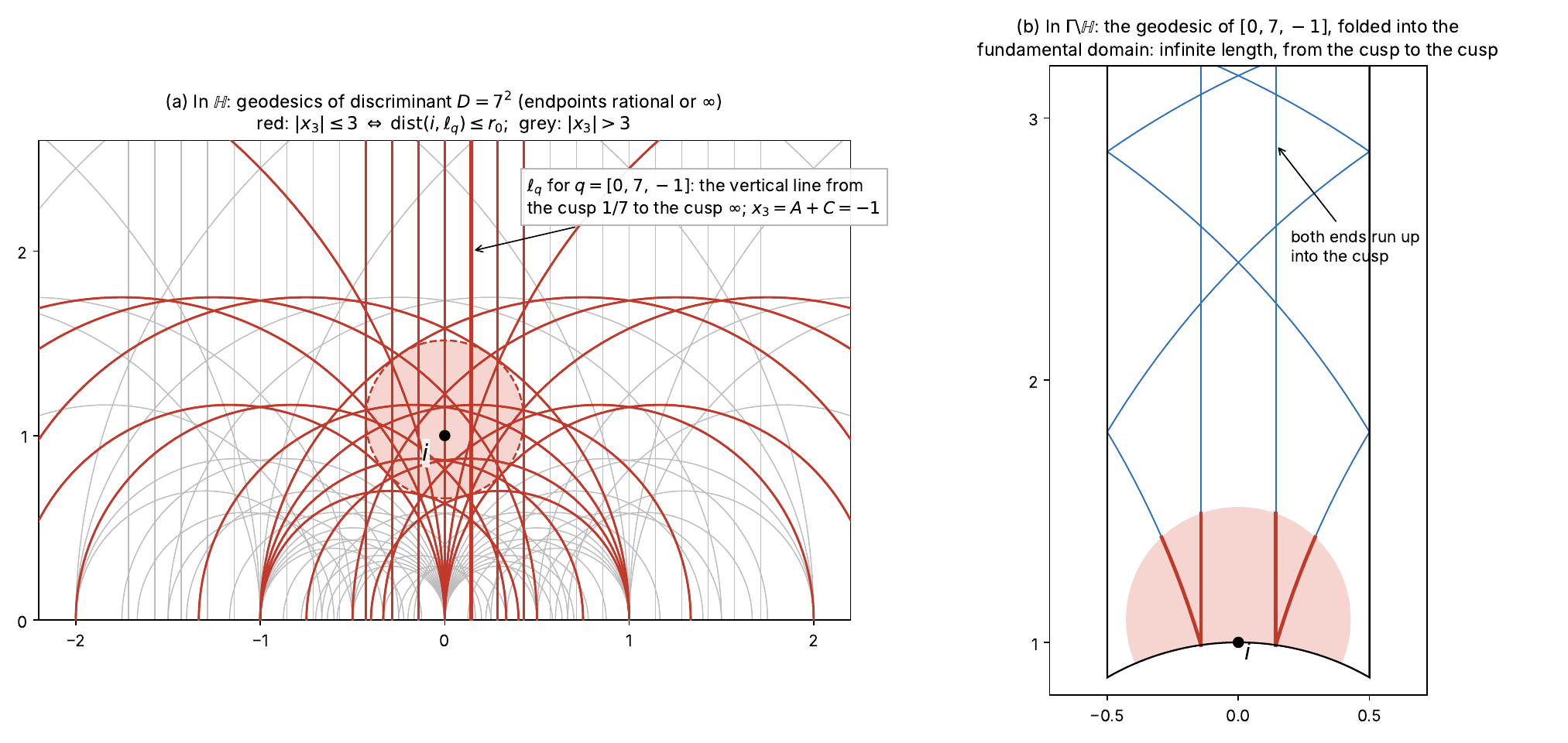}
\caption{Square discriminants, here $D=7^2$ and $M=1.5$. (a) In $\Hh$ the geodesics have rational endpoints (or $\infty$); the red ones meet the ball around $i$, i.e.\ have $|x_3|\le3$. The vertical lines are the geodesics of the forms $[0,7,C]$. (b) In $\Gamma\backslash\Hh$ the geodesic of $[0,7,-1]$, folded into the fundamental domain: it comes down from the cusp, crosses the ball around $i$, and goes back up into the cusp. Its length is infinite.}\label{fig:sqgeod}
\end{figure}

\emph{Regularisation.} Since $\ell_0$ has infinite length, the integral of the constant function along it, $\int_0^\infty\dd y/y$, diverges, and so does that of the Eisenstein series, whose constant term grows like $y^s$. The automorphic function $F_G$ itself vanishes outside a compact range of $\ell_0$ (Lemma~\ref{vanish}), so we may integrate over $Y^{-1}\le y\le Y$ with a large $Y$. In the spectral expansion, the constant function and the Eisenstein series then produce terms of size $\log Y$, and these cancel. What remains is finite. For the Eisenstein series it is a ``regularised period'', obtained by removing the divergent constant term (Subsection~\ref{sec:sqeisper}). The cancellation leaves behind the derivatives $H'(0)$ and $\eta_d'(0)$, and these produce the logarithm in the main term (the quantity $\Psi_1$ below) and the arithmetic function $\varpi(n)$.

\emph{Parity.} For odd $n$ we have $n^2\equiv1\bmod 8$, so by Proposition~\ref{prop:par} the parity of $x_3$ requires the count $N_\vartheta(n^2)$ of forms with $A$ and $C$ even. These forms are permuted by the theta group $\Gt$, not by all of $\Gamma$. We therefore treat both groups $G\in\{\Gamma,\Gt\}$ simultaneously, and write $N_\Gamma=N$, $\Qf_\Gamma=\Qf$, $\Qf_\vartheta=\Qf^\vartheta$. For even $n$ the parity reduction is a recursion in $n$.

\begin{table}[ht]
\centering\small\renewcommand{\arraystretch}{1.5}
\begin{tabularx}{\textwidth}{@{}>{\raggedright}p{0.2\textwidth}>{\raggedright}X>{\raggedright\arraybackslash}X@{}}
\toprule
 & Non-square $D=f^2D_0$ (Section~\ref{sec:nonsquare}) & Square $D=n^2$ (this section)\\
\midrule
Geodesics & closed, of finite length & from cusp to cusp, of infinite length\\
Classes of forms & finitely many & exactly $n$, represented by $[0,n,C]$, $C\bmod n$\\
Count as an integral & $N(D)=\Tr_D(F)$ (Theorem~\ref{thm:nsexact}) & $N(n^2)=\sum_{d\mid n}\mu(n/d)\int_{\ell_0}T_dF$ (Proposition~\ref{reduction})\\
Role of Hecke operators & reduce $D$ to $D_0$ (Corollary~\ref{cor:factor}) & reduce all geodesics to $\ell_0=i\R_{>0}$\\
Period of a cusp form & $\Tr_{D_0}(u_j)$; its square is $L(1/2,u_j)L(1/2,u_j\otimes\chi_{D_0})$ & $\Pp_j=\int_{\ell_0}u_j$; a multiple of $L(1/2,u_j)$\\
Mean square of the periods & second and cubic moments of $L$-functions (Proposition~\ref{prop:ms}) & approximate functional equation and large sieve (Proposition~\ref{periodmean})\\
Constant function, Eisenstein series & convergent integrals & divergent; regularised, $\log Y$ cancels\\
Parity of $x_3$ & only $D\equiv5\bmod 8$ and $4^kD$ & all $n$, by means of the theta group\\
Result & Theorem~\ref{thm:ns} & Theorems~\ref{thm:sqsimple} and~\ref{thm:sq}\\
\bottomrule
\end{tabularx}
\caption{The two cases side by side.}\label{tab:compare}
\end{table}

\emph{The main-term data.} We now state the results. Let $v=v_2(n)$ be the exponent of $2$ in $n$, and let $\Lambda$ be the von Mangoldt function. Put
\[
\Psi_1=\Psi_1(n;M)=\int_\R\Phi\Big(\frac{x}{M}\Big)\log\sqrt{n^2+x^2}\,\dd x,\qquad \varpi(n)=\sum_{d\mid n}\frac{\Lambda(d)}{d},\qquad \vpo(n)=\sum_{\substack{d\mid n\\ d\ \mathrm{odd}}}\frac{\Lambda(d)}{d},
\]
Recall that $\Psi_0=\int_\R\Phi(x/M)\dd x$. Since $|x|\le2M\le2n$ on the support of the integrand, $\Psi_1=(\log n)\Psi_0+O(M)$. Moreover $\varpi(n)=\sum_{p^a\mid n}(\log p)/p^a\ll\log\log n$. (In \cite{B26}, $\vpo(n)$ is denoted $\lambda(n)$.) Let
\[
E^*(i)=\lim_{s\to1}\Big(E(i,s)-\frac{3/\pi}{s-1}\Big),\qquad c_1=\frac{\pi}{3}E^*(i)+\frac{\log2}{3},
\]
\[
\beta_e(n)=\begin{cases}0,& n\ \text{odd},\ e=0,\\ \tfrac{1}{2},& n\ \text{odd},\ e=1,\\ 2-\tfrac{3}{2}\,2^{-v},& n\ \text{even},\ e=0,\\ 0,& n\ \text{even},\ e=1.\end{cases}
\]

\begin{theorem}[Square discriminants, precise form]\label{thm:sq}
For all $n\ge1$, $e\in\{0,1\}$, $1\le M\le n$ and $\varepsilon>0$,
\[
\Rr_e(n^2)=\frac{4}{\pi}\Big[\Psi_1+\big(c_1-\vpo(n)-\beta_e(n)\log2\big)\Psi_0\Big]+O\big(n^{1+\theta+\varepsilon}M^{-1/2}\big).
\]
Consequently,
\[
\Rr(n^2)=\frac{8}{\pi}\big[\Psi_1+\big(c_1-\varpi(n)\big)\Psi_0\big]-\frac{2^{1-v}\log2}{\pi}\Psi_0+O\big(n^{1+\theta+\varepsilon}M^{-1/2}\big).
\]
The implied constants depend only on $\Phi$ and $\varepsilon$.
\end{theorem}

\begin{proof}[Proof of Theorem~\ref{thm:sqsimple}]
We have $(2/3)(1+\theta)=71/96$, so the error term in Theorem~\ref{thm:sq} is $O(M)$ for $M\ge n^{71/96+\varepsilon}$. Moreover $\Psi_1=(\log n)\Psi_0+O(M)$, $\vpo(n)\ll\log\log n$, and $\beta_e(n)\le2$. Under the Ramanujan conjecture, $\theta=0$ in the error term.
\end{proof}

\begin{remark}[The constant $c_1$]\label{rem:c1}
By the Kronecker limit formula (\cite[Ch.~I, \S1]{Si80}) and the Chowla--Selberg value $|\eta(i)|=\Gamma(1/4)/(2\pi^{3/4})$ of the Dedekind eta function \cite{SC67},
\[
c_1=2\gamma+2\log2+3\log\pi-4\log\Gamma(\tfrac{1}{4})-2\frac{\zeta'(2)}{\zeta(2)}+\frac{\log2}{3}=2.19379629605\ldots
\]
This is the constant $c_0=\gamma+(L'/L)(1,\chi_{-4})+(\log2)/3-2\zeta'(2)/\zeta(2)$ of \cite{B26}, because $(L'/L)(1,\chi_{-4})=\gamma+2\log2+3\log\pi-4\log\Gamma(1/4)$. The main terms of Theorem~\ref{thm:sq} agree with those of \cite[Theorem~1 and Corollary~3]{B26} wherever both apply; the latter were obtained by a completely different method. Comparing the two theorems with $M=n^{0.9}$, which lies in both ranges, gives an independent proof that $c_0=c_1$, and hence of the above value of $L'/L(1,\chi_{-4})$. This comparison is not used anywhere in the paper.
\end{remark}

Theorem~\ref{thm:sq} is deduced from an exact formula for the counts of binary forms, $N_\Gamma(n^2)=N(n^2)$ and $N_\vartheta(n^2)$, which we state now. It is the analogue of the exact formula \eqref{nsfinal} of the non-square case. Recall that $\mathcal A_\Gamma=\N$ and $\mathcal A_\vartheta=\{n\ \text{odd}\}$ are the admissible $n$ for the two groups (for odd $n$, $n^2\equiv1\bmod8$, so that $N_\vartheta(n^2)$ is defined). The spectral data are as follows.
\begin{itemize}
\item $\{u_j\}$ is the basis of Hecke eigenforms for $G$ (Fact~\ref{F:hecke}): $T^G_du_j=\sqrt d\,\lambda_j(d)u_j$ for $d$ coprime to $w_G$, and $|\lambda_j(d)|\le\tau(d)d^\theta$.
\item For even $u_j$, $\Pp_j=\int_0^\infty u_j(iy)\dd y/y$ is the period along $\ell_0$ (Lemma~\ref{mellin}).
\item $\Pp_E(s)$ is the regularised period of $E(\cdot,s)$ along $\ell_0$, see \eqref{Jdef}, and $\kappa_G(t)$ is a bounded factor, see \eqref{qG}; $\kappa_\Gamma=1$.
\item The Hecke factors are
\[
A_j(n)=\sum_{d\mid n}\mu\Big(\frac{n}{d}\Big)\sqrt d\,\lambda_j(d),\qquad A(n,t)=\sum_{d\mid n}\mu\Big(\frac{n}{d}\Big)\eta_d(\tfrac{1}{2}+it),
\]
where $\eta_d(s)=\sum_{ab=d}a^sb^{1-s}$. They are the factors $A_j(f)$ and $A_E(f,s)$ of Corollary~\ref{cor:factor} with $f=n$ and with the character $\chi_{D_0}$ replaced by $1$, and they satisfy the same bounds: $|A_j(n)|\le\tau(n)^2n^{1/2+\theta}$ and $|A(n,t)|\le\tau(n)^2n^{1/2}$.
\end{itemize}

\begin{theorem}[Square discriminants, exact formula]\label{thm:sqexact}
Let $G\in\{\Gamma,\Gt\}$, $n\in\mathcal A_G$ and $1\le M\le n$. Then
\begin{multline*}
N_G(n^2)=\Mm_G(n)+\sum_{j:\ u_j\ \mathrm{even}}h(t_j)A_j(n)u_j(i)\Pp_j\\
+\frac{1}{4\pi}\int_\R h(t)A(n,t)E(i,\tfrac{1}{2}+it)\kappa_G(t)\Pp_E(\tfrac{1}{2}-it)\dd t+O(n^{-5}),
\end{multline*}
where
\[
\Mm_\Gamma(n)=\frac{6}{\pi}\Big[\Psi_1+\Big(c_1-\frac{\log2}{3}-\varpi(n)\Big)\Psi_0\Big],\qquad \Mm_\vartheta(n)=\frac{2}{\pi}\big[\Psi_1+(c_1-\varpi(n))\Psi_0\big].
\]
\end{theorem}

\begin{theorem}[Square discriminants, asymptotic formula]\label{thm:sqasymp}
Let $G\in\{\Gamma,\Gt\}$, $n\in\mathcal A_G$ and $1\le M\le n$. Then for every $\varepsilon>0$,
\[
N_G(n^2)=\Mm_G(n)+O\big(n^{1+\theta+\varepsilon}M^{-1/2}\big).
\]
\end{theorem}

The three theorems are related as follows:
\[
\text{Theorem~\ref{thm:sqexact} (exact)}\ \Longrightarrow\ \text{Theorem~\ref{thm:sqasymp} (asymptotic for $N_G$)}\ \Longrightarrow\ \text{Theorem~\ref{thm:sq} (asymptotic for $\Rr_e$)}.
\]
The rest of this section proves them, in the following steps.
\begin{enumerate}[leftmargin=*]
\item \emph{Reduction to one geodesic} (Subsection~\ref{sec:orbits}). The count $N_G(n^2)$ is expressed through the integrals $I_G(d)$ of $T_d^GF_G$ along the single geodesic $\ell_0=i\R_{>0}$, for $d\mid n$.
\item \emph{Periods of cusp forms} (Subsection~\ref{sec:sqper}). The integral $\Pp_j$ of $u_j$ along $\ell_0$ is a central $L$-value, and the large sieve bounds its mean square.
\item \emph{Eisenstein series} (Subsections~\ref{sec:Etheta} and~\ref{sec:sqeisper}). We collect the facts on the Eisenstein series of the two groups and compute their regularised periods along $\ell_0$.
\item \emph{Exact formula} (Subsection~\ref{sec:sqexp}). We integrate the spectral expansion of $T^G_dF_G$ along $\ell_0$. The divergent terms cancel, and Theorem~\ref{thm:sqexact} follows.
\item \emph{Conclusion} (Subsections~\ref{sec:sqasympproof} and~\ref{sec:sqproof}). Bounding the spectral terms gives Theorem~\ref{thm:sqasymp}, and the parity reduction gives Theorem~\ref{thm:sq}.
\end{enumerate}

\subsection{Orbits and reduction to one geodesic}\label{sec:orbits}
The aim of this subsection is Proposition~\ref{reduction}: the count $N_G(n^2)$ is a combination of integrals of Hecke translates of $F_G$ along the single geodesic $\ell_0=i\R_{>0}$. We proceed in three steps. First we describe the classes of forms of discriminant $n^2$ (Lemma~\ref{orbits}); their geodesics are vertical lines. Then we check that $F_G$ vanishes on these lines outside a compact range (Lemma~\ref{vanish}). Finally we use Hecke operators to move all the vertical lines to $\ell_0$.

Recall that $\Qf(n^2)$ is the set of integral forms $[A,B,C]$ of discriminant $n^2$, and that $\Qf^\vartheta(n^2)$ is the subset with $A$ and $C$ even (for odd $n$).

\begin{lemma}[Classes of forms of square discriminant]\label{orbits}
\begin{enumerate}[label=(\roman*)]
\item The map
\[
(C\bmod n,\ \gamma)\ \mapsto\ [0,n,C]\circ\gamma\qquad(\gamma\in\PSL)
\]
is a bijection onto $\Qf(n^2)$. Moreover, if $g'\in\PSL$ and
\[
[0,n,C]\circ g'=[0,n,C'],
\]
then $g'=T^k$ and $C'=C+nk$ for some $k\in\Z$.
\item Let $n$ be odd. The map
\[
(c\bmod n,\ \gamma)\ \mapsto\ [0,n,2c]\circ\gamma\qquad(\gamma\in\Gt)
\]
is a bijection onto $\Qf^\vartheta(n^2)$.
\end{enumerate}
\end{lemma}

\begin{proof}
\emph{(i), surjectivity.} A form $q$ of discriminant $n^2$ factors over $\Z$ as $q=\ell_1\ell_2$ into linear forms, ordered so that $\det(\ell_1,\ell_2)=n>0$. Let $(p,p')$ be a primitive zero of $\ell_2$ and let $\gamma_0\in\mathrm{SL}_2(\Z)$ have first column $(p,p')^T$. Then $\ell_2\circ\gamma_0=\lambda v$ and $\ell_1\circ\gamma_0=\alpha u+\beta v$, and the determinant condition gives $\lambda\alpha=n$. So $q\circ\gamma_0=[0,n,\lambda\beta]$, and translations $T^k$ change $\lambda\beta$ by multiples of $n$.

\emph{(i), the second statement.} Write $[0,n,C]=(nu+Cv)\cdot v$. The factorisation into linear forms is unique up to scalars and order, composition with $g'$ preserves the determinant of the ordered pair, and swapping the factors gives determinant $-n$. Hence $v\circ g'=\pm v$, so $g'=\pm T^k$, and then $C'=C+nk$. This also gives injectivity.

\emph{(ii), invariance.} The parities of $q(1,0)=A$ and $q(0,1)=C$ for $q\circ\gamma$ depend only on $\gamma\bmod2$. For $\gamma\equiv I$ they are unchanged, and for $\gamma\equiv\left(\begin{smallmatrix}0&1\\1&0\end{smallmatrix}\right)$ they are swapped. Hence $\Qf^\vartheta(n^2)$ is $\Gt$-invariant, and it contains $[0,n,2c]$.

\emph{(ii), surjectivity.} Let $q\in\Qf^\vartheta(n^2)$. By (i), $q=[0,n,C]\circ\gamma$ with $\gamma\in\PSL$. Since $[0,n,C]=[0,n,C-n]\circ T$, we may assume $C$ even. Then $[0,n,C]\equiv uv\bmod 2$, which is odd only at $(u,v)\equiv(1,1)$. Hence $q(1,0)$ and $q(0,1)$ are even if and only if neither column of $\gamma$ is $\equiv(1,1)^T$. So the columns are $(1,0)^T,(0,1)^T$ in some order, i.e.\ $\gamma\in\Gt$.

\emph{(ii), injectivity.} If $[0,n,2c]\circ\gamma=[0,n,2c']\circ\gamma'$ with $\gamma,\gamma'\in\Gt$, then by (i) $\gamma\gamma'^{-1}=T^k$ and $2c'=2c+nk$. Since $T^k\in\Gt$ only for even $k$, we get $c'\equiv c\bmod n$. For fixed representatives this forces $k=0$, and then $\gamma=\gamma'$.
\end{proof}

So there are exactly $n$ classes of forms of discriminant $n^2$, represented by $[0,n,C]$ with $C$ modulo $n$. The geodesic of $[0,n,C]=(nu+Cv)v$ joins the roots $-C/n$ and $\infty$; it is a vertical line. For $b\in\R$ let
\[
\ell_b=\{z\in\Hh:\Re z=b/n\}
\]
be the geodesic from the cusp $b/n$ to the cusp $\infty$. Recall from Subsection~\ref{sec:hecke} that $B_\Gamma(n)$ is a complete residue system modulo $n$, and $B_\vartheta(n)=\{0,2,\dots,2n-2\}$.

\begin{cor}\label{orbitsum}
For $n\in\mathcal A_G$,
\[
N_G(n^2)=\sum_{b\in B_G(n)}\ \sum_{\gamma\in G}k\big(\dist(\gamma i,\ell_b)\big).
\]
\end{cor}

\begin{proof}
Combine Corollary~\ref{geosum}, Lemma~\ref{orbits} and Lemma~\ref{dist}(ii). The geodesic of $[0,n,C]=(nu+Cv)v$ is $\ell_{-C}$. For $G=\Gt$, as $c$ runs modulo $n$, $b=-2c$ runs over the even residues modulo $2n$.
\end{proof}

Thus the count is a sum over the points $\gamma i$ of the weights $k(\dist(\gamma i,\ell_b))$, for the $n$ vertical lines $\ell_b$. By the kernel identity \eqref{kernelidentity}, each such weight is the integral of $K(d(\gamma i,\cdot))$ along $\ell_b$, and the sum over $\gamma$ produces $F_G=\sum_{\gamma\in G}K(d(\gamma i,\cdot))$. So the count is the sum of the integrals of $F_G$ along the lines $\ell_b$. These lines have infinite length, and we first check that only a compact part of each of them matters. Recall that $K(d)=0$ for $d>r_0$.

\begin{lemma}[Compact support along the geodesics]\label{vanish}
\begin{enumerate}[label=(\roman*)]
\item If $F_G(w)\ne0$, then $\Im(\gamma w)\le e^{r_0}$ for all $\gamma\in\PSL$.
\item For $d$ coprime to $w_G$, $(T^G_dF_G)(iy)=0$ unless $(de^{r_0})^{-1}\le y\le de^{r_0}$, and $(T^G_dF_G)(i/y)=(T^G_dF_G)(iy)$.
\end{enumerate}
\end{lemma}

\begin{proof}
\emph{(i)} We have $d(\gamma_0i,w)\le r_0$ for some $\gamma_0\in G$. For $\gamma\in\PSL$, $\Im(\gamma\gamma_0i)\le1$, and $\Im$ changes at most by the factor $e^{r_0}$ over a distance $r_0$.

\emph{(ii), support.} By Lemma~\ref{hecke}(iv) and $M^G_d\subset M_d$: if $K(d(i,\mu iy))\ne0$, then $e^{-r_0}\le\Im\mu(iy)\le e^{r_0}$. For $\mu=(\begin{smallmatrix}a&b\\c&e\end{smallmatrix})$ we have $\Im\mu(iy)=dy/(c^2y^2+e^2)$. This is $\le d/y$ if $c\ne0$, and $\ge y/d$ if $c=0$; in both cases $y\le de^{r_0}$.

\emph{(ii), symmetry.} $T^G_dF_G$ is $G$-invariant and $S\in G$. This also gives the lower bound for $y$.
\end{proof}

The Hecke operator $T^G_d$ is a sum over the maps $z\mapsto(az+b)/d'$ with $ad'=d$ and $b\in B_G(d')$ (Subsection~\ref{sec:hecke}). These maps send the geodesic $\ell_0$ to the vertical lines $\Re z=b/d'$. This is how the Hecke operators collect the integrals along all vertical lines into integrals along $\ell_0$.

\begin{proposition}[Reduction to one geodesic]\label{reduction}
For $n\in\mathcal A_G$,
\[
N_G(n^2)=\sum_{d\mid n}\mu\Big(\frac{n}{d}\Big)I_G(d),\qquad I_G(d):=\int_0^\infty(T^G_dF_G)(iy)\frac{\dd y}{y},
\]
where the integral runs over a compact range by Lemma~\ref{vanish}.
\end{proposition}

\begin{proof}
\emph{Step 1: integrals along geodesics.} For $m\in\mathcal A_G$ let $J_G(m)=\sum_{b\in B_G(m)}\int_0^\infty F_G((iy+b)/m)\dd y/y$; this is the sum of the integrals of $F_G$ along the vertical geodesics $\ell_b$ (for $D=m^2$). Each integral runs over a compact range, by Lemma~\ref{vanish}(i) applied near $\infty$ and near the rational point $b/m$. By Corollary~\ref{orbitsum} and Lemma~\ref{kernel},
\[
N_G(n^2)=\sum_{b\in B_G(n)}\int_{\ell_b}\sum_{\gamma}K(d(\gamma i,w))\dd s(w)=J_G(n),
\]
where the interchange is justified since only finitely many $\gamma$ contribute at each $w$.

\emph{Step 2: Hecke operators.} Substituting $y\mapsto y/a$,
\[
\int_0^\infty(T^G_mF_G)(iy)\frac{\dd y}{y}=\sum_{ad'=m}\ \sum_{b\in B_G(d')}\int_0^\infty F_G\Big(\frac{aiy+b}{d'}\Big)\frac{\dd y}{y}=\sum_{d'\mid m}J_G(d').
\]
In words: the Hecke operator $T_m^G$ moves the geodesic $\ell_0$ to all the vertical geodesics $\ell_b$ of the discriminants $d'^2$, $d'\mid m$.

\emph{Step 3.} Möbius inversion gives $J_G(n)=\sum_{d\mid n}\mu(n/d)I_G(d)$.
\end{proof}

\upshot{$N_G(n^2)=\sum_{d\mid n}\mu(n/d)I_G(d)$, where $I_G(d)$ is the integral of $T^G_dF_G$ along $\ell_0$. This is the analogue of $N(D)=\Tr_D(F)$. Next we insert the spectral expansion of $F_G$. The Hecke operator $T^G_d$ multiplies each eigenfunction by its Hecke eigenvalue, so we need the integrals of the eigenfunctions along $\ell_0$: the periods of the cusp forms (Subsection~\ref{sec:sqper}) and of the Eisenstein series (Subsection~\ref{sec:sqeisper}).}

\subsection{Periods of cusp forms}\label{sec:sqper}
The aim of this subsection is to compute the period
\[
\Pp_j=\int_0^\infty u_j(iy)\frac{\dd y}{y}
\]
of a cusp form along $\ell_0$ (Lemma~\ref{mellin}), and to bound the mean square of these periods over the spectrum (Proposition~\ref{periodmean}). The integral converges, since cusp forms decay exponentially in the cusp. For odd $u_j$ we have $u_j(iy)=0$, because $R$ fixes $\ell_0$ pointwise; so only even forms matter.

Let $u=u_j$ be even under $R$. Its Fourier expansion at the cusp $\infty$ is
\[
u(w_Gz)=\sqrt y\sum_{m\ne0}\rho_j(m)K_{it_j}(2\pi|m|y)e(mx),\qquad \rho_j(-m)=\rho_j(m),
\]
where $w_G$ is the width of the cusp ($w_\Gamma=1$, $w_\vartheta=2$). We shall apply the large sieve (Fact~\ref{F:LS}) to these coefficients.
\begin{itemize}
\item For $G=\Gamma$, Fact~\ref{F:LS} applies with $q=1$ and $\phi_j=u_j$.
\item For $G=\Gt$, the functions $\phi_j=u_j\circ g$ form an orthonormal basis for $\Gamma_0(2)$ (Lemma~\ref{conj}). Since $\phi_j(z)=u_j(2z+1)$, their coefficients are $(-1)^m\rho_j(m)$, so Fact~\ref{F:LS} applies with $q=2$ (the signs are absorbed into $a_m$).
\end{itemize}

The first lemma is technical: it gives crude bounds for the coefficients and for the decay of $u_j$ along $\ell_0$, which justify the manipulations below.

\begin{lemma}[Crude bounds]\label{coeff}
For $m\ne0$,
\[
|\rho_j(m)|\ll_\varepsilon\cosh(\pi t_j)^{1/2}\big(1+|t_j|+|m|^{1/2+\varepsilon}\big).
\]
Consequently, for $y\ge1$,
\[
|u_j(iy)|\ll e^{\pi|t_j|/2}(1+|t_j|)e^{-\pi y},
\]
and $\int_Y^\infty|u_j(iy)|\dd y/y\ll e^{-Y}$ whenever $Y\ge\max(10,10|t_j|)$.
\end{lemma}

\begin{proof}
\emph{Coefficients.} Apply Fact~\ref{F:LS} with $K=1+|t_j|$, a single $m$, and drop the other terms by positivity.

\emph{Bessel bound.} From $K_{\nu}(x)=\int_0^\infty e^{-x\cosh r}\cosh(\nu r)\dd r$ and $|\cosh(\nu r)|\le\cosh(r/2)$ for $\nu\in i\R$ and for $\nu\in(-1/2,1/2)$, we get $|K_{it_j}(x)|\le K_{1/2}(x)=\sqrt{\pi/(2x)}e^{-x}$.

\emph{Pointwise bound.} Hence
\[
|u(iy)|\le\sqrt{y/w_G}\sum_{m\ne0}|\rho_j(m)|\,|K_{it_j}(2\pi|m|y/w_G)|\ll\sum_{m\ge1}|\rho_j(m)|m^{-1/2}e^{-\pi my},
\]
using $w_G\le2$. Insert the coefficient bound.

\emph{Tail.} For $y\ge Y\ge10|t_j|$ we have $e^{\pi|t_j|/2-\pi y}\le e^{-(19\pi/20)y}$, and $1+|t_j|\le y$.
\end{proof}

The period $\Pp_j$ is a Mellin transform at $0$, and Hecke's classical argument relates it to the $L$-function of $u_j$. Let $\gamma_t(s)=(w_G/\pi)^s\Gamma((s+it)/2)\Gamma((s-it)/2)$.

\begin{lemma}[Hecke's period]\label{mellin}
$M_j(u)=\int_0^\infty u_j(iy)y^u\dd y/y$ is entire, bounded in vertical strips, and satisfies $M_j(u)=M_j(-u)$. For $\Re u>3/2$,
\[
M_j(u)=\frac{1}{2\sqrt{w_G}}\Lambda_j(u+\tfrac{1}{2}),\qquad \Lambda_j(s)=\gamma_{t_j}(s)\sum_{m\ge1}\rho_j(m)m^{-s}.
\]
In particular $\Pp_j=M_j(0)=\Lambda_j(1/2)/(2\sqrt{w_G})$, where $\Lambda_j$ is entire with $\Lambda_j(s)=\Lambda_j(1-s)$. For $G=\Gamma$ this gives
\[
|\Pp_j|^2=\frac{|\rho_j(1)|^2}{4\pi}\,\big|\Gamma(\tfrac{1}{4}+\tfrac{it_j}{2})\big|^4\,L(\tfrac{1}{2},u_j)^2 .
\]
\end{lemma}

\begin{proof}
\emph{Analytic properties.} $u_j(i/y)=u_j(S\,iy)=u_j(iy)$, so $M_j(u)=M_j(-u)$. By Lemma~\ref{coeff}, $u_j(iy)$ decays exponentially as $y\to\infty$, and by symmetry as $y\to0$. Hence $M_j$ is entire and bounded in vertical strips.

\emph{Mellin transform.} For $\Re u>3/2$ insert
\[
u_j(iy)=2\sqrt{y/w_G}\sum_{m\ge1}\rho_j(m)K_{it_j}(2\pi my/w_G)
\]
and use the Mellin transform of the Bessel function,
\[
\int_0^\infty K_{it}(x)x^\beta\frac{\dd x}{x}=2^{\beta-2}\,\Gamma\Big(\frac{\beta+it}{2}\Big)\Gamma\Big(\frac{\beta-it}{2}\Big).
\]
The interchange of sum and integral is justified by Lemma~\ref{coeff}. The term $m$ gives
\[
\frac{1}{2}\,w_G^{u}\,\pi^{-u-1/2}\,\Gamma\Big(\frac{u+1/2+it}{2}\Big)\Gamma\Big(\frac{u+1/2-it}{2}\Big)\rho_j(m)m^{-u-1/2},
\]
and summing over $m$ gives the formula.

\emph{The case $G=\Gamma$.} Here $\rho_j(m)=\rho_j(1)\lambda_j(m)$, so
\[
\Lambda_j(\tfrac{1}{2})=\pi^{-1/2}\,\big|\Gamma(\tfrac{1}{4}+\tfrac{it_j}{2})\big|^2\rho_j(1)L(\tfrac{1}{2},u_j).
\]
\end{proof}

We need the mean square of the periods over the spectrum. Since $\Pp_j$ is essentially a central $L$-value, we write it as a short Dirichlet polynomial (approximate functional equation) and apply the large sieve.

\begin{proposition}[Mean square of the periods]\label{periodmean}
For $T\ge1$ and $\varepsilon>0$,
\[
\sum_{T<|t_j|\le2T}|\Pp_j|^2\ll_\varepsilon T^{1+\varepsilon},\qquad \sum_{|t_j|\le1}|\Pp_j|^2\ll1,
\]
where the sums run over even $u_j$.
\end{proposition}

\begin{proof}
\emph{Approximate functional equation.} By Lemma~\ref{mellin}, shifting contours from $\Re u=2$ to $\Re u=-2$ and using $u\mapsto-u$,
\[
\Lambda_j(\tfrac{1}{2})=\frac{2}{2\pi i}\int_{(2)}\Lambda_j(\tfrac{1}{2}+u)e^{u^2}\frac{\dd u}{u}=2\sum_{m\ge1}\rho_j(m)m^{-1/2}W_{t_j}(m),
\]
where
\[
W_t(x)=\frac{1}{2\pi i}\int_{(2)}\gamma_t(\tfrac{1}{2}+u)e^{u^2}x^{-u}\frac{\dd u}{u} .
\]

\emph{Stirling bound.} For $u=\sigma+iv$ with $\sigma\in[\varepsilon,2]$, Stirling's formula and $|v+t|+|v-t|\ge2|t|$ give
\[
|\gamma_t(\tfrac{1}{2}+u)|\ll_\varepsilon(1+|t|)^{\sigma-1/2}(1+|v|)^{2}e^{-\pi|t|/2}
\]
for real $t$. For the finitely many $t\in i(0,1/2)$ the left-hand side is $\ll_\varepsilon(1+|v|)^2$.

\emph{Dyadic decomposition.} Let $\omega$ be smooth with support in $[1/2,2]$ and $\sum_{k\ge0}\omega(x/2^k)=1$ for $x\ge1$. For $N=2^k$ and any $\sigma>0$,
\[
a_N(j):=\sum_m\rho_j(m)m^{-1/2}\omega(\tfrac{m}{N})W_{t_j}(m)=\frac{1}{2\pi i}\int_{(\sigma)}\gamma_{t_j}(\tfrac{1}{2}+u)e^{u^2}S_N(u)\frac{\dd u}{u},
\]
with $S_N(u)=\sum_m\rho_j(m)\omega(m/N)m^{-1/2-u}$. Take $\sigma=\varepsilon$ if $N\le T^2$ and $\sigma=2$ if $N>T^2$. By Cauchy--Schwarz in $v$, for $T<|t_j|\le2T$,
\[
|a_N(j)|^2\ll_\varepsilon T^{2\sigma-1}e^{-\pi|t_j|}\int_\R(1+|v|)^{2}e^{-v^2}|S_N(\sigma+iv)|^2\dd v .
\]

\emph{Large sieve.} Since $e^{-\pi|t|}\le2/\cosh(\pi t)$, Fact~\ref{F:LS} (applied to the two blocks $(N/2,N]$ and $(N,2N]$) gives
\[
\sum_{T<|t_j|\le2T}|a_N(j)|^2\ll_\varepsilon T^{2\sigma-1}(T^2+N^{1+\varepsilon})N^{-2\sigma}.
\]
This is $\ll T^{1+4\varepsilon}$ for $N\le T^2$, and $\ll T^{3}N^{-3+\varepsilon}$ for $N>T^2$.

\emph{Summation over $N$.} Finally
\[
|\Pp_j|^2\ll\Big(\sum_N|a_N(j)|\Big)^2\le2\log(4T^2)\sum_{N\le T^2}|a_N(j)|^2+2\Big(\sum_{N>T^2}\Big(\frac{T^2}{N}\Big)^{1/2}\Big)\sum_{N>T^2}\Big(\frac{N}{T^2}\Big)^{1/2}|a_N(j)|^2 .
\]
Summing over $j$ gives the first claim. The second follows in the same way with $T=1$.
\end{proof}

\subsection{Eisenstein series of the two groups}\label{sec:Etheta}
The spectral expansion \eqref{Fexp} of $F_G$ contains the Eisenstein series of $G$. For $G=\Gamma$ this is the series $E(z,s)$ of Subsection~\ref{sec:spec}. In this subsection we express the Eisenstein series of the theta group through $E$ (Lemma~\ref{lem:Etheta}), and then collect, for both groups in a uniform notation, the data which enter the expansion at the point $i$ and along $\ell_0$ (Lemma~\ref{lem:Edata}). A reader who restricts to $G=\Gamma$ needs from this subsection only the following: $e_\Gamma=E$, $c_\Gamma=1$, $a_\Gamma=1$, $b_\Gamma=\varphi$, and the Laurent expansion of $E(i,1-s)$ at $s=0$ in Lemma~\ref{lem:Edata}(iv).

Recall that $\Gt$ has the two cusps $\infty$ and $1$, with scaling matrices $\sigma_\infty z=2z$ and $\sigma_1=(\begin{smallmatrix}1&0\\1&1\end{smallmatrix})$, and stabilisers $\langle T^2\rangle$ and $\Gamma_1$ (Subsection~\ref{sec:Gconv}). The Eisenstein series of these cusps are
\[
E^\vartheta_\infty(z,s)=\sum_{\gamma\in\langle T^2\rangle\backslash\Gt}\Im(\sigma_\infty^{-1}\gamma z)^s,\qquad E_1^\vartheta(z,s)=\sum_{\gamma\in\Gamma_1\backslash\Gt}\Im(\sigma_1^{-1}\gamma z)^s .
\]
They can be expressed through the Eisenstein series of $\Gamma$.

\begin{lemma}\label{lem:Etheta}
With $E_2(z,s):=E((z+1)/2,s)$,
\[
E^\vartheta_\infty=\frac{2^{-s}E-2^{-2s}E_2}{1-2^{-2s}},\qquad E^\vartheta_1=\frac{2^{-s}E_2-2^{-2s}E}{1-2^{-2s}} .
\]
Consequently:
\begin{enumerate}[label=(\roman*)]
\item each $E^\vartheta_{\mathfrak a}$ is holomorphic in $\Re s>1/2$ except for a simple pole at $s=1$ with residue $1/\pi$, so the residual spectrum of $\Gt$ consists of the constants;
\item $E^\vartheta_\infty(i,s)=E_1^\vartheta(i,s)=E(i,s)/(2^s+1)$;
\item $E^\vartheta_\infty+E^\vartheta_1=e_\vartheta/(2^s+1)$, where $e_\vartheta(z,s):=E(z,s)+E_2(z,s)$;
\item $e_\vartheta$ is $\Gt$- and $S$-invariant, and $T^\vartheta_de_\vartheta=\eta_d(s)e_\vartheta$ for odd $d$.
\end{enumerate}
\end{lemma}

\begin{proof}
\emph{Reduction to lattice sums.} For a set $P$ of parity classes of pairs, let
\[
E_P(z,s)=\frac{1}{2}\sum_{\substack{(c,d)=1\\ (c,d)\bmod2\in P}}y^s|cz+d|^{-2s}.
\]
\begin{itemize}
\item By the description of the cusps, $\langle T^2\rangle\backslash\Gt$ corresponds to the coprime bottom rows $\pm(c,d)$ with $c\not\equiv d\bmod 2$, each exactly once: in a coset $\langle T\rangle\gamma$ the elements $\gamma,T\gamma$ have different classes modulo $2$. Hence $E^\vartheta_\infty=2^{-s}E_{\{(0,1),(1,0)\}}$, the factor $2^{-s}$ coming from $\Im(\sigma_\infty^{-1}w)=\Im w/2$.
\item Likewise $\sigma_1^{-1}\Gt$ consists of the matrices with bottom row $\equiv(1,1)$, and whole $\langle T\rangle$-cosets occur. Hence $E^\vartheta_1=E_{\{(1,1)\}}$.
\end{itemize}

\emph{Evaluation.} Let $G(z,s)=\sum'_{(m,n)}y^s|mz+n|^{-2s}=2\zeta(2s)E(z,s)$.
\begin{itemize}
\item The pairs with $m\equiv n\equiv0$ contribute $2^{-2s}G(z,s)$.
\item Writing $mz+n=2\big(m(z+1)/2+(n-m)/2\big)$, the pairs with $m\equiv n\bmod 2$ contribute $2^{-s}G((z+1)/2,s)$.
\item A non-zero pair in a class $\ne(0,0)$ has odd $\gcd k$, and the primitive pair has the same class. Hence for $P\not\ni(0,0)$ the restricted sum equals $2(1-2^{-2s})\zeta(2s)E_P$.
\end{itemize}
This gives the two formulas.

\emph{(i).} Use Fact~\ref{F:E}; $1-2^{-2s}$ has no zeros in $\Re s>0$. At $s=1$ we have $2^{-s}=1/2$, $2^{-2s}=1/4$ and $1-2^{-2s}=3/4$, and $E$ and $E_2$ both have the residue $3/\pi$. So the residues are
\[
\frac{3}{\pi}\cdot\frac{1/2-1/4}{3/4}=\frac{1}{\pi}\qquad\text{in both cases}.
\] The residual spectrum is spanned by the residues of the Eisenstein series at their poles in $(1/2,1]$ \cite[\S7.3]{Iw}.

\emph{(ii).} $(i+1)/2=T(\begin{smallmatrix}1&0\\-1&1\end{smallmatrix})i$, so $E_2(i,s)=E(i,s)$, and $(2^{-s}-2^{-2s})/(1-2^{-2s})=1/(2^s+1)$.

\emph{(iii).} This follows by adding the two formulas.

\emph{(iv), invariance.} $e_\vartheta$ is $\Gt$-invariant by (iii). $E$ is $\Gamma$-invariant. Moreover $E_2(Sz)=E(\nu z)$ with
\[
\nu=\begin{pmatrix}1&-1\\2&0\end{pmatrix}=\begin{pmatrix}1&-1\\2&-1\end{pmatrix}\begin{pmatrix}1&1\\0&2\end{pmatrix},
\]
where the first factor lies in $\Gamma$ and the second maps $z$ to $(z+1)/2$. So $E_2(Sz)=E_2(z)$.

\emph{(iv), Hecke.} $T_d^\vartheta E=\eta_dE$ by Lemma~\ref{hecke}(ii) and Subsection~\ref{sec:hecke}. For $E_2$, with $w=(z+1)/2$:
\[
\frac{(az+b)/d'+1}{2}=\frac{aw+(b+d'-a)/2}{d'},
\]
and $(b+d'-a)/2$ runs over the residues modulo $d'$ as $b$ runs over $B_\vartheta(d')$. Hence $T^\vartheta_dE_2=(T_dE)(w)=\eta_dE_2$.
\end{proof}

For both groups we now collect the Eisenstein data needed at $z=i$ and along $\ell_0$. In the spectral expansion \eqref{Fexp} the Eisenstein series enter through the sum $\sum_{\mathfrak a}E^G_{\mathfrak a}(i,s)\overline{E^G_{\mathfrak a}(w,s)}$ over the cusps. Part (i) of the next lemma writes this sum as one product; part (ii) gives the constant term of the resulting function along $\ell_0$; parts (iii) and (iv) prepare the cancellation of the divergent terms. Recall from Fact~\ref{F:E} that $\varphi(s)$ is the coefficient of $y^{1-s}$ in the Fourier expansion of $E(z,s)$, and that $E^*(i)$ is the constant term of $E(i,s)$ at its pole $s=1$ (Subsection~\ref{sec:sqoverview}). Put
\[
e_\Gamma=E,\quad c_\Gamma(s)=1,\quad E_\Gamma(i,s)=E(i,s);
\]
\[
e_\vartheta=E+E_2,\quad c_\vartheta(s)=|2^s+1|^{-2},\quad E_\vartheta(i,s)=\frac{E(i,s)}{2^s+1}.
\]

\begin{lemma}\label{lem:Edata}
Let $\Re s=1/2$, $s\ne1/2$, and $s'=1-s=\bar s$.
\begin{enumerate}[label=(\roman*)]
\item \emph{The sum over the cusps.}
\[
\sum_{\mathfrak a}E^G_{\mathfrak a}(i,s)\overline{E^G_{\mathfrak a}(w,s)}=c_G(s)\,E(i,s)\,e_G(w,s').
\]
\item \emph{The constant term along $\ell_0$.}
\[
e_G(iy,s')=a_G(s')y^{s'}+b_G(s')y^{1-s'}+e^\sharp_G(y,s'),
\]
where $e^\sharp_G(y,s')$ is the sum of the non-constant Fourier terms. Here $a_\Gamma=1$, $b_\Gamma=\varphi$, $a_\vartheta(s')=1+2^{-s'}$ and $b_\vartheta(s')=\varphi(s')(1+2^{s'-1})$.
\item \emph{Two identities.}
\[
c_G(s)E(i,s)a_G(s')=E_G(i,s)\,w_G^{-s'},\qquad c_G(s)E(i,s)b_G(s')=E_G(i,s')\,w_G^{-s}.
\]
\item \emph{Behaviour near $s=0$.} $E_G(i,1-s)$ is meromorphic in $-1\le\Re s\le1/2$ with polynomial growth in $|\Im s|$ away from $s=0$, and
\[
E_G(i,1-s)=-\frac{1/\vol_G}{s}+E^*_G+O(s)\quad(s\to0),\qquad E_\Gamma^*=E^*(i),\quad E^*_\vartheta=\frac{E^*(i)}{3}-\frac{2\log2}{3\pi}.
\]
\end{enumerate}
\end{lemma}

\begin{proof}
\emph{(i).} For $\Gamma$ this holds because $\overline{E(w,s)}=E(w,\bar s)$. For $\Gt$ use Lemma~\ref{lem:Etheta}(ii),(iii) and $|2^s+1|^2=(2^s+1)(2^{s'}+1)$.

\emph{(ii).} This is Fact~\ref{F:E}; for $E_2(iy,s')=E(1/2+iy/2,s')$ the constant term is $(y/2)^{s'}+\varphi(s')(y/2)^{1-s'}$.

\emph{(iii), case $\Gamma$.} Use $E(i,s)\varphi(s')=E(i,s')$.

\emph{(iii), case $\Gt$.} Use $(1+2^{-s'})/(2^{s'}+1)=2^{-s'}$ and $1+2^{s'-1}=1+2^{-s}=2^{-s}(2^s+1)$, together with $E(i,s)\varphi(s')=E(i,s')$.

\emph{(iv).} For $\Gamma$ this is Fact~\ref{F:E}. For $\Gt$, $2^{1-s}+1\ne0$ for $\Re s<1$, and
\[
\frac{1}{2^{1-s}+1}=\frac{1}{3}+\frac{2\log2}{9}s+O(s^2),\qquad E(i,1-s)=-\frac{3/\pi}{s}+E^*(i)+O(s).\qedhere
\]
\end{proof}

\subsection{Regularised periods of Eisenstein series}\label{sec:sqeisper}
The aim of this subsection is the analogue of Lemma~\ref{mellin} for Eisenstein series. The Eisenstein series do not decay in the cusp, so their integrals along $\ell_0$ diverge. Their constant term, however, is explicit, and removing it leaves a function which decays rapidly. This leads to the following definition. For an $S$-invariant function $e$ on $i\R_{>0}$ of the form
\[
e(iy)=ay^{s'}+by^{1-s'}+e^\sharp(y),
\]
with $s'\ne1/2$ and $e^\sharp$ decaying exponentially as $y\to\infty$, define the regularised period
\[
\Pp(e)=2\int_1^\infty e^\sharp(y)\frac{\dd y}{y}-\frac{2a}{s'}-\frac{2b}{1-s'} .
\]
It is the value at $w=0$ of the analytic continuation of the Mellin transform $\int_0^\infty e^\sharp(y)y^w\dd y/y$ (see the proof of Lemma~\ref{eisperiod}), and for a cusp form ($a=b=0$) it is the ordinary period. Put
\begin{equation}\label{Jdef}
\Pp_E(s')=\pi^{s'-1/2}\frac{\Gamma(s'/2)\Gamma((1-s')/2)\zeta(s')\zeta(1-s')}{\Gamma(s')\zeta(2s')} .
\end{equation}

Recall that $E_2(z,s)=E((z+1)/2,s)$ and $e_\vartheta=E+E_2$ (Subsection~\ref{sec:Etheta}).

\begin{lemma}[Regularised periods]\label{eisperiod}
Let $s'=1/2-it$, $t\ne0$. Then $\Pp(E(\cdot,s'))=\Pp_E(s')$, $\Pp(E_2(\cdot,s'))=(2^{s'}+2^{1-s'}-2)\Pp_E(s')$, and
\[
|\Pp_E(\tfrac{1}{2}-it)|\ll\frac{|\zeta(1/2+it)|^2\log(2+|t|)}{(1+|t|)^{1/2}} .
\]
\end{lemma}

\begin{proof}
\emph{$\Pp(e)$ as a Mellin transform.} Let $M^\sharp(w)=\int_0^\infty e^\sharp(y)y^w\dd y/y$ for $\Re w$ large. Split at $y=1$, and in $\int_0^1$ substitute $y\mapsto1/y$, using $e^\sharp(1/y)=e^\sharp(y)+c(y)-c(1/y)$ with $c(y)=ay^{s'}+by^{1-s'}$. This gives
\[
M^\sharp(w)=\int_1^\infty e^\sharp(y)(y^w+y^{-w})\frac{\dd y}{y}+\frac{a}{w-s'}+\frac{b}{w-1+s'}-\frac{a}{w+s'}-\frac{b}{w+1-s'}.
\]
The right-hand side is meromorphic and regular at $w=0$, where it equals $\Pp(e)$.

\emph{The Mellin transform of the non-constant part.} For $x_0\in\{0,1/2\}$, Fact~\ref{F:E} gives the Mellin transform of the non-constant part of $E(x_0+iy,s')$ as
\[
4\pi^{s'}(2\pi)^{-w-1/2}2^{w-3/2}\frac{\Gamma((w+s')/2)\Gamma((w+1-s')/2)}{\Gamma(s')\zeta(2s')}\sum_{m\ge1}\epsilon_{x_0}(m)\sigma_{1-2s'}(m)m^{s'-1-w},
\]
where $\epsilon_0=1$ and $\epsilon_{1/2}(m)=(-1)^m$.

\emph{The twisted Dirichlet series.} For multiplicative $a(m)$ with local factor $L_2(z)=\sum_ka(2^k)2^{-kz}$,
\[
\sum_{m\ge1}(-1)^ma(m)m^{-z}=\Big(1-\frac{2}{L_2(z)}\Big)\sum_{m\ge1}a(m)m^{-z},
\]
because the odd $m$ contribute $L_2(z)^{-1}\sum a(m)m^{-z}$. For $a=\sigma_\alpha$, $L_2(z)^{-1}=(1-2^{-z})(1-2^{\alpha-z})$. With $\alpha=1-2s'$, $z=1-s'+w$ and $w=0$, the factor is
\[
1-2(1-2^{s'-1})(1-2^{-s'})=2^{s'}+2^{1-s'}-2 .
\]

\emph{Conclusion.} Since $E_2(iy,s')=E(1/2+iy/2,s')$, its Mellin transform at $w$ carries an extra factor $2^w$, which equals $1$ at $w=0$. Both functions are $S$-invariant (Lemma~\ref{lem:Etheta}(iv)). The value at $w=0$ for $x_0=0$ is $\Pp_E(s')$. The bound follows from Stirling's formula and Fact~\ref{F:zeta}.
\end{proof}

Accordingly, with $p_\Gamma(s')=1$ and $p_\vartheta(s')=2^{s'}+2^{1-s'}-1$ we have $\Pp(e_G(\cdot,s'))=p_G(s')\Pp_E(s')$, and we put
\begin{equation}\label{qG}
\kappa_G(t)=c_G(\tfrac{1}{2}+it)\,p_G(\tfrac{1}{2}-it):\qquad \kappa_\Gamma=1,\qquad \kappa_\vartheta(t)=\frac{2\sqrt2\cos(t\log2)-1}{3+2\sqrt2\cos(t\log2)} .
\end{equation}
In particular $|\kappa_G(t)|\le(2\sqrt2+1)/(\sqrt2-1)^2<23$.

\subsection{The spectral expansion and proof of Theorem~\ref{thm:sqexact}}\label{sec:sqexp}
By Proposition~\ref{reduction},
\[
N_G(n^2)=\sum_{d\mid n}\mu\Big(\frac{n}{d}\Big)I_G(d),\qquad I_G(d)=\int_0^\infty(T^G_dF_G)(iy)\frac{\dd y}{y} .
\]
In this subsection we compute $I_G(d)$ by integrating the spectral expansion of $T_d^GF_G$ along $\ell_0$. The result is formula \eqref{Gd} below, and Theorem~\ref{thm:sqexact} follows from it by summing over $d$.

\emph{The cut-off.} Fix $d\mid n$ (so $d$ is coprime to $w_G$) and put $Y=(n+5)^{10}$. The function $T^G_dF_G$ vanishes on $\ell_0$ outside $(de^{r_0})^{-1}\le y\le de^{r_0}$ (Lemma~\ref{vanish}), and $de^{r_0}\le5n<Y$, because $e^{r_0}\le2+\sqrt5$. Hence
\[
I_G(d)=\int_{1/Y}^Y(T^G_dF_G)(iy)\frac{\dd y}{y} .
\]
The individual terms of the spectral expansion do not vanish for large $y$, and the integrals of the constant function and of the Eisenstein series over $[1/Y,Y]$ grow like $\log Y$. But the left-hand side does not depend on $Y$. \emph{So the terms which depend on $Y$ must cancel.} We shall see this cancellation explicitly, and we shall find what it leaves behind: the main term.

The computation has three parts, one for each part of the spectrum.

\emph{The expansion.} Apply $T^G_d$ in the variable $w$ to the spectral expansion \eqref{Fexp} of $F_G$. This is a finite sum of translates, and the series converges uniformly on compacta, so $T^G_d$ may be applied term by term. Each eigenfunction is multiplied by its Hecke eigenvalue:
\begin{itemize}
\item $T^G_d1=\sigma(d)$ (Lemma~\ref{hecke}(ii));
\item $T^G_du_j=\sqrt d\,\lambda_j(d)\,u_j$ (Fact~\ref{F:hecke});
\item for the Eisenstein series, the sum over the cusps equals $c_G(s)E(i,s)e_G(w,s')$ by Lemma~\ref{lem:Edata}(i), and $e_G$ is a Hecke eigenfunction with $T^G_de_G(\cdot,s')=\eta_d(s')e_G(\cdot,s')$ by Lemma~\ref{lem:Etheta}(iv).
\end{itemize}
Hence
\[
(T^G_dF_G)(w)=\frac{\sigma(d)H(0)}{\vol_G}+\sum_jh(t_j)\sqrt d\lambda_j(d)u_j(i)u_j(w)+\frac{1}{4\pi}\int_\R h(t)c_G(s)E(i,s)\eta_d(s')e_G(w,s')\dd t,
\]
with $s=1/2+it$, $s'=1-s$, and $\eta_d(s')=\eta_d(s)$. Integrating over $w=iy$, $Y^{-1}\le y\le Y$, gives
\[
I_G(d)=I_{\rm const}+I_{\rm cusp}+I_{\rm Eis}.
\]

\emph{Constant part.} $I_{\rm const}=\vol_G^{-1}\sigma(d)H(0)\cdot2\log Y$. This term diverges as $Y\to\infty$; it will be cancelled by the Eisenstein part.

\emph{Cuspidal part.} Only even $u_j$ contribute, and $\int_{1/Y}^Y=\Pp_j-2\int_Y^\infty$ by symmetry.
\begin{itemize}
\item For $|t_j|\le Y/10$ the tail is $\ll e^{-Y}$ by Lemma~\ref{coeff}.
\item For $|t_j|>Y/10$ we use $|h(t_j)|\ll_A\delta(\delta|t_j|)^{-A}$, Cauchy--Schwarz, Fact~\ref{F:weyl} in the form
\[
\sum_{T<|t_j|\le2T}\Big|\int_{1/Y}^Yu_j(iy)\frac{\dd y}{y}\Big|^2\ll(\log Y)^2(T^2+TY),
\]
and Proposition~\ref{periodmean}.
\end{itemize}
Since $\delta\ge n^{-1}$, this gives
\[
I_{\rm cusp}=\sum_{j:\ u_j\ \mathrm{even}}h(t_j)\sqrt d\lambda_j(d)u_j(i)\Pp_j+O(n^{-10}).
\]

\emph{Eisenstein part: decomposition.} We need the integral of $e_G(iy,s')$ over $1/Y\le y\le Y$. By Lemma~\ref{lem:Edata}(ii),
\[
e_G(iy,s')=a_G(s')y^{s'}+b_G(s')y^{1-s'}+e^\sharp_G(y,s'),
\]
where $e^\sharp_G$ decays exponentially. Since $e_G(i/y,s')=e_G(iy,s')$ by $S$-invariance, the integral over $[1/Y,Y]$ is twice the integral over $[1,Y]$. Integrating the two powers of $y$ explicitly and comparing with the definition of the regularised period in Subsection~\ref{sec:sqeisper}, we get
\[
\int_{1/Y}^Ye_G(iy,s')\frac{\dd y}{y}=p_G(s')\Pp_E(s')+\frac{2a_G(s')Y^{s'}}{s'}+\frac{2b_G(s')Y^{1-s'}}{1-s'}-2\int_Y^\infty e^\sharp_G(y,s')\frac{\dd y}{y} .
\]
The first term is the regularised period $\Pp(e_G(\cdot,s'))=p_G(s')\Pp_E(s')$, which does not depend on $Y$. The next two terms come from the constant term and carry the dependence on $Y$. The last term is a tail.

\emph{Eisenstein part: the tail.} For $|t|\le Y/10$ the tail is $O(e^{-Y})$, by Fact~\ref{F:E}, $|\Gamma(1/2-it)|^{-1}\le e^{\pi|t|/2}$ and the Bessel bound of Lemma~\ref{coeff}. For $|t|>Y/10$ the whole integral is bounded polynomially in $t$ and $Y$ (Fact~\ref{F:weyl} for $E$ at $iy$ and $(1+iy)/2$), and $h$ is negligible there. Hence the tail contributes $O(n^{-10})$.

\emph{Eisenstein part: the $Y$-terms and the cancellation.} By Lemma~\ref{lem:Edata}(iii), they contribute
\[
\frac{1}{4\pi}\int_\R h(t)\eta_d(s)\Big[\frac{2E_G(i,s)Y^{s'}}{w_G^{s'}s'}+\frac{2E_G(i,s')Y^{1-s'}}{w_G^{s}(1-s')}\Big]\dd t=\frac{1}{\pi i}\int_{(1/2)}H(s)\eta_d(s)E_G(i,1-s)\Big(\frac{Y}{w_G}\Big)^s\frac{\dd s}{s},
\]
where the two terms are equal after $t\mapsto-t$, since $h$ and $\eta_d$ are even in $t$. Shift the contour to $\Re s=-1$.
\begin{itemize}
\item By Lemma~\ref{hbounds}(i),(iii), $H$ decays rapidly in the strip.
\item $\eta_d(s)\ll\tau(d)d^2$.
\item By Lemma~\ref{lem:Edata}(iv), the only pole is at $s=0$, and the integral on $\Re s=-1$ is $\ll d^3Y^{-1}$.
\end{itemize}
Write $H(s)\eta_d(s)=a_0+a_1s+\dots$ with $a_0=H(0)\sigma(d)$ and $a_1=H'(0)\sigma(d)+H(0)\eta_d'(0)$. By Lemma~\ref{lem:Edata}(iv), twice the residue at $s=0$ is
\[
2a_0E_G^*-\frac{2}{\vol_G}a_0\log\frac{Y}{w_G}-\frac{2}{\vol_G}a_1 .
\]
Its $\log Y$ term is $-2\vol_G^{-1}\sigma(d)H(0)\log Y$, and this cancels $I_{\rm const}$ exactly, as predicted. The remaining terms do not depend on $Y$. They involve the derivatives $H'(0)$ and $\eta_d'(0)$, and they produce the main term.

\emph{Result.} We collect the three parts. Recall that $\kappa_G(t)=c_G(1/2+it)\,p_G(1/2-it)$ is the bounded factor from \eqref{qG}, with $\kappa_\Gamma=1$, and that $E^*_G$ is the constant term of $E_G(i,1-s)$ at $s=0$ (Lemma~\ref{lem:Edata}(iv)). Altogether,
\begin{equation}\label{Gd}
\begin{split}
I_G(d)={}&2\sigma(d)H(0)\Big(E^*_G+\frac{\log w_G}{\vol_G}\Big)-\frac{2}{\vol_G}\big(\sigma(d)H'(0)+H(0)\eta_d'(0)\big)\\
&+\sum_{j:\ u_j\ \mathrm{even}}h(t_j)\sqrt d\lambda_j(d)u_j(i)\Pp_j\\
&+\frac{1}{4\pi}\int_\R h(t)\eta_d(\tfrac{1}{2}+it)E(i,\tfrac{1}{2}+it)\kappa_G(t)\Pp_E(\tfrac{1}{2}-it)\dd t+O(d^3n^{-10}).
\end{split}
\end{equation}

\begin{proof}[Proof of Theorem~\ref{thm:sqexact}]
Insert \eqref{Gd} into Proposition~\ref{reduction}.

\emph{Divisor sums.} We have $\sum_{d\mid n}\mu(n/d)\sigma(d)=n$. Next, $\eta_d(s)=\sum_{ab=d}a^sb^{1-s}$ gives $\eta_d'(0)=\sum_{ab=d}b\,(\log a-\log b)$. In terms of the Dirichlet convolution $*$ and the function $\mathrm{id}(m)=m$, this is $\eta_d'(0)=(\log*\mathrm{id})(d)-((\mathrm{id}\cdot\log)*1)(d)$. Since $\mu*\log=\Lambda$ (the von Mangoldt function) and $\mu*1$ is the unit of the convolution,
\[
\sum_{d\mid n}\mu\Big(\frac{n}{d}\Big)\eta'_d(0)=(\Lambda*\mathrm{id})(n)-n\log n=n\sum_{d\mid n}\frac{\Lambda(d)}{d}-n\log n=n(\varpi(n)-\log n).
\]

\emph{Main term.} By Lemma~\ref{hbounds}(ii) and the substitution $x=n\sinh r$, $nH(0)=\Psi_0$ and $-nH'(0)=\Psi_1-(\log n)\Psi_0$. The non-spectral terms therefore give
\[
\frac{2}{\vol_G}\Psi_1+\Big(2E_G^*+\frac{2}{\vol_G}(\log w_G-\varpi(n))\Big)\Psi_0 .
\]
For $\Gamma$ this is $(6/\pi)\Psi_1+(2E^*(i)-(6/\pi)\varpi(n))\Psi_0=\Mm_\Gamma(n)$. For $\Gt$, by Lemma~\ref{lem:Edata}(iv), $2E^*_\vartheta+(2/\pi)\log2=(2/\pi)\big((\pi/3)E^*(i)+(\log2)/3\big)=(2/\pi)c_1$, which gives $\Mm_\vartheta(n)$.

\emph{Spectral terms and error.} The spectral terms combine as stated, since $\sqrt d\sum_{ab=d}(a/b)^{it}=\eta_d(1/2+it)$. The error is $\sum_{d\mid n}O(d^3n^{-10})=O(n^{-5})$.
\end{proof}

\subsection{Proof of Theorem~\ref{thm:sqasymp}}\label{sec:sqasympproof}
We bound the two spectral terms of Theorem~\ref{thm:sqexact}. The argument is the same as in the proof of Theorem~\ref{thm:ns} (Subsection~\ref{sec:nsproof}), with the periods $\Pp_j$ in place of the traces $\Tr_{D_0}(u_j)$. We use:
\begin{itemize}
\item $|h(t)|\ll\delta\min(1,(\delta|t|)^{-2})$ with $\delta=M/n$; this is \eqref{hbound};
\item $|A_j(n)|\le\tau(n)^2n^{1/2+\theta}$ and $|A(n,t)|\le\tau(n)^2n^{1/2}$ (Kim--Sarnak);
\item $\sum_{|t_j|\le T}|u_j(i)|^2\ll T^2$ and $\int_{-T}^T|E(i,1/2+it)|^2\dd t\ll T^2$ (local Weyl law, Fact~\ref{F:weyl});
\item $\sum_{T<|t_j|\le2T}|\Pp_j|^2\ll T^{1+\varepsilon}$ (Proposition~\ref{periodmean}).
\end{itemize}

\emph{Cuspidal part.} By the Cauchy--Schwarz inequality,
\[
\sum_{T<|t_j|\le2T}|u_j(i)\Pp_j|\ll T\cdot T^{1/2+\varepsilon}\qquad(T\ge1),
\]
and the block $|t_j|\le1$ is $O(1)$. With \eqref{hbound}, dyadic summation bounds the cuspidal part by
\[
\ll\tau(n)^2n^{1/2+\theta}\delta\sum_{T=2^k}T^{3/2+\varepsilon}\min(1,(\delta T)^{-2})\ll n^{1/2+\theta+\varepsilon}\delta^{-1/2-\varepsilon}\ll n^{1+\theta+2\varepsilon}M^{-1/2}.
\]

\emph{Eisenstein part.} Here $|\kappa_G|<23$, and by Lemma~\ref{eisperiod},
\[
|\Pp_E(\tfrac{1}{2}-it)|\ll\frac{|\zeta(1/2+it)|^2\log(2+|t|)}{(1+|t|)^{1/2}} .
\] By the Cauchy--Schwarz inequality, the local Weyl law and the fourth moment of $\zeta$ (Fact~\ref{F:zeta}),
\[
\int_{T\le|t|\le2T}|E(i,\tfrac{1}{2}+it)\Pp_E(\tfrac{1}{2}-it)|\dd t\ll T\Big(\int_T^{2T}\frac{|\zeta(1/2+it)|^4\log^2t}{t}\dd t\Big)^{1/2}\ll T(\log T)^3,
\]
and the range $|t|\le1$ is bounded. Dyadic summation against $\delta\min(1,(\delta T)^{-2})$ gives $\ll n^{1/2+\varepsilon}$.

\emph{Conclusion.} Since $M\le n$, we have $n^{1/2}\le n^{1+\theta}M^{-1/2}$, and Theorem~\ref{thm:sqasymp} follows from Theorem~\ref{thm:sqexact}.\qed

\subsection{Proof of Theorem~\ref{thm:sq}}\label{sec:sqproof}
It remains to pass from the counts $N_G(n^2)$ of binary forms to the counts $\Rr_e(n^2)$ of all solutions with $x_3\equiv e\bmod2$. This is done by the parity reduction (Proposition~\ref{prop:par}), which we recall. For odd $n$ (then $n^2\equiv1\bmod 8$),
\[
\Rr_0(n^2)=2N_\vartheta(n^2),\qquad \Rr_1(n^2)=N(n^2)-N_\vartheta(n^2),
\]
and for even $n$,
\[
\Rr_0(n^2;M)=\Rr\big((\tfrac{n}{2})^2;\tfrac{M}{2}\big),\qquad \Rr_1(n^2;M)=2N(n^2;M)-2\Rr\big((\tfrac{n}{2})^2;\tfrac{M}{2}\big).
\]
Let $\Mm^{(e)}(n;M)$ be the main term in Theorem~\ref{thm:sq}, $\Mm^{\rm all}=\Mm^{(0)}+\Mm^{(1)}$, and $\Mm_G(n;M)=\Mm_G(n)$ as in Theorem~\ref{thm:sqexact}. We first check that the main terms satisfy the same identities as the counts, and then control the errors by induction on $v_2(n)$.

\begin{lemma}\label{lem:mainid}
For odd $n$, $\Mm^{(0)}=2\Mm_\vartheta$ and $\Mm^{(1)}=\Mm_\Gamma-\Mm_\vartheta$; for even $n$,
\[
\Mm^{(0)}(n;M)=\Mm^{\rm all}(\tfrac{n}{2};\tfrac{M}{2}),\qquad \Mm^{(1)}(n;M)=2\Mm_\Gamma(n;M)-2\Mm^{\rm all}(\tfrac{n}{2};\tfrac{M}{2}).
\]
\end{lemma}

\begin{proof}
\emph{Odd $n$.} Here $\varpi(n)=\vpo(n)$, so $2\Mm_\vartheta=(4/\pi)[\Psi_1+(c_1-\vpo)\Psi_0]=\Mm^{(0)}$. Moreover
\[
\Mm_\Gamma-\Mm_\vartheta=\frac{4}{\pi}\Psi_1+\Big(\frac{6}{\pi}\Big(c_1-\frac{\log2}{3}-\vpo\Big)-\frac{2}{\pi}(c_1-\vpo)\Big)\Psi_0=\Mm^{(1)} .
\]

\emph{Even $n$, $v=v_2(n)\ge1$: auxiliary identities.} We have $\vpo(n/2)=\vpo(n)$ and $\varpi(n)=\vpo(n)+(1-2^{-v})\log 2$. Substituting $x=x'/2$ gives
\[
\Psi_1(\tfrac{n}{2};\tfrac{M}{2})=\tfrac{1}{2}\big(\Psi_1(n;M)-(\log2)\Psi_0(M)\big),\qquad \Psi_0(\tfrac{M}{2})=\tfrac{1}{2}\Psi_0(M).
\]
Finally $\beta_0(n/2)+\beta_1(n/2)=2-3\cdot2^{-v}$, both for $v=1$ (when this equals $1/2$) and for $v\ge2$.

\emph{Even $n$: the identities.} With the abbreviations $\Psi_i=\Psi_i(n;M)$ and $\vpo=\vpo(n)$,
\[
\Mm^{\rm all}(\tfrac{n}{2};\tfrac{M}{2})=\frac{4}{\pi}\Big[\Psi_1-(\log2)\Psi_0+\big(c_1-\vpo-(1-\tfrac{3}{2} 2^{-v})\log2\big)\Psi_0\Big]=\Mm^{(0)}(n;M).
\]
Moreover
\[
2\Mm_\Gamma(n;M)-2\Mm^{(0)}(n;M)=\frac{4}{\pi}\big[\Psi_1+(c_1-\vpo)\Psi_0\big]-\frac{\log2}{\pi}\Big[12\Big(\frac{4}{3}-2^{-v}\Big)-8\Big(2-\frac{3}{2} 2^{-v}\Big)\Big]\Psi_0 .
\]
The last bracket vanishes, so this equals $\Mm^{(1)}(n;M)$.

\emph{The formula for $\Rr(n^2)$.} Writing $\Mm^{(0)}+\Mm^{(1)}$ in terms of $\varpi(n)$, using $\beta_0+\beta_1=2-(3/2) 2^{-v}$ for all $v\ge0$, gives the second formula of Theorem~\ref{thm:sq}.
\end{proof}

\begin{proof}[Proof of Theorem~\ref{thm:sq}]
Write $\Ee^{(e)}(n;M)=\Rr_e(n^2;M)-\Mm^{(e)}(n;M)$, $\Ee^{\rm all}=\Ee^{(0)}+\Ee^{(1)}$ and $\Ee_G=N_G(n^2)-\Mm_G(n)$.

\emph{Error identities.} By Proposition~\ref{prop:par} and Lemma~\ref{lem:mainid}:
\begin{itemize}
\item for odd $n$: $\Ee^{(0)}=2\Ee_\vartheta$, $\Ee^{(1)}=\Ee_\Gamma-\Ee_\vartheta$, and $\Ee^{\rm all}=\Ee_\Gamma+\Ee_\vartheta$;
\item for even $n$:
\begin{gather*}
\Ee^{(0)}(n;M)=\Ee^{\rm all}(\tfrac{n}{2};\tfrac{M}{2}),\qquad \Ee^{(1)}(n;M)=2\Ee_\Gamma(n;M)-2\Ee^{\rm all}(\tfrac{n}{2};\tfrac{M}{2}),\\
\Ee^{\rm all}(n;M)=2\Ee_\Gamma(n;M)-\Ee^{\rm all}(\tfrac{n}{2};\tfrac{M}{2}).
\end{gather*}
\end{itemize}

\emph{Small $M$.} For $0<M'<2$ and any $n'\ge1$, only $|x_3|\le3$ contribute. By the divisor bound $\Rr_e(n'^2;M')\ll n'^{\varepsilon}$, and clearly $\Mm^{(e)}(n';M')\ll\log(n'+2)$. Hence $\Ee^{(e)}(n';M')\ll n'^{\varepsilon}$.

\emph{Induction.} Let $C$ be such that $|\Ee_G(n;M)|\le Cn^{1+\theta+\varepsilon}M^{-1/2}$ in the range of Theorem~\ref{thm:sqasymp}. We show by induction on $v_2(n)$ that $|\Ee^{\rm all}(n;M)|\le K_1n^{1+\theta+\varepsilon}M^{-1/2}$ for $1\le M\le n$, with $K_1=\max(4C/(1-2^{-1/2}),K_0)$, where $K_0$ covers the small-$M$ case (note $n^{\varepsilon}\le n^{1+\theta+\varepsilon}M^{-1/2}$).
\begin{itemize}
\item For odd $n$, $|\Ee^{\rm all}|\le2Cn^{1+\theta+\varepsilon}M^{-1/2}$.
\item For even $n$ and $M<2$, use the small-$M$ bound.
\item For even $n$ and $M\ge2$, we have $1\le M/2\le n/2$, and
\[
|\Ee^{\rm all}(n;M)|\le2Cn^{1+\theta+\varepsilon}M^{-1/2}+K_1(n/2)^{1+\theta+\varepsilon}(M/2)^{-1/2}\le K_1n^{1+\theta+\varepsilon}M^{-1/2}.
\]
\end{itemize}

\emph{Conclusion.} The bounds for $\Ee^{(0)},\Ee^{(1)}$ now follow from the error identities, using the small-$M$ bound when $M/2<1$.
\end{proof}

\section{Numerical checks and open problems}\label{sec:num}

\subsection{Numerical checks}\label{sec:numchecks}
Every step of the argument was also checked numerically. Tables~\ref{tab:sq} and~\ref{tab:ns} compare exact counts with the main terms. We restrict to $x_3>0$, i.e.\ we take $\Phi(t)=e^{-1/((t-1)(2-t))}$ on $(1,2)$ and $0$ elsewhere; this halves both sides for the even extension of $\Phi$. The counts were computed exactly by sieving $D+x_3^2$.

\emph{Square discriminants} (Table~\ref{tab:sq}). The differences are of size $O(\sqrt M)$, far below the error term of Theorem~\ref{thm:sq}. The main term includes the secondary terms of Theorem~\ref{thm:sq}. The cases $n=2^{20}$ and $n=2^{23}$ test the term $\beta_0(n)=2-(3/2) 2^{-v}$.

\begin{table}[ht]
\centering\small
\begin{tabular}{rrrcrrr}
\toprule
$n$ & $v_2(n)$ & $M$ & $e$ & $\Rr_e(n^2)$ ($x_3>0$) & main term & difference\\
\midrule
1000003 & 0 & 100000 & 0 & 14373.2 & 14339.5 & 33.7\\
1000003 & 0 & 100000 & 1 & 14018.5 & 14029.3 & $-10.8$\\
4849845 & 0 & 484984 & 0 & 68993.6 & 69005.7 & $-12.1$\\
4849845 & 0 & 484984 & 1 & 67507.6 & 67501.2 & 6.4\\
52026975 & 0 & 2000000 & 0 & 327574.3 & 327500.6 & 73.7\\
52026975 & 0 & 2000000 & 1 & 321311.4 & 321296.5 & 14.9\\
2000006 & 1 & 300000 & 0 & 42630.5 & 42589.8 & 40.7\\
2000006 & 1 & 300000 & 1 & 44946.3 & 44916.3 & 30.0\\
4000012 & 2 & 400000 & 0 & 58285.5 & 58288.5 & $-3.0$\\
4000012 & 2 & 400000 & 1 & 62345.6 & 62321.2 & 24.4\\
6000000 & 7 & 1000000 & 0 & 140393.7 & 140390.6 & 3.0\\
6000000 & 7 & 1000000 & 1 & 152709.3 & 152726.2 & $-16.9$\\
1048576 & 20 & 200000 & 0 & 26352.2 & 26333.0 & 19.1\\
1048576 & 20 & 200000 & 1 & 28761.9 & 28814.7 & $-52.8$\\
8388608 & 23 & 800000 & 0 & 119926.6 & 120011.6 & $-85.0$\\
8388608 & 23 & 800000 & 1 & 129954.6 & 129938.2 & 16.4\\
\bottomrule
\end{tabular}
\caption{Square discriminants: $\Rr_e(n^2)$ against the main term of Theorem~\ref{thm:sq}. The odd $n$ are $1000003$ (prime), $4849845=3\cdot5\cdots19$ and $52026975=3^35^27^211^213$.}\label{tab:sq}
\end{table}

\emph{Non-square discriminants} (Table~\ref{tab:ns}). Here $M=10^5$ ($2\cdot10^5$, $3\cdot10^5$ in the last two rows), and $L(1,\chi_{D_0})$ was computed from its rapidly convergent $\mathrm{erfc}/E_1$ series. The rows test:
\begin{itemize}
\item fundamental $D$ in each residue class $1,5\bmod8$ and $0\bmod4$;
\item both values of $\chi_{D_0}(f)$ for $f=2,3$, and $f=35$;
\item the factor $4-2^{1-k}$ of Corollary~\ref{cor:tern} for $k\le3$.
\end{itemize}

\begin{table}[ht]
\centering\small
\begin{tabular}{rrrlrrr}
\toprule
$D$ & $D_0$ & $f$ & count & exact & main term & rel.\ diff.\\
\midrule
100000000005 & $D$ & 1 & $N$ & 901.1 & 894.6 & $+0.7\%$\\
100000000009 & $D$ & 1 & $N$ & 4622.2 & 4619.8 & $+0.05\%$\\
100000000012 & $D$ & 1 & $N$ & 1610.8 & 1606.2 & $+0.3\%$\\
100000000024 & $D$ & 1 & $N$ & 2422.0 & 2435.9 & $-0.6\%$\\
10000000089 & 1111111121 & 3 & $N$ & 4601.5 & 4596.7 & $+0.1\%$\\
10000000053 & 1111111117 & 3 & $N$ & 1513.3 & 1508.5 & $+0.3\%$\\
10000000020 & 2500000005 & 2 & $N$ & 1496.6 & 1490.6 & $+0.4\%$\\
10000000004 & 2500000001 & 2 & $N$ & 2266.0 & 2257.4 & $+0.4\%$\\
99999201225 & 81632001 & 35 & $N$ & 3013.9 & 3037.8 & $-0.8\%$\\
\midrule
100000000005 & $D$ & 1 & $\Rr$, $k=0$ & 1802.2 & 1789.2 & $+0.7\%$\\
10000000020 & 2500000005 & 2 & $\Rr$, $k=1$ & 2242.3 & 2235.9 & $+0.3\%$\\
40000000080 & 2500000005 & 4 & $\Rr$, $k=2$ & 2619.1 & 2608.5 & $+0.4\%$\\
160000000320 & 2500000005 & 8 & $\Rr$, $k=3$ & 5651.0 & 5589.7 & $+1.1\%$\\
900000000045 & 100000000005 & 3 & $\Rr$, $k=0$ & 7170.9 & 7156.7 & $+0.2\%$\\
\bottomrule
\end{tabular}
\caption{Non-square discriminants: exact counts against the main terms of Theorem~\ref{thm:ns} and Corollary~\ref{cor:tern}.}\label{tab:ns}
\end{table}

\emph{Further checks.} Independently of the tables, the following statements were checked numerically; in each case the two sides agree to the accuracy of the computation (relative error at most $10^{-7}$, in most cases far smaller).
\begin{itemize}
\item \emph{Arithmetic and algebra:} Proposition~\ref{prop:par}, pointwise in $x_3$, for all $D\le400$ and $|x_3|\le60$; Lemmas~\ref{conj}, \ref{hecke}(i) and~\ref{lem:nu} by exhaustive search in boxes; Lemma~\ref{lem:mainid} and Corollary~\ref{cor:tern} symbolically.
\item \emph{Eisenstein series:} Fact~\ref{F:E} against lattice sums; Lemmas~\ref{lem:Etheta} and~\ref{lem:Edata}; Lemma~\ref{eisperiod} at three values of $s'$; the value of $c_1$ in Remark~\ref{rem:c1}.
\item \emph{The kernel:} the construction of Lemma~\ref{kernel} by numerical Abel inversion, the identity $h(t)=\hat h(1/4+t^2)$ for the directly computed Selberg transform, and Lemmas~\ref{Pbound} and~\ref{hbounds}(ii).
\item \emph{The distance formula:} $|Q_q(z)|=\sqrt D\sinh\dist(z,\ell_q)$ (Lemma~\ref{dist}) for several thousand random forms, including $A=0$; its analogue for definite forms (Remark~\ref{rem:definite}); and the counts in Figure~\ref{fig:geod} ($64$ forms, $16$ passages).
\item \emph{The theta group:} the fundamental domain of Subsection~\ref{sec:Gconv}, by checking that a generic point has exactly three $\Gamma$-translates in it, which are pairwise inequivalent under $\Gt$.
\item \emph{Geometric side, square case:} the identity of Proposition~\ref{reduction}, with $F_G$ built from the numerical kernel, against direct counts, for both groups and several $n$.
\item \emph{Geometric side, non-square case:} the unfolding $N(D)=\Tr_D(F)$ in the proof of Theorem~\ref{thm:nsexact}; Theorem~\ref{thm:heckerel} for $F=1$ (720 pairs $(D',p)$), for $F=E(\cdot,2)$ (30 discriminants $f^2D_0$), and for the $\Gamma$-periodised kernel, which is not a Hecke eigenfunction; Fact~\ref{F:gauss}(iv) for 455 fundamental $D_0$; Fact~\ref{F:Eis} at $s=2$ for 18 fundamental $D_0$, odd and even.
\item \emph{Periods of a Maass form:} for the first even Maass cusp form on $\Gamma$ ($t_j=13.7797\ldots$), computed by Hejhal's algorithm, Fact~\ref{F:DIT} for $D_0\in\{5,8,12,13\}$ and the formula for $|\Pp_j|^2$ in Lemma~\ref{mellin}, with the central values computed by an approximate functional equation; and the constant $c=2$ of Fact~\ref{F:L}(c), to the accuracy of a truncated Euler product.
\end{itemize}
The scripts are available as ancillary files with the arXiv version of this paper.

\subsection{Open problems}
\begin{enumerate}[leftmargin=*]
\item \emph{Beyond the exponent $71/96$.} In both cases the exponent comes only from bounding the Hecke factors ($A_j(n)$, resp.\ the factor $A_j(f)$ in Corollary~\ref{cor:factor}) for each $j$ separately by Kim--Sarnak. Any bound for $\sum_{t_j\le T}|A_j(n)|\,|u_j(i)\Pp_j|$ with $T\approx n/M$ that beats Kim--Sarnak on average over $j$ would improve the range.
\item \emph{The full $\Z^3$-count for other non-square $D$.}
\begin{itemize}
\item For non-square $D\equiv1\bmod8$, Proposition~\ref{prop:par}(a) reduces $\Rr_e(D)$ to $N(D)$ and $N_\vartheta(D)$. The latter requires cycle integrals over closed geodesics for the theta group, as in the square case; this is not done here.
\item For $D\equiv2,3\bmod4$ no solution lies in the lattice $x_2\equiv x_3\bmod 2$, and other binary-form models are needed.
\end{itemize}
\item \emph{Other quadrics.} For the sphere $x_1^2+x_2^2+x_3^2=n$ with $x_3$ small, see \cite{HR}. For the two-sheeted hyperboloid $x_1^2-x_2^2-x_3^2=n$ with $x_3$ small, the substitution $x=(A+C,A-C,B)$ gives positive definite forms $[A,B,C]$ of discriminant $-n$, and $|B|/\sqrt n=\sinh\dist(z_q,i\R_{>0})$ for the Heegner point $z_q$ (the root of $q(x,1)$ in $\Hh$). So the count measures Heegner points near the geodesic from $0$ to $\infty$. This is dual to our setting (geodesics near a point), and P.~Humphries has pointed out that the spectral method should apply in the same way. For the one-sheeted hyperboloid with a small coordinate of the same sign as $n$, for example $x_1^2+x_2^2-x_3^2=n$ with $x_1$ small, $|B|/\sqrt n$ is the cosine of the angle at which the closed geodesic $\ell_q$ crosses $i\R_{>0}$. The weight then depends on the direction of $\ell_q$ at the crossing, not only on a point of $\ell_q$, so the method of this paper does not apply directly. None of this is worked out here.
\item \emph{Limits of the method.} Under the generalised Lindelöf hypothesis and the Ramanujan conjecture, the same argument gives the range $M\ge D^{1/3+\varepsilon}$, uniformly in $f$ and including $D=n^2$. This matches the conditional exponent $\delta<1/6$ of \cite{Hu}.
\end{enumerate}

\subsection*{Acknowledgements}
The author thanks the Ramakrishna Mission Vivekananda Educational and Research Institute for an excellent work environment. We thank Peter Humphries for pointing out \cite[Theorem~4.1]{AW} and \cite[Theorem~11.1]{GHLN}, which removed the restriction to odd $D_0$ in the first version of Theorem~\ref{thm:ns}, and for comments on the relation with \cite{HR}.

\end{document}